\documentclass[10pt,  a4paper]{article}

\usepackage{amssymb}
\usepackage{amsthm}
\usepackage{enumerate}
\usepackage{amsmath}
\usepackage[all]{xy}
\usepackage{verbatim}
\usepackage{hyperref}
\usepackage[top=30truemm,bottom=30truemm,left=25truemm,right=25truemm]{geometry}

\numberwithin{equation}{section}

\theoremstyle{definition}
\newtheorem{theorem}{Theorem}[subsection]
\newtheorem*{theorem*}{Theorem}
\newtheorem{definition}[theorem]{Definition}
\newtheorem*{definition*}{Definition}
\newtheorem{prop}[theorem]{Proposition}
\newtheorem*{prop*}{Proposition}
\newtheorem{lemma}[theorem]{Lemma}
\newtheorem*{lemma*}{Lemma}
\newtheorem{rem}[theorem]{Remark}
\newtheorem*{rem*}{Remark}
\newtheorem{cor}[theorem]{Corollary}
\newtheorem*{cor*}{Corollary}
\newtheorem{prop-def}[theorem]{Proposition-Definition}
\newtheorem*{prop-def*}{Proposition-Definition}
\newtheorem{example}[theorem]{Example}

\begin{document}

\title{On comparison between relative log de Rham-Witt cohomology 
and relative log crystalline cohomology}

\author{Kazuki Hirayama\thanks{Sompo Japan Nipponkoa Insurance Inc., 
1-26-1, Nishi-Shinjuku, Shinjuku-ku, Tokyo 160-8338, Japan.} \,\,\, and\,\,\, Atsushi Shiho\thanks{Graduate School of Mathematical Sciences, the Unversity of Tokyo, 3-8-1, Komaba, Meguro-ku, Tokyo 153-8914, Japan.}}

\date{}

\maketitle

\begin{abstract} 
In this article, we prove the comparison theorem 
between the relative log de Rham-Witt cohomology and 
the relative log crystalline cohomology for a log smooth saturated 
morphism of fs log schemes satisfying certain condition. 
Our result covers the case where the base fs log scheme is 
etale locally log smooth over a scheme with trivial log structure or the case 
where the base fs log scheme is hollow, and so it generalizes the previously known results of 
Matsuue. In Appendix, we prove that our relative log de Rham-Witt complex and 
our comparison map are compatible with those of Hyodo-Kato. 
\end{abstract}

\tableofcontents 

\section*{Introduction}

For a scheme $X$ of characteristic $p>0$, its de Rham-Witt complex 
$W_m\Omega_{X}^{\bullet}$ is defined by Illusie \cite{Ill}. 
When $X$ is smooth over a perfect scheme $Y$, 
he proved the theorem that 
the cohomology of this complex relative to $Y$ is isomorphic to 
the relative crystalline cohomology of $X/W_m(Y)$, which we call the comparison theorem. 
In \cite[I 2]{Ill}, he calculated the concrete `basis' of $W_m \Omega_X^{\bullet}$ when 
$X$ is the spectrum $\mbox{Spec}(\mathbb{F}_p[T_1, \ldots, T_r])$
of the polynomial ring, and it is the key ingredient for the proof of 
comparison theorem. 

Langer and Zink \cite{LZ} generalized the construction of Illusie and 
defined the notion of relative de Rham-Witt 
complex $W_m\Omega_{X/Y}^{\bullet}$ when $Y$ is a scheme 
on which $p$ is nilpotent and $X$ is a $Y$-scheme. 
Moreover, they calculated the concrete `basis' (called the 
basic Witt differentials) of $W_m \Omega_{X/Y}^{\bullet}$
when $Y={\rm Spec}\,R$ is affine and $X$ is the spectrum 
${\rm  Spec}\,R[T_1, \ldots, T_r]$ of the polynomial ring, and using them, 
they proved the comparison theorem between the cohomology 
of the complex $W_m \Omega_{X/Y}^{\bullet}$ relative to $Y$ and 
the relative crystalline cohomology of $X/W_m(Y)$ when 
$Y={\rm Spec}\,R$ as above and $X$ is smooth over $Y$. 

Matsuue \cite{Ma} generalized the construction of Langer-Zink 
and defined the notion of 
relative log de Rham-Witt complex $W_m\Lambda_{(X, \mathcal{M})/(Y, \mathcal{N})}^{\bullet}$ 
for a morphism $(X,{\mathcal M}) \to (Y, {\mathcal N})$ of fine log schemes on 
which $p$ is nilpotent. Also, 
$(1)$ when $(Y,  \mathcal{N}) = {\rm Spec}(R, *)$ (where $*$ denotes the trivial 
log structure) and $(X, \mathcal{M})$ is a pair of a smooth $Y$-scheme and 
the log structure associated to a relative normal crossing divisor, or 
$(2)$ when $(Y,  \mathcal{N}) = {\rm Spec}(R, {\mathbb{N}})$ is the log scheme 
associated to the monoid homomorphism ${\mathbb{N}} \to R; 1 \mapsto 0$ and 
$(X, \mathcal{M}) \to (Y,  \mathcal{N})$ is a semistable log scheme
(logarithmic semistable reduction in the terminology of \cite[p.~346]{KF2}), 
he proved the comparison theorem between the cohomology of the complex 
$W_m \Lambda_{(X,  \mathcal{M})/(Y,  \mathcal{N})}^{\bullet}$ relative to $Y$ and 
the relative crystalline cohomology of 
$(X,  \mathcal{M})/W_m(Y,  \mathcal{N})$, again via the calculation of 
basic Witt differentials. 

In this article, we prove the comparison theorem 
between the cohomology of the relative log de Rham-Witt complex 
$W_m \Lambda_{(X,  \mathcal{M})/(Y,  \mathcal{N})}^{\bullet}$ relative to $Y$ and 
the relative crystalline cohomology of 
$(X,  \mathcal{M})/W_m(Y,  \mathcal{N})$ when 
$f:(X, \mathcal{M}) \to (Y, \mathcal{N})$ is a log smooth saturated morphism 
of fs log schemes on which $p$ is nilpotent 
and the log scheme $(Y, \mathcal{N})$ satisfies the following condition: 
\bigskip 

\noindent 
$(\spadesuit)$ \, Etale locally around any geometric point $y$ of $Y$, 
there exist a chart $\varphi: Q_Y \to {\mathcal N}$ inducing the isomorphism 
$Q \xrightarrow{\cong} {\mathcal N}_y/{\mathcal O}_{Y,y}^*$, 
a radical ideal $J$ of $Q$ such that the composite 
$J_Y \subset Q_Y \xrightarrow{\varphi} {\mathcal N} \to {\mathcal O}_Y$ is zero  
and a ring homomorphism 
$\psi: R \to \Gamma(Y,{\mathcal O}_Y)$ such that 
the morphism $(Y, \mathcal{N}) \to {\rm Spec}(R[Q]/JR[Q],Q)$ induced by 
$\varphi, \psi$ is strict smooth. 
\bigskip 

%either (1)' $(Y, \mathcal{N})$ is an fs log scheme which is log smooth over 
%${\rm Spec}R$ or (2)' $(Y, \mathcal{N})$ is hollow. 
As we will explain in Remark \ref{rem:remrem}, the condition $(\spadesuit)$ is satisfied 
in the cases (1), (2) of Matsuue, 
% The case (1)' (resp. (2)') of our result covers the case (1) (resp. (2)) of Matsuue's result 
and so our result generalizes that of Matsuue. 
Also, our result would contain many interesting cases which are not covered by 
his result. 

The key ingredient of the proof is the calculation of 
relative log de Rham-Witt complex 
% $W_m \Lambda_{(X,  \mathcal{M})/(Y,  \mathcal{N})}^{\bullet}$ 
% when $(X,  \mathcal{M}) \to (Y,  \mathcal{N})$ is a morphism associated to 
associated to certain homomorphisms of pre-log rings. More concretely, we will calculate 
the following three types of relative log de Rham-Witt complexes, where $R$ is 
a ${\mathbb Z}_{(p)}$-algebra: 
\begin{enumerate}[(a)] 
\item The complex $W_m\Lambda_{(R[P],  P)/(R, *)}^{\bullet}$ 
associated to the pre-log ring $(R[P],  P)$ over $(R, *)$, 
where $P$ is an fs monoid with $P^{\rm gp}$ torsion free. 
\item The complex $W_m\Lambda_{(R[P],  P)/(R[Q], Q)}^{\bullet}$ 
associated to the pre-log ring $(R[P],  P)$ over $(R[Q], Q)$ 
which is induced by an injective $p$-saturated morphism 
$Q \to P$ of fs monoids with $Q^{\rm gp}, P^{\rm gp}, P^{\rm gp}/Q^{\rm gp}$ 
torsion free. 
\item The complex 
$W_m\Lambda_{(R[P]/JR[P], P)/(R[Q]/JR[Q], Q)}^{\bullet}$
associated to the pre-log ring $(R[P]/JR[P], P)$ over $(R[Q]/JR[Q], Q)$ 
which we obtain by dividing the rings 
$R[P], R[Q]$ in (b) by the ideals $JR[P], JR[Q]$ 
generated by a radical ideal $J$ of the monoid $Q$. 
%with $Q^{*} = \{0\}$ by the morphism $(R[Q], Q) \to (R,Q)$ 
%associated to the monoid homomorphism $Q \to R; q \mapsto 0 \, (\forall q \not= 0)$.  
%% which sends every nonzero element to $0$. 
\end{enumerate}

The calculation in the case (a) is the basis of all the calculations: 
Contrary to the cases treated in \cite{LZ} and \cite{Ma}, it is hard to 
construct the `basis' of $W_m\Lambda_{(R[P],  P)/(R, *)}^{\bullet}$ 
directly. Nonetheless, by comparing the complex 
$W_m\Lambda_{(R[P],  P)/(R, *)}^{\bullet}$ with 
the complex $W_m\Lambda_{(R[P^{\rm gp}],  P^{\rm gp})/(R, *)}^{\bullet}$
which is accessible by the method in \cite{LZ}, 
we can form a decomposition of the complex $W_m\Lambda_{(R[P],  P)/(R, *)}^{\bullet}$ 
indexed by the set $P[\frac{1}{p}]$. Using this decomposition, we can prove the 
comparison isomorphism when
($f:(X, \mathcal{M}) \to (Y, \mathcal{N})$ is as above and) 
the log structure ${\mathcal N}$ of 
$(Y,{\mathcal N})$ is trivial, namely, when we can take the monoid $Q$ 
in the condition $(\spadesuit)$ to be the trivial one and the ideal $J$ to be the empty set. 

In the case (b), we prove a similar decomposition by using the decomposition 
in the case (a) and the fundamental exact sequence for 
relative log de Rham-Witt complexes proven in \cite{Ma}. Using this decomposition, 
we can prove the comparison isomorphism 
when
($f:(X, \mathcal{M}) \to (Y, \mathcal{N})$ is as above and)  
$(Y, \mathcal{N})$ is etale locally log smooth over a scheme with trivial log structure, namely,  
when we can take the radical ideal $J$ in the condition $(\spadesuit)$ to be the empty set. 

In the case (c), we prove a similar decomposition (with a smaller index set) 
by using the decomposition in the case (b) and noting the fact (proven in \cite{Ma}) that the complex in (c) is a certain 
quotient of the complex in (b).   Using this decomposition, 
we can prove the comparison isomorphism in general case. 

The content of each section is as follows. In Section 1, 
we give a review of relative de Rham-Witt complexes and basic Witt differentials 
defined in \cite{LZ} and 
%notions on log structures defined in 
%\cite{Kato}, \cite{KF}, \cite{O}, \cite{Ma} and 
relative log de Rham-Witt complexes defined in \cite{Ma}. 
We also define the comparison morphism from 
the relative log crystalline cohomology to 
the cohomology of relative log de Rham-Witt complex 
in a slightly more general situation than that treated in \cite{Ma}. 
In Section 2, we calculate the relative 
log de Rham-Witt complex $W_m\Lambda_{(R[P], P)/(R, {*})}^{\bullet}$ 
in the case (a) above and prove our main comparison theorem  
when the log structure on the base log scheme $(Y,{\mathcal N})$ is trivial. 
%when $R$ is an $\mathbb{Z}_{(p)}$-algebra in which $p$ is nilpotent 
%and $P$ is an fs monoid such that $P^{\rm gp}$ is torsion free. 
%Using this calculation, we prove the comparison theorem in the case (1)' 
%when the log structure on the base scheme $(Y,{\mathcal N})$ is trivial. 
In Section 3, we calculate the relative 
log de Rham-Witt complex $W_m\Lambda_{(R[P], P)/(R[Q], Q)}^{\bullet}$
in the case (b) and prove our main comparison theorem 
when the base log scheme $(Y, \mathcal{N})$ is etale locally log smooth over a scheme with trivial log structure. 
%%%%%we can take the radical ideal $J$ in the condition $(\spadesuit)$ to be the trivial one.   
%when the log structure on the base scheme $(Y,{\mathcal N})$ is trivial. 
%when $R$ is as above, $(R[P],  P)/(R[Q],  Q)$ is 
%log smooth, $Q^{\rm gp} \to P^{\rm gp}$ is injective and 
%$Q^{\rm gp} ,P^{\rm gp}, P^{\rm gp}/Q^{\rm gp}$ are torsion-free, by using 
%the calculation in the previous section. Using this calculation, 
%we prove the comparison theorem in the case (1)'. 
In Section 4, we calculate  the relative 
log de Rham-Witt complex $W_m\Lambda_{(R[P]/JR[P], P)/(R[Q]/JR[Q], Q)}^{\bullet}$
in the case (c) and prove our main comparison theorem in general case. 

%when $R$ is as above, $(R[P],  P)/(R[Q],  Q)$ is 
%log smooth, $Q^{\rm gp} \to P^{\rm gp}$ is injective, $Q^* = P^* = \{0\}$ and 
%$Q^{\rm gp} ,P^{\rm gp}, P^{\rm gp}/Q^{\rm gp}$ are torsion-free, by using 
%the calculation in the previous section. Using this calculation, 
%we prove the comparison theorem in the case (2)'. 
%between the cohomology of 
%$W_m \Lambda_{(X,  \mathcal{M})/(S,  \mathcal{N})}^{\bullet}$ relative to $S$ and 
%the relative crystalline cohomology of 
%$(X,  \mathcal{M})/W_m(S,  \mathcal{N})$

When $k$ is a perfect field of characteristic $p>0$, $Y = {\rm Spec}\,k$ and 
$f: (X,{\mathcal M}) \to (Y,{\mathcal N})$ is a log smooth $p$-saturated morphism of 
fine log schemes, Hyodo-Kato defined in \cite{HK} the log de Rham-Witt complex 
$W_m\omega^{\bullet}_{(X,{\mathcal M})/(Y,{\mathcal N})}$ whose definition is 
a priori different from \cite{Ma} and the comparison map 
from the complex computing the crystalline cohomology to their log de Rham-Witt complex 
$W_m\omega^{\bullet}_{(X,{\mathcal M})/(Y,{\mathcal N})}$. 
In Appendix, we prove that, in the above situation, their log de Rham-Witt complex 
$W_m\omega^{\bullet}_{(X,{\mathcal M})/(Y,{\mathcal N})}$ is isomorphic to 
our log de Rham-Witt complex $W_m\Lambda^{\bullet}_{(X,{\mathcal M})/(Y,{\mathcal N})}$
 and their comparison map is compatible with ours. This compatibility is proven in 
\cite{GL} when $f$ is a semistable log scheme, and our result generalizes it. 

Throughout this article, $\mathbb{N}$ denotes the set of 
integers equal or greater than $0$. A monoid is always commutative and 
we write the operation of a monoid by addition, unless it is the multiplication of 
a commutative ring. 
% otherwise stated. 
Also, the symbol $*$ denotes the trivial monoid. 
We will freely use the notions concerning log structures defined in 
\cite{Kato}, \cite{KF2}, \cite{KF} and \cite{O}, although some of them will be reviewed
in the text. 

\section{Preliminaries}

\subsection{Relative de Rham-Witt complex and basic Witt differential}
In this sebsection, we give a review of 
relative de Rham-Witt complexes and basic Witt differentials which are defined in 
\cite{LZ}. 

\begin{definition}\label{def:1.1.1}
Let $R$ be a commutative ring and $S$ a commutative $R$-algebra. 

\begin{enumerate}[(1)]

\item  \cite[Def.~1.1]{LZ} 
Assume moreover that $S$ is equipped with a divided power structure 
$(I, \{ \gamma_n \})$. For an $S$-module $M$, an $R$-linear derivation 
$D: S \to M$ is called a pd-derivation if we have the equality 
$D(\gamma_{n}(b)) = \gamma_{n-1} (b) D(b) $ for any $n \geq 1$ and $b \in I$. 
We have the universal pd-derivation, which we denote by 
$d: S \to \breve{\Omega}_{S/R}^{1}$. 

\item A differential graded $S/R$-algebra $(E^{\bullet}, D)$ is a pair of 
an $\mathbb{N}$-graded $S$-algebra $E^{\bullet}$ and an $R$-linear map 
$D: E^{\bullet} \to E^{\bullet}$ homogeneous of degree $1$ which satisfies 
the equalities 
%\begin{align*}
%\omega \eta & =  (-1)^{ij} \eta \omega  \quad (\omega \in E^{i}, \eta \in E^{j}), \\
%d(\omega \eta) & =  (d\omega)\eta + (-1)^i \omega d\eta  \quad (\omega \in E^{i}, \eta \in E^{j}), \\
%d^2 & = 0. 
%\end{align*}
\begin{eqnarray*}
\omega \eta & = & (-1)^{ij} \eta \omega  \quad (\omega \in E^{i}, \eta \in E^{j}), \\
d(\omega \eta) & = & (d\omega)\eta + (-1)^i \omega d\eta  \quad (\omega \in E^{i}, \eta \in E^{j}), \\
d^2 & = &0. 
\end{eqnarray*}
In the situation of (1), if we put 
$\breve{\Omega}_{S/R}^{i} := \bigwedge_S^{i} \breve{\Omega}_{S/R}^{1}$, 
the pd-derivation $d$ induces a structure of 
differential graded $S/R$-algebra on 
$\breve{\Omega}_{S/R}^{\bullet}$.

\item \cite[Def.~1.4]{LZ} An $F$-$V$-procomplex over an $R$-algebra $S$ is 
a projective system of differential graded $W_m(S)/W_m(R)$-algebras 
$\{E_m^{\bullet}:= (E_m^{\bullet}, D_m), \pi_m : E_{m+1}^{\bullet} \to E_m^{\bullet} \}_{m \in \mathbb{N}}$ with $E_0^{\bullet}=0$ equipped with maps of graded Abelian groups 
\[ F : E_{m+1}^{\bullet} \to E_m^{\bullet} \quad V: E_m^{\bullet} \to E_{m+1}^{\bullet} \]
satisfying the following conditions: 

\begin{enumerate}[(i)]
\item For any $m \geq 0$, the canonical map $W_m(S) \to E_m^{0}$ 
is compatible with $F,  V$. 

\item If we denote by $E_{m,  [F]}^{\bullet}$ the algebra 
$E_{m}^{\bullet}$ regarded as a $W_{m+1}(S)$-algebra via the map 
$F:W_{m+1}(S) \to W_{m}(S)$, $F$ induces a morphism of graded 
$W_{m+1}(S)$-algebras $E_{m+1}^{\bullet} \to E_{m,  [F]}^{\bullet}$. 

\item The following equalities hold. 
\begin{eqnarray*}
  {}^{FV}\omega & = & p\omega \quad (\omega \in  E_m^{\bullet}),   \\
   {}^{F}D_{m+1}^{V}\omega & = & D_m \omega \quad (\omega \in  E_m^{\bullet}), \\
  {}^{F}D_{m+1}[x] & = & [x^{p-1}]D_m[x]  \quad ( x \in S ), \\
  {}^{V}(\omega {}^{F}\eta) & = & {}^{V}(\omega) \eta \quad (\eta \in E_{m+1}^{\bullet}).   
\end{eqnarray*} 
A morphism of $F$-$V$-procomplexes over an $R$-algebra $S$ is a map of 
projective systems which is 
compatible with $D_m (m \in \mathbb{N}), F , V$. 

\end{enumerate}

\end{enumerate}

\end{definition}

The definition of relative de Rham-Witt complex $\{W_m \Omega_{S/R}^{\bullet} \}_m$
is given as follows. 

\begin{prop-def}
\cite[Prop.~1.6]{LZ}
The category of $F$-$V$-procomplexes over an $R$-algebra $S$ has the initial 
object. We denote it by 
$\{(W_m \Omega_{S/R}^{\bullet}, d_m)\}_m$ and call it the 
relative de Rham-Witt complex of the $R$-algebra $S$. 
We will denote $d_m$ simply by $d$ in the following. 
Also, we put 
$W\Omega_{S/R}^{\bullet} := \displaystyle \lim_{\begin{subarray}{c} \longleftarrow\\ m\end{subarray}} W_m \Omega_{S/R}^{\bullet}$. 
\end{prop-def}

By definition, for any $F$-$V$-procomplex $\{ (E_m^{\bullet}, D_m) \}_m$ 
over an $R$-algebra $S$, there exists uniquely a morphism of 
$F$-$V$-procomplexes 
\[ \{(W_m \Omega_{S/R}^{\bullet} , d)\}_m \to \{(E_m^{\bullet}, D_m) \}_m \]
over the $R$-algebra $S$. 

Concretely, $W_m \Omega_{S/R}^{\bullet}$ is defined as a 
certain quotient of the differential graded 
$W_m(S)/W_m(R)$-algebra 
$\breve{\Omega}_{W_m(S)/W_m(R)}^{\bullet}$. 

\quad

Next we recall the definition of basic Witt differentials, which are used 
to express the elements of $W_m \Omega_{S/R}^{\bullet}$ when 
$S = R[\mathbb{T}] := R[T_1, \ldots , T_r]$ is the polynomial ring over 
a ${\mathbb Z}_{(p)}$-algebra $R$. 
Let $X_i := [T_i] \in W(R[\mathbb{T}]) $ be the Teichm${\rm \ddot{u}}$ller
lift of $T_i$. We call a map of sets 
$k: [1,  r] \to \mathbb{N}[\frac{1}{p}] $ as weight, and we will denote 
its value $k(i)$ at $i$ simply by $k_i$ in the following. 
We put ${\rm supp}\,k := \{ i \in [1, r ] ; k_i \neq 0 \} $. 
For each weight $k$, we put a total order ${\rm supp}\,k= \{ i_1, i_2, \ldots, i_s \}$ on 
the set ${\rm supp}\,k$ in such a way that the inequality 
$\mbox{ord}_p k_{i_{j}} \leq \mbox{ord}_p k_{i_{j+1}} $ holds for any $j$. 
Also, we assume that the total order on ${\rm supp}\,k$ is the same as that 
on ${\rm supp}\,p^{a}k~(a \in \mathbb{Z})$. 

We call a subset $I$ of ${\rm supp}\,k$ an interval if any element 
$a \in {\rm supp}\,k$ which is bigger than some element $b$ in $I$ and 
smaller than some element $c$ in $I$ with respect to the total order 
on ${\rm supp}\,k$ belongs to $I$.  

A tuple $\mathcal{P}:= (I_0, I_1, \ldots, I_l)$ of intervals of 
${\rm supp}\,k$ is called a partition of $k$ if ${\rm supp}\,k = I_0 \sqcup \cdots \sqcup I_l$, $I_1 , \ldots, I_l \neq \emptyset$ and 
if any element in $I_j$ is smaller than any element in $I_{j+1}$ with respect to the total order on 
${\rm supp}\,k$ for $j = 0, \dots , l-1$. 
 
For a weight $k$ and a nonempty set $I \subset [1,r]$, let $t(k_{I}) \in \mathbb{Z}$ be the unique integer 
such that the elements $p^{t(k_I)} k_i \, (i \in I)$ are all integers and at least one of them is prime to $p$. 
Also, we put $u(k_{I}) := \max \{ t(k_I),  0 \}$. When there is no risk of confusion, we will denote 
$t(k_{I}), u(k_I)$ simply by $t(I), u(I)$, respectively. Also, we put $t(\emptyset) = u(\emptyset) := 0$. 

For a triple $(\xi, k, \mathcal{P})$ consisting of a weight $k$, its partition $\mathcal{P} = (I_0, I_1, \ldots, I_l)$ and $\xi =  {}^{V^{u(I)}} \eta \in {}^{V^{u(I)}}W(R)$ (where $I = {\rm supp}\,k$), 
we define the element $e(\xi , k, \mathcal{P})$ in $W\Omega_{R[\mathbb{T}]/R}^{l} $ 
in the following way: We put $X^{p^{t(I_i)}k_{I_i}} := \prod_{j \in I_i} X_j^{p^{t(I_i)}k_{j}}$ for each $0 \leq i \leq l$. Then, 
\leftline{when $I_0 \neq \emptyset$, }
\[e(\xi , k, \mathcal{P}) := {}^{V^{u(I_0)}}(\eta X^{p^{u(I_0)}k_{I_0}}) d^{V^{t(I_1)}}(X^{p^{t(I_1)}k_{I_1}})\cdots d^{V^{t(I_l)}}(X^{p^{t(I_l)}k_{I_l}}). 
 \]
When $I_0 = \emptyset, t(I_1) \geq 1$, 
\[
e(\xi , k, \mathcal{P}) := d^{V^{t(I_1)}}(\eta X^{p^{t(I_1)}k_{I_1}}) d^{V^{t(I_2)}}(X^{p^{t(I_2)}k_{I_2}}) \cdots d^{V^{t(I_l)}}(X^{p^{t(I_l)}k_{I_l}}). 
 \]
When $I_0 = \emptyset, t(I_1) \leq 0$, 
\[
e(\xi , k, \mathcal{P}) := \eta d^{V^{t(I_1)}}(X^{p^{t(I_1)}k_{I_1}})\cdots d^{V^{t(I_l)}}(X^{p^{t(I_l)}k_{I_l}}). \]
Here, when $t(I_i) <0$, we put 
\[ d^{V^{t(I_i)}}( X^{p^{t(I_i)}k_{I_i}}) := {}^{F^{-t(I_i)}}d X^{p^{t(I_i)}k_{I_i}} \]
by abuse of notation. We will denote the image of the element $e(\xi , k, \mathcal{P})$ 
in $W_m \Omega_{R[\mathbb{T}]/R}^{\bullet}$ by the same symbol. 
We call an element in $W \Omega_{R[\mathbb{T}]/R}^{\bullet}$ or 
$W_m \Omega_{R[\mathbb{T}]/R}^{\bullet}$
of the form $e(\xi , k , \mathcal{P})$ a basic Witt differential. 

A weight $k$ is called integral if $k_i \in {\mathbb Z}$ for any $i$. 

\begin{prop}\label{2.17}
\cite[Prop.~2.17]{LZ}
Any element $\omega$ in $W_m\Omega_{R[\mathbb{T}]/R}^{\bullet}$
is written uniquely as a sum of basic Witt differentials 
\[ \omega = \sum_{k, \mathcal{P}} e(\xi_{k, \mathcal{P}} , k , \mathcal{P}) \]
such that, in the sum, $k$ runs through weights with $p^{m-1} k$ integral
and ${\mathcal P}$ runs through partitions of $k$. 
\end{prop}

We will not recall the proof of Proposition \ref{2.17}, but we recall here the 
definition of the map 
giving the phantom component of relative de Rham-Witt complex, which is the key 
ingredient of the proof of Proposition \ref{2.17} (\cite[Sec.~2.4]{LZ}). 
% , because we will need it later. 
Let $R$ be a ${\mathbb Z}_{(p)}$-algebra, 
let $S$ be an $R$-algebra and let ${\bf{w}}_m: W(S) \to S$ be 
the $m$-th Witt polynomial. For an $S$-module $M$, we will denote 
by $M_{{\bf{w}}_m}$ the module $M$ regarded as a $W(S)$-module via the map 
${\bf{w}}_m$. We put 
\[ P_m^{\bullet} := \bigoplus_{i=0}^{m-1} \Omega_{S/R, {\bf w}_i}^{\bullet} \]
and regard $\{P_m^{\bullet} \}_m$ as a projective system by natural projections. 
For $m \geq 0$, define the map $\delta_m : W(S) \to \Omega_{S/R, {{\bf w}_m}}^{1}$ by  
\[ (x_i)_{i \in \mathbb{N}} \mapsto \sum_{i=0}^{m} x_i^{p^{m-i}-1}dx_i. \]
This is a $W(R)$-linear pd-derivation and so it induces the map 
\[ \omega_m: \breve{\Omega}_{W_{m+1}(S)/W_{m+1}(R)}^{1} \to \Omega_{S/R, {{\bf w}_m}}^{1}. \]
By taking exterior products, this induces the maps 
\[ \omega_m: \breve{\Omega}_{W_{m+1}(S)/W_{m+1}(R)}^{i} \to \Omega_{S/R, {{\bf w}_m}}^{i} \quad (i \in \mathbb{N}), \]
hence the map of projective systems of graded algebras 
\[ \{(\omega_0, \ldots , \omega_{m-1})\}_m: 
\{\breve{\Omega}_{W_{m+1}(S)/W_{m+1}(R)}^{\bullet}\}_m \to 
\{P_m^{\bullet}\}_m. \] 
Then it is shown in \cite[Prop.~2.15]{LZ} that it induces the map 
 of projective system of graded algebras 
\[ \{\underline{\omega}^{m}\}_m : \{W_m\Omega_{S/R}^{\bullet}\}_m \to \{P_m^{\bullet}\}_m. \]
(However, note that the maps $\underline{\omega}^{m}$ are not compatible with 
$d$.) As a part of this map, we obtain the map of graded algebras 
\begin{equation}\label{phantom}
\omega_{m} : W_{m+1}\Omega_{S/R}^{\bullet} \to \Omega_{S/R, {\bf w}_m}^{\bullet}. 
\end{equation}
We will use the map \eqref{phantom} later in this article. 

\subsection{Relative log de Rham-Witt complex}

In this subsection, we give a review of 
relative log de Rham-Witt complexes which is defined in 
\cite{Ma}. First we fix terminologies on monoids, pre-log rings and log Witt schemes. 

\begin{definition}
Let $P$ be a monoid. 
 \begin{enumerate}[(1)]
  \item $P$ is called integral if an equality 
$x+z=y+z$ in $P$ implies the equality $x=y$. 
  \item $P$ is called fine if $P$ is integral and finitely generated. 
  \item $P$ is called saturated if $P$ is integral and if any element $x$ in $P^{\rm gp}$
such that $nx \in P$ for some $n \in \mathbb{N}_{>0}$ belongs to $P$.  
  \item $P$ is called fs if it is fine and saturated. 
 \end{enumerate} 
\end{definition}

\begin{definition}

\begin{enumerate}[(1)]

\item A pre-log ring is a triple $(R,  P,  \alpha)$ consisting of a commutative ring 
$R$, a fine monoid $P$ and a monoid homomorphism 
$\alpha:P \to R$, where $R$ is regarded as a monoid with respect to the multiplication on it. We will often denote a pre-log ring $(R,  P,  \alpha)$ simply by $(R,P)$. 
The notion of a morphism of pre-log rings is defined in natural way. 

\item For a commutative ring $R$ and a fine monoid $P$, we denote by 
$(R[P],P)$ the pre-log ring which consists of the monoid algebra $R[P]$ over $R$ 
associated to $P$, the monoid $P$ and the canonical map $P \to R[P]$. 

\item For a pre-log ring $(A, P, \alpha)$, we denote by 
$\mbox{Spec}(A, P)$ the log scheme whose underlying scheme is 
$X = \mbox{Spec}A$ and whose log structure is the one associated to 
the pre-log structure $P_X \to \mathcal{O}_X$ (where $P_X$ is the 
constant sheaf on $X_{{\rm et}}$ associated to $P$) which is induced by 
the monoid homomorphism $\alpha: P \to A$. 

\item A morphism of pre-log rings $\varphi: 
(A,  P) \to (B,  Q)$ is log smooth~(resp. log etale)
if the kernel and the torsion part of the cokernel (resp. the cokernel) 
of the map $P^{\rm gp} \to Q^{\rm gp}$ have order invertible in $B$ and if 
the map $A \otimes_{\mathbb{Z}[P]} \mathbb{Z}[Q] \to B$ induced by $\varphi$ 
is etale. In this case, $\varphi$ induces a log smooth (resp. log etale) 
morphism $\mbox{Spec}(B, Q) \to \mbox{Spec}(A, P)$ of log schemes by 
\cite[Thm.~3.5]{Kato}. 
\end{enumerate}

\end{definition}

\begin{definition}

\begin{enumerate}[(1)]

\item Let $(X, \mathcal{M}, \alpha)$ be a log scheme on which 
$p$ is nilpotent and let $m \in {\mathbb N}_{>0}$. Then its log Witt scheme 
$(X, \mathcal{M}, \alpha)$ is the log scheme 
$(W_m(X), W_m(\mathcal{M}), W_m(\alpha))$, where 
$W_m(X)$ is the Witt scheme of $X$, $W_m(\mathcal{M})$ is the monoid sheaf 
$\mathcal{M} \oplus \mbox{Ker}(W_m(\mathcal{O}_X)^{*} \to \mathcal{O}_X^{*})$ and  $W_m(\alpha): W_m(\mathcal{M}) \to W_m(\mathcal{O}_{X})$ is the homomorphism of monoid sheaves 
defined by $(m,u) \mapsto [\alpha(m)]u \, \,(m \in {\mathcal M}, u \in 
 \mbox{Ker}(W_m(\mathcal{O}_X)^{*} \to \mathcal{O}_X^{*}))$. 
We will denote the log Witt scheme $(W_m(X), W_m(\mathcal{M}), W_m(\alpha))$
simply by $W_m(X, \mathcal{M})$ in the following. 

\item For a pre-log ring $(R,  P,  \alpha)$ and $m \in \mathbb{N}_{>0}$, 
its Witt pre-log ring is the pre-log ring $(W_m(R), P , \beta)$, where $\beta: P \to W_m(R)$
is the monoid homomorphism defined by 
$\beta(x) = [ \alpha(x)] \, (x \in P)$. We will denote the Witt pre-log ring 
$(W_m(R), P , \beta)$ simply by $W_m(R,P)$ in the following. 
By definition,when $p$ is nilpotent in $R$, the log Witt scheme of 
${\rm Spec}\,(R,P)$ is ${\rm Spec}\,W_m(R,P)$. 

\end{enumerate}

\end{definition}

Now we recall the definition of relative log de Rham-Witt complex 
in a parallel way to Section 1.1.

\begin{definition}\label{def:rellogdrw}
Let $(R, P) \to (S , Q)$ be a morphism of pre-log rings. 

\begin{enumerate}[(1)]

\item \cite[Def.~1.1.9]{O} Let $M$ be an $S$-module. A log derivation on 
$(S,  Q)/(R,  P)$ is a pair $(D: S \to M, \delta : Q \to M)$
consisting of an $R$-linear derivation $D$ and a monoid homomorphism 
$\delta$ which satisfies the following two conditions: 

\begin{enumerate}[(i)]

\item If we denote the structure monoid homomorphism 
$Q \to S$ of $(S,Q)$ by $\beta$, we have the equality 
$D(\beta(x))= \beta(x) \delta(x)$ for any $x \in Q$. 

\item If we denote the monoid part $P \to Q$ of the morphism 
$(R, P) \to (S , Q)$ by 
$\theta$, we have the equality $\delta(\theta(x)) =0$ for any $x \in P$. 

\end{enumerate}

We have the universal log derivation, which we denote by 
$(d: S \to \Lambda_{(S, Q)/(R, P)}^{1}, \delta: Q \to \Lambda_{(S, Q)/(R, P)}^{1})$.

\item  \cite[Def.~3.2]{Ma},  \cite[Def.~1.1.9]{O} 
Assume moreover that $S$ is equipped with a divided power structure $(I, \{ \gamma_n \})$. For an $S$-module $M$, a log derivation $(D: S \to M, \delta : Q \to M)$
is called a log pd-derivation if $D$ is a pd-derivation. 
We have the universal log pd-derivation, which we denote by $(d: S \to \breve{\Lambda}_{(S, Q)/(R, P)}^{1}, \delta: Q \to \breve{\Lambda}_{(S, Q)/(R, P)}^{1} )$.

\item \cite[Def.~3.3]{Ma} A log differential graded $(S,  Q)/(R,  P)$-algebra 
is a triple $(E^{\bullet}, D, \partial)$ consisting of a 
differential graded $S/R$-algebra $(E^{\bullet}, D)$ and 
a map $\partial : Q \to E^{1}$ such that the pair 
$(d, \partial)$ is a log derivation with $d\partial =0$. 
In the situation of (2), if we put 
$\breve{\Lambda}_{S/R}^{i} := \bigwedge_S^{i} \breve{\Lambda}_{S/R}^{1}$, the 
log pd-derivation $(d, \delta)$ induces a structure of log differential graded 
$(S,  Q)/(R,  P)$-algebra on 
$\breve{\Lambda}_{S/R}^{\bullet}$. 

\item \cite[Def.~3.4]{Ma} 
A log $F$-$V$-procomplex over an $(R, P)$-algebra $(S,  Q)$ is a projective system 
of log differential graded $W_m(S,  Q)/W_m(R,  P)$-algebras 
$\{E_m^{\bullet}:= (E_m^{\bullet}, D_m, \partial_m), \pi_m : E_{m+1}^{\bullet} \to E_m^{\bullet} \}_{m \in \mathbb{N}}$ with $E_0^{\bullet}=0$ equipped with maps of 
graded Abelian groups 
\[ F : E_{m+1}^{\bullet} \to E_m^{\bullet}  \quad V: E_m^{\bullet} \to E_{m+1}^{\bullet} \]
satisfying the following conditions: 

\begin{enumerate}[(i)]
\item $\partial_m$ is compatible with $\pi_m$. 

\item For any $m \geq 0$, the canonical map $W_m(S) \to E_m^{0}$ is 
compatible with $F,  V$. 

\item For any $m \geq 1$, the map $\pi_m$ commutes with $F,  V$. 

\item If we denote by $E_{m,  [F]}^{\bullet}$ the algebra $E_{m}^{\bullet}$
regarded as a $W_{m+1}(S)$-algebra via the map 
$F:W_{m+1}(S) \to W_{m}(S)$, $F$ induces a morphism of graded 
$W_{m+1}(S)$-algebras $E_{m+1}^{\bullet} \to E_{m,  [F]}^{\bullet}$. 

\item The following equalities hold. 
\begin{eqnarray*}
     {}^{FV}\omega  & = & p\omega \quad (\omega \in E_m^{\bullet}),   \\
   {}^{F}D_{m+1}^{V}\omega & = &D_m \omega \quad (\omega \in E_m^{\bullet}), \\
  {}^{F}D_{m+1}[x] & = & [x^{p-1}]D_m[x]  \quad (x \in S), \\
  {}^{V}(\omega {}^{F}\eta) & = & {}^{V}(\omega) \eta \quad ( \eta \in E_{m+1}^{\bullet}), \\  
  {}^{F}(\partial_{m+1} q) & = &  \partial_{m} q \quad( q \in Q).
\end{eqnarray*}

A morphism of log $F$-$V$-procomplexes over an $(R, P)$-algebra is 
a map of projective systems which is compatible with 
$D_m, \partial_m (m \in \mathbb{N}), F, V$. 

\end{enumerate}

\end{enumerate}

\end{definition}

The definition of relative log de Rham-Witt complex 
$\{W_m \Lambda_{S/R}^{\bullet} \}_m$ is given as follows. 

\begin{prop-def}
\label{thm30}
\cite[Prop.~3.5]{Ma}\footnote{$R$ is assumed to be a $\mathbb{Z}_{(p)}$-algebra in \cite[Prop.~3.5]{Ma}, but this assumption is unnecessary (cf. \cite[Prop.~1.6]{LZ}).} 
The category of log $F$-$V$-procomplexes over an
$(R, P)$-algebra $(S,  Q)$ has the initial object. We denote it by 
$\{(W_m \Lambda_{(S,  Q)/(R,  P)}^{\bullet}, d_m, d \log_m)\}_m$ and call it 
the relative log de Rham-Witt complex of the $(R, P)$-algebra $(S,  Q)$. 
We will denote $d_m, d \log_m  (m\in \mathbb{N})$ simply $d, d\log$
respectively in the following. Also, we put 
$W\Lambda_{(S,  Q)/(R,  P)}^{\bullet} := \displaystyle \lim_{\begin{subarray}{c} \longleftarrow\\ m\end{subarray}} W_m \Lambda_{(S,  Q)/(R,  P)}^{\bullet}$. 
\end{prop-def}

By definition, for any log $F$-$V$-procomplex $\{ (E_m^{\bullet}, D_m, \partial_m) \}_m$ 
over an $(R, P)$-algebra $(S,  Q)$, there exists uniquely 
a morphism of log $F$-$V$-procomplexes 
\[ \{(W_m \Lambda_{(S,  Q)/(R,  P)}^{\bullet} ,  d , d\log) \}_m \to \{ (E_m^{\bullet}, D_m, \partial_m)\}_m\]
over the $(R, P)$-algebra $(S,  Q)$. 

Concretely, $W_m \Lambda_{(S,  Q)/(R,  P)}^{\bullet}$
is defined as a quotient of the log differential graded 
$W_m(S,  Q)/W_m(R,  P)$-algebra 
$\breve{\Lambda}_{W_m(S,  Q)/W_m(R,  P)}^{\bullet}$. 

We can sheafify the definition above as follows. 

\begin{prop-def}\label{thm15}
\cite[Prop.-Def.~3.10]{Ma}  
Let $f : (X,  \mathcal{M}) \to (Y, \mathcal{N})$ be a morphism of 
fine log schemes such that $p$ is nilpotent on $Y$. 
(Then we can identify the etale site on $W_m(X)$ with that on $X$.) 
Then there exists uniquely a complex of quasi-coherent sheaves 
$W_m \Lambda_{(X, \mathcal{M})/(Y, \mathcal{N})}^{\bullet}$
which satisfies the following condition: 
For any commutative diagram 
\[ \xymatrix{
 U = \mbox{Spec}~S' \ar[d]_{\gamma'} \ar[r] & V = \mbox{Spec}~R' \ar[d]_{\gamma} \\
 X \ar[r]_{f}  & Y 
}
\] with $\gamma,  \gamma'$ etale and a chart 
$(Q_U \to \mathcal{M}|_U, P_V \to \mathcal{N}|_V , P \to Q)$ of the morphism 
$(U, \mathcal{M}|_U) \to (V ,  \mathcal{N}|_V)$, there exists the canonical 
equality 
\[ \Gamma(U, W_m \Lambda_{(X, \mathcal{M})/(Y, \mathcal{N})}^{\bullet} ) = W_m \Lambda_{(S' ,  Q)/(R' , P)}^{\bullet}. \]
\end{prop-def}

\subsection{Definition of comparison morphism}

In this subsection, we give a definition of the comparison morphism from 
the complex computing relative log crystalline cohomology to 
the relative log de Rham-Witt complex, following the construction in 
\cite[Sec.~6.2]{Ma}. We work in a slightly more general setting than that in 
\cite{Ma}, where the base log scheme was assumed to be affine. 

First we recall the definition of a log Frobenius lift of a log smooth morphism of 
pre-log rings. 

\begin{definition}
\cite[Def.~5.2]{Ma} \, 
Let $R$ be a commutative ring on which $p$ is nilpotent and let 
$(R,  Q) \to (S,  P)$ be a log smooth morphism of pre-log rings. 
A log Frobenius lift of $(S, P)$ over $(R,Q)$ is a triple 
\begin{equation}\label{eq:fl1}
( \{(S_n, P_n)\}_{n \geq 1}, 
\{\phi_n: (S_n, P_n) \to (S_{n-1},  P_{n-1})\}_{n \geq 2}, 
\{\delta_n: (S_n, P_n) \to W_n(S,  P)\}_{n \geq 1}) 
\end{equation}
satisfying the following conditions: 
\begin{enumerate}
\item[(1)] $\{(S_n, P_n)\}_{n \geq 1}$ is a projective system of 
pre-log rings over $\{W_n(R,  Q)\}_{n \geq 1}$ with 
$(S_1, P_1) = (S,  P)$ such that each $(S_n, P_n)$ is log smooth over 
$W_n(R,  Q)$ and that the transition morphisms induce the isomorphisms 
$$ W_n(R,  Q) \otimes_{W_{n+1}(R,  Q)} (S_{n+1}, P_{n+1}) \xrightarrow{\cong} 
(S_n, P_n) \qquad (n \geq 1). $$
\item[(2)] 
$\{\phi_n: (S_n, P_n) \to (S_{n-1},  P_{n-1})\}_{n \geq 2}$
is a morphism of projective systems which is compatible with 
the morphism $\{F: W_n(R,  Q) \to W_{n-1}(R,  Q) \}_{n \geq 2}$,  
the absolute Frobenius 
$S/pS \to S/pS$ and the map 
$P \to P$ of multiplication by $p$. 
\item[(3)] $\{\delta_n: (S_n, P_n) \to W_n(S,  P)\}_{n \geq 1}$
is a morphism of projective systems over 
$\{W_n(R,  Q)\}_{n \geq 1}$ such that the following diagram 
is commutative: 
\[ \xymatrix{
(S_{n+1}, P_{n+1}) \ar[r]_{\delta_{n+1}} \ar[d]_{\phi_{n+1}} & W_{n+1}(S, P) \ar[d]_F \\
(S_{n}, P_{n}) \ar[r]_{\delta_n} \ar[r]_{\delta_{n}} & W_{n}(S, P).
}
\]
In particular, if ${\bf w}_0: W_n(S,  P) \to (S,  P)$ denotes the morphism induced by the $0$-th Witt polynomial 
${\bf w}_0: W_n(S) \to S$ and ${\rm id}_P$, the composite ${\bf w}_0 \circ \delta_n: 
(S_n, P_n) \to (S,  P)$ is equal to the transition morphism
$(S_n, P_n) \to (S_1, P_1) = (S,  P). $ 
\end{enumerate}
In the following, we will denote the log Frobenius lift \eqref{eq:fl1}
simply by $((S_n,  P_n), \phi_n, \delta_n)_n$ or by 
$(S_n, \phi_n, \delta_n)_n$. 
\end{definition}

It is proven in \cite[Lem.~5.5(1)]{Ma} that 
there always exists 
a log Frobenius lift of a log smooth morphism of 
pre-log rings as above. 

\begin{example}\label{exam}
Here we give a construction of log Frobenis lift when 
the log smooth morphism $(R,  Q) \to (S,  P)$ in the definition is of the form 
$(R[Q],  Q) \to (R[P],  P)$, where $P,  Q$ are fine monoids 
and the morphism is the one naturally induced by a monoid 
homomorphism $Q \to P$. 

Put $A_n := W_n(R[Q]) \otimes_{{\mathbb{Z}}[Q]} {\mathbb{Z}}[P]$, and let 
$\phi:A_{n+1} \to A_{n}$ is the morphism over $F: W_{n+1}(R[Q]) \to W_{n}(R[Q])$
which sends $1 \otimes T^x \, (x \in P)$ to $1 \otimes T^{px} \, (x \in P)$, 
where we denoted the image of $x \in P$ by the canonical map $P \to  
R[P]$ by $T^x$. Also, let $\delta_{n}: A_n \to W_n(R[P])$ be the morphism 
extending $W_n(R[Q]) \to W_n(R[P])$ which sends 
$T^x \, (x \in P)$ to $[T^x]$. Then the triple $((A_n, P), \phi_n, \delta_n)_n$
is a log Frobenius lift of $(R[P],  P)$ over 
$(R[Q], Q)$. 
\end{example}

Next we recall the definition of a log Frobenius lift of a log smooth morphism of 
fine log schemes.

\begin{definition}
\cite[Def.~5.4]{Ma} \, Let $Y$ be a scheme on which $p$ is nilpotent 
and let $f: (X, {\mathcal M}) \to (Y , {\mathcal N})$ be a log smooth 
morphism of fine log schemes. 
A log Frobenius lift of $(X, {\mathcal M})$ over $(Y , {\mathcal N})$ 
is a triple 
\begin{align}
( \{(X_n, {\mathcal M}_n)\}_{n \geq 1}, 
\{\Phi_n: (X_{n-1}, {\mathcal M}_{n-1}) & \to (X_{n},  {\mathcal M}_{n})\}_{n \geq 2}, 
\label{eq:fl2} \\ 
& \{\Delta_n: W_n(X, {\mathcal M}) \to (X_n, {\mathcal M}_n)\}_{n \geq 1}) \nonumber 
\end{align}
satisfying the following conditions: 
\begin{enumerate}
\item[(1)] 
$\{(X_n, {\mathcal M}_n)\}_{n \geq 1}$ 
is an inductive system of fine log schemes over 
$\{W_n(Y,  {\mathcal N})\}_{n \geq 1}$ with 
$(X_1, {\mathcal M}_1) = (X,  {\mathcal M})$ such that each 
$(X_n, {\mathcal M}_n)$ is log smooth over $W_n(Y,  {\mathcal N})$
and that the transition morphisms induce the isomorphisms 
$$ (X_n, {\mathcal M}_n) \overset{\cong}{\to} 
W_n(Y,  {\mathcal N}) \times_{W_{n+1}(Y,  {\mathcal N})} 
(X_{n+1}, {\mathcal M}_{n+1}) \qquad (n \geq 1). $$
\item[(2)] 
$\{\Phi_n: (X_{n-1}, {\mathcal M}_{n-1}) \to (X_{n},  {\mathcal M}_{n})
\}_{n \geq 2}$ is a morphism of inductive systems which is compatible 
with the morphism 
$\{F: W_{n-1}(Y,  {\mathcal N}) \to W_{n-1}(Y,  {\mathcal N}) \}_{n \geq 2}$
 and the absolute Frobenius 
$(X,  {\mathcal M}) \otimes_{{\mathbb Z}_{(p)}} {\mathbb F}_p 
\to (X,  {\mathcal M}) \otimes_{{\mathbb Z}} {\mathbb F}_p$. 
\item[(3)] $\{\Delta_n: W_n(X, {\mathcal M}) \to (X_n, {\mathcal M}_n)\}_{n \geq 1}$ 
is a morphism of inductive systems over 
$\{W_n(Y,  {\mathcal N})\}_{n \geq 1}$ such that the following diagram 
is commutative: 
\[ \xymatrix{
W_n(X,  {\mathcal M}) \ar[r]_{\Delta_n} \ar[d]_{F} & (X_n, {\mathcal M}_n) \ar[d]_{\Phi_{n+1}} \\
W_{n+1}(X,  {\mathcal M}) \ar[r]_{\Delta_{n+1}} & (X_{n+1}, {\mathcal M}_{n+1}) . 
}
\]
In particular, if $w_0: (X, {\mathcal M}) \to W_n(X,  {\mathcal M})$
denotes the morphism induced by the $0$-th Witt polynomial 
$w_0: X \to W_n(X)$ and 
${\rm id}_{\mathcal M}$, the composite $\Delta_n \circ w_0: 
(X, {\mathcal M}) \to (X_n,  {\mathcal M}_n)$ is equal to the transition 
morphism 
$$ (X, {\mathcal M}) = (X_1, {\mathcal M}_1) \to (X_n, {\mathcal M}_n). $$ 
In the following, we will denote the log Frobenius lift \eqref{eq:fl2} simply by 
$((X_n,  {\mathcal M}_n), \Phi_n, \Delta_n)_n$ or by $(X_n, \Phi_n, \Delta_n)_n$.
\end{enumerate}
\end{definition}

It is proven in \cite[Lem.~5.5(2)]{Ma} that 
there always exists 
a log Frobenius lift of a log smooth morphism 
$(X, {\mathcal M}) \to (Y , {\mathcal N})$ of fine log schemes as above 
etale locally on $X$. 

\quad

Now we define the comparison morphism using log Frobenius lift, 
following \cite[Sec.~6.2]{Ma}.

Let $Y$ be a scheme on which $p$ is nilpotent and let 
$f: (X, \mathcal{M}) \to (Y, \mathcal{N})$ be a morphism of fine log schemes. 
%such that the canonical pd-structure on 
%the Witt scheme $W_m(Y)$ extends to $X$. 
Then we have a canonical morphism of topoi 
\[ u_m: ((X, \mathcal{M})/W_m(Y, \mathcal{N}))_{{\rm crys}}^{{\rm log}} \to X_{{\rm et}}. \]
We denote the structure sheaf on the log crystalline topos 
$((X, \mathcal{M})/W_m(Y, \mathcal{N}))_{{\rm crys}}^{{\rm log}}$ by 
$\mathcal{O}_m$. 

Then we define the comparison morphism 
\[\mathbb{R}u_{m*} \mathcal{O}_m \to W_m \Lambda_{(X, \mathcal{M})/(Y, \mathcal{N})}^{\bullet},  \]
which is a morphism in the derivd category 
$D^{+}(X_{\rm et},  f^{-1}(W_m({\mathcal O}_Y)))$, in the following way. 

First we consider the case where 
$(X, \mathcal{M})$ admits an immersion into a fine log scheme 
 $(Z, \mathcal{L})$ which is log smooth over $(Y, \mathcal{N})$ and 
there exists a log Frobenius lift $((Z_m, \mathcal{L}_m), \Phi_m, \Delta_m)_m$ 
of the morphism  $(Z, \mathcal{L}) \to (Y, \mathcal{N})$. 
In this case, we have the following commutative diagram: 
\[
\xymatrix{
(X, \mathcal{M}) \ar[r] \ar[d]_{w_0} & (Z, \mathcal{L}) \ar[d]_{w_0} \ar[r] & (Z_m, \mathcal{L}_m) \\
W_m(X, \mathcal{M}) \ar[r]  & W_m(Z, \mathcal{L})  \ar[ru]_{\Delta_m}. &  \\
}
\]
So the composite 
$W_m(X, \mathcal{M}) \to W_m(Z, \mathcal{L}) \xrightarrow{\Delta_m} (Z_m, \mathcal{L}_m)$ 
factors through the log pd-envelope $(\bar{Z}_m, \bar{\mathcal{N}}_m)$ of 
the composite $(X, \mathcal{M}) \to  (Z,  \mathcal{L}) \to (Z_m,  \mathcal{L}_m)$
with respect to the canonical pd-structure on $W_m(Y)$. 
We denote the resulting morphism by 
$\mu : W_m(X, \mathcal{M}) \to (\bar{Z}_m, \bar{\mathcal{L}}_m) $. 
Then, by \cite[Thm 6.4]{Kato}, we have the isomorphism 
\[ \mathbb{R}u_{m*}\mathcal{O}_m \to \mathcal{O}_{\bar{Z}_m} \otimes_{\mathcal{O}_{Z_m}} \Lambda_{(Z_m, \mathcal{L}_m)/W_m(Y, \mathcal{N})}^{\bullet} \]
in $D^{+}(X_{\rm et},  f^{-1}(W_m({\mathcal O}_Y)))$. 
(Note that, since the morphism $X \to \bar{Z}_m$ defined by the construction of log 
pd-envelope is a nilimmersion, the sheaves appearing on 
right hand side can be regarded as sheaves on $X_{{\rm et}}$.)

Also, by \cite[Sec.~6.1]{Ma}, we have the isomorphism of complexes 
\[ \breve{\Lambda}_{(\bar{Z}_m, \bar{\mathcal{L}}_m)/W_m(Y, \mathcal{N})}^{\bullet} \xrightarrow{\cong} \mathcal{O}_{\bar{Z}_m} \otimes_{\mathcal{O}_{\bar{Z}_m}} \Lambda_{(Z_m, \mathcal{L}_m)/W_m(Y, \mathcal{N})}^{\bullet}. \]
Moreover, the morphism 
$\mu$ induces the map 
$\breve{\Lambda}_{(\bar{Z}_m, \bar{\mathcal{L}}_m)/W_m(Y, \mathcal{N})}^{\bullet}  \to  \breve{\Lambda}_{W_m (X, \mathcal{M})/W_m(Y, \mathcal{N})}^{\bullet}$, and 
there exists a canonical surjection 
$\breve{\Lambda}_{W_m (X, \mathcal{M})/W_m(Y, \mathcal{N})}^{\bullet} \to W_m\Lambda_{(X, \mathcal{M})/(Y, \mathcal{N})}^{\bullet}$ as we explained after Proposition-Definition \ref{thm30}. 
The comparison morphism is defined as the composition of these maps: 
\begin{eqnarray}
\mathbb{R}u_{m*}\mathcal{O}_m 
& \xrightarrow{\cong} & \mathcal{O}_{\bar{Z}_m} \otimes_{\mathcal{O}_{{Z}_m}} \Lambda_{(Z_m, \mathcal{L}_m)/W_m(Y, \mathcal{N})}^{\bullet} \nonumber \\
& \xleftarrow{\cong} & \breve{\Lambda}_{(\bar{Z}_m, \bar{\mathcal{L}}_m)/W_m(Y, \mathcal{N})}^{\bullet} \nonumber \\
& \to & \breve{\Lambda}_{W_m (X, \mathcal{M})/W_m(Y, \mathcal{N})}^{\bullet} \nonumber \\
& \to & W_m \Lambda_{(X, \mathcal{M})/(Y, \mathcal{N})}^{\bullet}. \label{eq:comp}
\end{eqnarray}

In general case, we take a strict etale covering $\{(X(i),{\mathcal M}(i)) \}_{i \in I} $
of $(X, \mathcal{M})$ such that each $(X(i),{\mathcal M}(i))$ admits 
an immersion into a fine log scheme 
 $(Z(i), \mathcal{L}(i))$ which is log smooth over $(Y, \mathcal{N})$ such that 
there exists a log Frobenius lift $((Z_m(i), \mathcal{L}_m(i)), \Phi_m(i), \Delta_m(i))_m$ 
of the morphism  $(Z(i), \mathcal{L}(i)
) \to (Y, \mathcal{N})$. For $i_1, \ldots, i_r \in I$, we put 
\begin{align*}
& (X(i_1, \ldots, i_r), {\mathcal M}(i_i, \ldots, i_r)) :=
(X(i_1), {\mathcal M}(i_1)) \times_{(X {\mathcal M})} \cdots \times_{(X {\mathcal M})} 
(X(i_r), {\mathcal M}(i_r)), \\
& (Z(i_1, \ldots, i_r), {\mathcal L}(i_1, \ldots, i_r)) :=  
(Z(i_1), {\mathcal L}(i_1)) \times_{(Y, \mathcal{N})} \cdots 
\times_{(Y, \mathcal{N})}  (Z(i_r), {\mathcal L}(i_r)), \\ 
& (Z_m(i_1, \ldots, i_r), {\mathcal L}_m(i_1, \ldots, i_r)) :=  
(Z_m(i_1), {\mathcal L}_m(i_1)) \times_{W_m(Y, \mathcal{N})} \cdots 
\times_{W_m(Y, \mathcal{N})}  (Z_m(i_r), {\mathcal L}_m(i_r)), 
\end{align*}
and for 
% Then we have the induced immersion $X(i_1, \ldots, i_r) \to Z_m(i_1, \ldots, i_r)$. 
$r \in \mathbb{N}$, we put 
\begin{align*}
& (X^r, {\mathcal M}^r) := \sqcup_{i_1,  \ldots i_r \in I}(X(i_1, \ldots, i_r), {\mathcal M}(i_i, \ldots, i_r)), \\ 
& (Z^r, {\mathcal L}^r) := \sqcup_{i_1,  \ldots i_r \in I} 
(Z (i_1, \ldots, i_r),  {\mathcal L}(i_1, \ldots, i_r)), \\ 
& (Z_m^r, {\mathcal L}_m^r) := \sqcup_{i_1,  \ldots i_r \in I} 
(Z_{m} (i_1, \ldots, i_r),  {\mathcal L}_m(i_1, \ldots, i_r)). 
\end{align*}
Then we have a immersion 
$(X^{\bullet}, {\mathcal M}^{\bullet}) \hookrightarrow (Z^{\bullet}, {\mathcal L}^{\bullet})$
of simplicial fine log schemes such that 
$ (Z^{\bullet}, {\mathcal L}^{\bullet})$ is log smooth over 
$(Y, \mathcal{N})$ and $\{(Z_m^{\bullet}, \mathcal{L}_m^{\bullet})\}_m$ 
is a part of the data of a simplicial version of log Frobenius lift of the morphism 
$(Z^{\bullet}, {\mathcal L}^{\bullet}) \to (Y, \mathcal{N})$. 
We denote the log pd-envelope of the composite 
$(X^{\bullet}, {\mathcal M}^{\bullet}) \to (Z^{\bullet}, {\mathcal L}^{\bullet}) 
\to (Z_m^{\bullet}, \mathcal{L}_m^{\bullet})$
with respect to the canonical pd-structure on $W_m(Y)$ 
by $(\bar{Z}_m^{\bullet}, \bar{\mathcal{L}}_m^{\bullet})$. 
Also, let $\theta: (X^{\bullet})_{{\rm et}}^{~} \to X_{{\rm et}}^{~}$ be the canonical augmentation morphism. 

Then, by combining the simplicial version 
\[ \mathcal{O}_{\bar{Z}_m^{\bullet}} \otimes_{\mathcal{O}_{Z_m^{\bullet}}} \Lambda_{(Z_m^{\bullet}, \mathcal{L}_m^{\bullet})/W_m(Y, \mathcal{N})}^{\bullet}
 \to W_m \Lambda_{(X^{\bullet}, \mathcal{M}^{\bullet})/(Y, \mathcal{N})}^{\bullet} \]
of the morphism \eqref{eq:comp}, the isomorphism 
 \[ \mathbb{R}u_{m*}\mathcal{O}_m \xrightarrow{\cong} \mathbb{R}\theta_{*} (\mathcal{O}_{\bar{Z}_m^{\bullet}} \otimes_{\mathcal{O}_{\bar{Z}_m^{\bullet}}} \Lambda_{(Z_m^{\bullet}, \mathcal{L}_m^{\bullet})/W_m(Y, \mathcal{N})}^{\bullet}) \]
in \cite[Prop.~2.20]{HK} and the isomorphism 
\[ \mathbb{R}\theta_{*} W_m\Lambda_{(X^{\bullet}, \mathcal{M}^{\bullet})/(Y, \mathcal{N})}^{\bullet} \xrightarrow{\cong} W_m\Lambda_{(X, \mathcal{M})/(Y, \mathcal{N})}^{\bullet} \]
which follows from the etale base change property of 
relative log de Rham-Witt complex (\cite[Prop.~3.7]{Ma}), we can construct 
the comparison morphism 
\[ \mathbb{R}u_{m*}\mathcal{O}_m \to W_m\Lambda_{(X, \mathcal{M})/(Y, \mathcal{N})}^{\bullet}, \]
as required. 

\begin{rem}
One can prove that the comparison morphism is independent of any choices 
we made, and it is functorial with respect to $(X,{\mathcal M})/(Y,{\mathcal N})$. 
Thus, to prove that it is a quasi-isomorphism, we can work etale locally on 
$X$ and on $Y$. 
\end{rem}

\section{Relative log de Rham-Witt complex (I)}

In the first four subsection in this section, we study the relative log de Rham-Witt complex 
$W_m\Lambda^{\bullet}_{(R[P],P)/(R,*)}$ for a ${\mathbb Z}_{(p)}$-algebra $R$ 
% in which $p$ is nilpotent 
and an fs monoid $P$ with $P^{\rm gp}$ torsion free. 
The main result is the existence of natural decomposition 
\begin{equation}\label{eq:decomp}
W_m\Lambda^{\bullet}_{(R[P],P)/(R,*)} 
= \bigoplus_{x \in P[\frac{1}{p}]} W_m\Lambda^{\bullet}_{(R[P],P)/(R,*), x}
\end{equation}
indexed by $P[\frac{1}{p}]$ (see Definition \ref{2.1.1} for the definition of $P[\frac{1}{p}]$). 
The strategy of the proof is as follows. 
\begin{enumerate}
\item[(1)] 
First we introduce the notion of $P$-basic Witt differentials, which generates 
$W_m\Lambda^{\bullet}_{(R[P],P)/(R,*)}$. 
\item[(2)] 
Next we introduce the notion of basic Witt differentials in the sense of 
Langer-Zink for the complex 
$W_m\Lambda^{\bullet}_{(R[P^{\rm gp}],P^{\rm gp})/(R,*)} \cong 
W_m\Omega^{\bullet}_{R[P^{\rm gp}]/R}
$ and prove that they form `a basis'. This is a variant of the result in 
\cite[Sec.~2]{LZ}. 
\item[(3)] 
Using $P$-basic Witt differentials, we define the $x$-part 
$W_m\Lambda^{\bullet}_{(R[P],P)/(R,*), x}$ of 
$W_m\Lambda^{\bullet}_{(R[P],P)/(R,*)}$ for $x \in P[\frac{1}{p}]$ 
with \[W_m\Lambda^{\bullet}_{(R[P],P)/(R,*)} 
= \sum_{x \in P[\frac{1}{p}]} W_m\Lambda^{\bullet}_{(R[P],P)/(R,*), x}. \] 
If we apply this construction to 
$W_m\Lambda^{\bullet}_{(R[P^{\rm gp}],P^{\rm gp})/(R,*)}$ (using 
$P^{\rm gp}$-basic Witt differentials), we can define the $x$-part 
$W_m\Lambda^{\bullet}_{(R[P^{\rm gp}],P^{\rm gp})/(R,*), x}$ of  
$W_m\Lambda^{\bullet}_{(R[P^{\rm gp}],P^{\rm gp})/(R,*)}$ for $x \in 
P^{\rm gp}[\frac{1}{p}]$. 
By relating $P^{\rm gp}$-basic Witt differentials 
with the basic Witt differentials in (2), we prove that 
\[ 
W_m\Lambda^{\bullet}_{(R[P^{\rm gp}],P^{\rm gp})/(R,*)} = 
\bigoplus_{x \in P^{\rm gp}[\frac{1}{p}]} W_m\Lambda^{\bullet}_{(R[P^{\rm gp}],P^{\rm gp})/(R,*), x}.\] 
Then, by relating $P$-basic Witt differentials in (1) with the basic Witt differentials in (2), 
we prove that 
\begin{equation}\label{eq:iso}
W_m\Lambda^{\bullet}_{(R[P],P)/(R,*), x} \cong 
W_m\Lambda^{\bullet}_{(R[P^{\rm gp}],P^{\rm gp})/(R,*), x} \quad (x \in P[\tfrac{1}{p}]). 
\end{equation}
By combining these, we obtain the decomposition \eqref{eq:decomp}. 
\end{enumerate}
In the last subsection, by using the decomposition \eqref{eq:decomp} and the isomorphism \eqref{eq:iso}, we prove that 
the comparison morphism 
\[\mathbb{R}u_{m*} \mathcal{O}_m \to W_m \Lambda_{(X, \mathcal{M})/(Y, \mathcal{N})}^{\bullet} \]
constructed in Section 1.3 is a quasi-isomorphism when ${\mathcal N}$ is trivial and 
$(X, \mathcal{M}) \to (Y, \mathcal{N})$ is a log smooth morphism of fs log schemes. 

\subsection{Preliminaries on monoids (I)}
In this subsection, we give several more notions on monoids which we use 
in the rest of the article. 

Let $p$ be a prime and let $\mathbb{Z} [\frac{1}{p}]:=\{ \frac{x}{p^n} \in \mathbb{Q}; n \in \mathbb{Z} \} $. Then, for a monoid $P$, the group 
$ P^{\rm gp} \otimes_{\mathbb{Z}} \mathbb{Z} [\frac{1}{p}] $ is defined. 

\begin{definition}\label{2.1.1}
For an integral monoid $P$, we put 
\[ P[\tfrac{1}{p} ] := \{ x \in P^{\rm gp} \otimes \mathbb{Z} [ \tfrac{1}{p} ] \, ; \,  \exists n \in \mathbb{N} , p^n x \in P \}. \]
\end{definition}

\begin{definition}
For an integral monoid $P$, define the function 
$u: P[\frac{1}{p}] \to \mathbb{N}$ by 
\[u(x)=u_{P}(x) := \min \{ n \geq 0 ; p^n x \in P \}. \]
\end{definition}

\begin{lemma}
Let $P$ be an integral monoid. 
 \begin{enumerate}[(1)]
  \item For any $x \in  P[\frac{1}{p}] $, $u(px)=u(x)-1$. 
  \item For any $x,  y  \in P[\frac{1}{p}] $, $ u(x+y) \leq \max \{ u(x), u(y) \}. $
%Moreover, if $P$ is saturated and $u(x) \not= u(y)$, $ u(x+y) = \max \{ u(x), u(y) \}. $
\end{enumerate}
\end{lemma}

\begin{proof}
Trivial. 
\end{proof}

\begin{comment}
\begin{proof} 
The assertion (1) and the former assertion of (2) are trivial. 
We prove the latter assertion of (2). We take elements $x,  y  \in P[\frac{1}{p}]$ 
with $u(x) < u(y), u(x+y) \leq u(y) -1$ and deduce a contradiction. 
% In this case, $ u(y) \geq 1$. 
By assumption, $ p^{u(y)-1}(x+y) \in P,  p^{u(y)-1}x \in P$ and so 
$ p^{u(y)-1} y =  p^{u(y)-1}(x+y) -   p^{u(y)-1}x  \in P^{\rm gp} $. 
Since $ p^{u(y)} y  \in P$ and $P$ is saturated, we deduce that 
$ p^{u(y)-1}y \in P $. This contradicts the definition of $u(y)$. 
\end{proof}
\end{comment}

Note that, for an integral monoid $P$, the function 
$u_{P^{\rm gp}}: P^{\rm gp}[\frac{1}{p}] \to \mathbb{N}$ is also defined. 

\begin{lemma}\label{thm3}
Let $P$ be a saturated monoid. Then 
$ u_{P}(x)=u_{P^{\rm gp}}(x)$ for 
$x \in P[\frac{1}{p}] $. 
\end{lemma}

\begin{proof}
We have $ u_{P}(x) \geq u_{P^{\rm gp}}(x) $ by definition. 
On the other hand, since $p^{u_P(x) - u_{P^{\rm gp}}(x)} p^{ u_{P^{\rm gp}}(x)} x \in P $
and $P$ is saturated, $p^{ u_{P^{\rm gp}}(x)} x \in P$. Hence 
$ u_{P}(x) \leq u_{P^{\rm gp}}(x) $. 
\end{proof}

\subsection{$P$-basic Witt differential}

Let $p$ be a prime, let $R$ be a $\mathbb{Z}_{(p)}$-algebra and let $P$ be an fs monoid 
with $P^{\rm gp}$ torsion free, and we fix an isomorphism 
$P^{\rm gp} \cong \mathbb{Z}^r$. 
In the following in this article, we denote the image of $x \in P$
by the canonical map $P \to R[P]$ by $T^x$ and put 
$X^x := [T^x] \in W(R[P])$. In this subsection, we introduce the notion of 
$P$-basic Witt differentials and prove that any element in 
$W_m \Lambda_{(R[P], P )/(R, {*})}^{\bullet}$ can be written as a sum of 
$P$-basic Witt differentials.  

\begin{definition}
For $x \in P[\frac{1}{p}]$ and $\xi \in {}^{V^{u(x)}}W(R)$, 
we define the $P$-basic Witt differential $b(\xi, x) \in W(R[P])$ of degree $0$ 
by 
\[ b(\xi, x) = {}^{V^{u(x)}}(\eta X^{p^{u(x)}x}),\] where $\xi = {}^{V^{u(x)}}\eta$. 
We call $x$ the weight of $b(\xi, x)$.  
\end{definition}

Note that, by construction in 
\cite[Sec.~3.4]{Ma}, $W\Lambda_{(R[P],  P)/(R,  {*})}^{0} = W(R[P])$. 

\begin{lemma}\label{thm11}
Any element in $W(R[P])$ can be written uniquely as the following convergent sum: 
\[  \sum_{x \in P[\frac{1}{p}]}  b(\xi_x, x) \qquad (\xi_x \in W(R)).  \] 
Here the convergence means that, for any $m \in {\mathbb{N}}$, 
${\xi}_{x} \in {}^{V^{m}}W(R)$ for almost all $x$. Moreover, the above element 
belongs to $\mbox{Ker}(W(R[P]) \to W_m (R[P]))$ if and only if 
${\xi}_{x} \in {}^{V^{m}}W(R)$ for all $x \in P[\frac{1}{p}]$. 
\end{lemma}

\begin{proof} (cf. \cite[Prop.~2.3]{LZ}) 
Take any $z \in W(R[P])$ and let 
${\bf w}_{0} (z)=\displaystyle{\sum_{x \in P} a_{x}T^{x}} \, (a_x \in R)$. 
Then 
\[ z -  \sum [a_{x}] X^{x}  \in {}^{V}W(R[P]).\]
Using this and arguing by induction, we obtain the unique expression of $z$ of the form 
\[ z = \sum_{m \geq 0 , x \in P}  {}^{V^{m}}([a_{x,  m}]X^{x}). \]
Next we rewrite the each term in the following way: For each $m, x$, let $n$ 
be the maximal element in $\mathbb{N}$ with $p^{-n}x \in P$ and $n \leq m$. 
Then 
\[^{V^{m}}([a_{x,  m}]X^{x}) =^{V^{m-n}}(^{V^{n}}[a_{x,  m}] X^{p^{-n}x}). \]
Then we obtain the required expression for $z$. 
The latter assertion also follows from this argument. 
\end{proof}

\begin{lemma}\label{thm12}
A product of $P$-basic Witt differentials of degree $0$ 
is again a $P$-basic Witt differential of degree $0$. 
\end{lemma}
\begin{proof}
It suffices to prove that the product of 
$b(\xi_1, x_1)$, $b(\xi_2,  x_2)$ is a $P$-basic Witt differential of degree $0$ 
when $u(x_1) \geq u(x_2)$. The product is calculated as follows: 
\begin{eqnarray*}
b(\xi_1, x_1) b(\xi_2,  x_2)
&= & ^{V^{u(x_1)}}(\eta_1 X^{p^{u(x_1)}x_1}) {}^{V^{u(x_2)}}(\eta_2 X^{p^{u(x_2)}x_2}) \\
&= & ^{V^{u(x_2)}}( {}^{V^{u(x_1)-u(x_2)}} (\eta_1 X^{p^{u(x_1)}x_1}) {}^{F^{u(x_2)}V^{u(x_2)}}(\eta_2 X^{p^{u(x_2)}x_2})) \\
&= & ^{V^{u(x_2)}}( {}^{V^{u(x_1)-u(x_2)}} (\eta_1 X^{p^{u(x_1)}x_1})(p^{u(x_2)} \eta_2 X^{p^{u(x_2)}x_2})) \\
&= & ^{V^{u(x_2)}}( {}^{V^{u(x_1)-u(x_2)}} (\eta_1 p^{u(x_2)} X^{p^{u(x_1)}x_1})( \eta_2 X^{p^{u(x_2)}x_2})) \\
&= & ^{V^{u(x_1)}}(\eta_1  p^{u(x_2)} {}^{F^{u(x_1)-u(x_2)}}( \eta_2 ) X^{p^{u(x_1)}(x_1+x_2)} ).
\end{eqnarray*}
Since $u(x_1+x_2) \leq u(x_1)$, we can calculate further as follows: 
\begin{eqnarray*}
&\quad & ^{V^{u(x_1)}}(\eta_1 p^{u(x_2)} {}^{F^{u(x_1)-u(x_2)}}( \eta_2 ) X^{p^{u(x_1)}(x_1+x_2)} )  \\
&= & ^{V^{u(x_1)}}(\eta_1 p^{u(x_2)} {}^{F^{u(x_1)-u(x_2)}}(  \eta_2 ) {}^{F^{u(x_1)-u(x_1+x_2)}} X^{p^{u(x_1+x_2)}(x_1+x_2)} ) \\
&= & ^{V^{u(x_1+x_2)}}({}^{V^{u(x_1)-u(x_1+x_2)}}(\eta_1 p^{u(x_2)} {}^{F^{u(x_1)-u(x_2)}}(\eta_2 ))  X^{p^{u(x_1+x_2)}(x_1+x_2)} ). 
\end{eqnarray*}
The last one is a $P$-basic Witt differential of degree $0$.  
\end{proof}

Next we discuss the elements of degree $1$. 
Since $P \to W\Lambda_{(R[P],  P)/(R,  {*})}^{1}; \, x \mapsto d\log X^x$
is a homomorphism of monoids, it induces the group homomorphism 
${\mathbb{Z}}^r \cong P^{\rm gp} \to W\Lambda_{(R[P],  P)/(R,  {*})}^{1}$. 
We denote the image of the $i$-th standard basis of $\mathbb{Z}^r$
by this map by $d \log X_i$. 

\begin{definition}
For $x \in P[\frac{1}{p}]$, $\xi \in {}^{V^{u(x)}}W(R)$ and $n = 0,  1, \ldots, r$, 
we define the $P$-basic Witt differential $b(\xi, x, n) \in W\Lambda^{1}_{(R[P],P)/(R,*)}$ 
of degree $1$ as follows: 
 \[ b(\xi, x, n) = \left\{ \begin{array}{ll}
  db(\xi, x) & (n= 0),  \\
  b(\xi, x)d\log X_n  & (n \neq 0). 
\end{array}  \right.\]
Here $b(\xi, x)$ is the $P$-basic Witt differential of degree $0$. 
$x$ is called the weight of $b(\xi, x, n)$. 
\end{definition}

\begin{prop}\label{thm16}
Any element in 
$W\Lambda_{(R[P],  P)/(R,  {*})}^{1}$
can be written as the following convergent sum 
\[ \sum_{x \in P[\frac{1}{p}], 0 \leq n \leq r} b(\xi_{x,  n}, x, n) \quad (\xi_{x,  n} \in {}^{V^{u(x)}}W(R) ), \]
%\[ \sum_{x \in P[\frac{1}{p}], 1 \leq i \leq r} {}^{V^{u(x)}}({\eta}_{x, i }X^{p^{u(x)}x})d\log X_i + \sum_{x \in P[\frac{1}{p}]} d^{V^{u(x)}}({\eta}_{x}X^{p^{u(x)}x}) \].
where the convergence is defined in the same way as that in Lemma 
\ref{thm11}. 
\end{prop}

Note that we do not claim the uniqueness of the expression in the above 
proposition. 

\begin{proof}
Note first that, by Lemma \ref{thm11} and Lemma \ref{thm12}, any element in 
$W\Lambda_{(R[P],  P)/(R,  {*})}^{1}$ is written as the convergent sum of the 
elements of the following types: 
\begin{enumerate}
\item $b(\xi, x) d \log X^y$,
\item $b(\xi_1, x_1) db(\xi_2,  x_2)$.  
\end{enumerate}
Here $y \in P$ and $b(\xi, x)$, $b(\xi_1, x_1)$, $b(\xi_2,  x_2)$
are $P$-basic Witt differentials of degree $0$. 
So it suffices to prove that the element of the above types can be written as 
a sum of $P$-basic Witt differentials of degree $1$. 

An element of type 1 is calculated as follows: Denote the image of 
$y \in P$ by the map $P \to P^{\rm gp} \cong \mathbb{Z}^r$ by $(y_i)_{i}$. 
Then  
\[ b(\xi, x) d\log X^y = b(\xi, x) \sum_{i=1}^{r} y_{i} d\log X_i = \sum_{i=1}^r b(y_i\xi,x,i). 
\]

An element of type 2 is calculated as follows: 
\medskip 

\noindent 
{\bf Case 1} \, The case $u(x_1) \geq u(x_2)$. 

We denote the image of $p^{u(x_2)}x_2 \in P$ by the map $P \to P^{\rm gp} \cong {\mathbb{Z}}^r$ by $(z_i)_i$. Then,    
{\allowdisplaybreaks{
\begin{eqnarray*}
b(\xi_1, x_1) db(\xi_2,  x_2)
&= & ^{V^{u(x_1)}}(\eta_1 X^{p^{u(x_1)}x_1}) d^{V^{u(x_2)}}(\eta_2 X^{p^{u(x_2)}x_2}) \\
&= & ^{V^{u(x_2)}}( {}^{V^{u(x_1)-u(x_2)}} (\eta_1 X^{p^{u(x_1)}x_1}) {}^{F^{u(x_2)}}d^{V^{u(x_2)}}(\eta_2 X^{p^{u(x_2)}x_2})) \\
&= & ^{V^{u(x_2)}}( {}^{V^{u(x_1)-u(x_2)}} (\eta_1 X^{p^{u(x_1)}x_1}) d(\eta_2 X^{p^{u(x_2)}x_2})) \\
&= & ^{V^{u(x_2)}}( {}^{V^{u(x_1)-u(x_2)}}(\eta_1 X^{p^{u(x_1)}x_1}) (\eta_2 X^{p^{u(x_2)}x_2} d\log X^{p^{u(x_2)}x_2})) \\
&= & ^{V^{u(x_2)}}( {}^{V^{u(x_1)-u(x_2)}}(\eta_1 X^{p^{u(x_1)}x_1}) (\eta_2 X^{p^{u(x_2)}x_2} \sum_{i} z_i d\log X_i) )\\
&= & \sum_{i} {}^{V^{u(x_2)}}( {}^{V^{u(x_1)-u(x_2)}}(\eta_1 X^{p^{u(x_1)}x_1}) (\eta_2 z_i X^{p^{u(x_2)}x_2} d\log X_i)) \\
&= & \sum_{i} {}^{V^{u(x_2)}}( {}^{V^{u(x_1)-u(x_2)}}(\eta_1z_i {}^{F^{u(x_1)-u(x_2)}}( \eta_2 ) X^{p^{u(x_1)}(x_1+x_2)} d\log X_i )) \\
&= & \sum_{i} {}^{V^{u(x_1)}}(\eta_1 z_i {}^{F^{u(x_1)-u(x_2)}}( \eta_2 ) X^{p^{u(x_1)}(x_1+x_2)} d\log X_i ) \\
&= & \sum_{i} {}^{V^{u(x_1)}}(\eta_1 z_i {}^{F^{u(x_1)-u(x_2)}}( \eta_2 ) X^{p^{u(x_1)}(x_1+x_2)} )d\log X_i. 
\end{eqnarray*}}}%
The coefficient of $d\log X_i$ on the right hand side is written as a sum of 
$P$-basic Witt differentials of degree $0$ by Lemma \ref{thm12}. 
Thus the above element is written as a sum of the elements of the form $b(\xi,x,n)$ with 
$n = 1, \ldots, r$. 
\medskip 

\noindent
{\bf Case 2} \, The case $u(x_1) < u(x_2)$. 

By Leibniz rule, 
\[b(\xi_1, x_1) db(\xi_2,  x_2) =d \left( b(\xi_1, x_1) b(\xi_2,  x_2) \right)- b(\xi_2, x_2) db(\xi_1,  x_1). \]
By Lemma 
\ref{thm12}, $b(\xi_1, x_1) b(\xi_2,  x_2)$ is a $P$-basic Witt differential of degree $0$ 
and so the first term of the above element is of the form $b(\xi,x,0)$. 
The second term is written as a sum of the elements of the form $b(\xi,x,n)$ with 
$n = 1, \ldots, r$ by Case 1. 

So we finished the proof of the proposition. 
\end{proof}

Finally we discuss the elements of higher degree. 

\begin{definition}
For $x \in P[\frac{1}{p}]$, $\xi \in {}^{V^{u(x)}}W(R)$ and 
$I \subset \{ 0,  1, \ldots r \}$, we define the $P$-basic Witt differential 
$b(\xi, x, I) \in W\Lambda^{|I|}_{(R[P], R)/(R,*)}$ 
in the following way: 
 \[ b(\xi, x, I) = \left\{ \begin{array}{ll}
  db(\xi, x) \bigwedge_{i \in I \setminus \{0\}} d \log X_i & ( 0 \in I), \\
    b(\xi,   x) \bigwedge_{i \in I } d \log X_i & (0 \notin I ). 
 \end{array}  \right. \]
Here $b(\xi, x)$ is the $P$-basic Witt differential of degree $0$ and for 
$I = \{i_1, \ldots i_s \} \subset [1,  r]~ (i_1 < \cdots < i_s)$, 
\[ \bigwedge_{i \in I } d \log X_i := d \log X_{i_1} \wedge \cdots \wedge d \log X_{i_s}. \]
(When $I = \emptyset$, we put $\bigwedge_{i \in I } d \log X_i := 1$.)  
We call $x$ the weight of $b(\xi, x, I)$. 
\end{definition}

By definition, we have 
\[ db(\xi, x, I) = \left\{ \begin{array}{ll}
  0 & ( 0 \in I), \\
    b(\xi,   x,   I \cup \{ 0 \}) & (0 \notin I ).  
 \end{array}  \right. \]

\begin{cor}\label{thm23}
Any element in $W\Lambda_{(R[P],  P)/(R,  {*})}^{n}$ can be written as the following  convergent sum  
\[ \sum_{x \in P[\frac{1}{p}], \atop I \subset [0,  r], |I|=n} b(\xi_{x,  I}, x, I) \quad (\xi_{x,  n} \in {}^{V^{u(x)}}W(R) ), \]
%\[ \sum_{x \in P[\frac{1}{p}], \atop I \subset [1,  r], |I|=n} {}^{V^{u(x)}}({\eta}_{x, I}X^{p^{u(x)}x}) \bigwedge_{I} dlog X_i + \sum_{x \in P[\frac{1}{p}], \atop J \subset [1,  r], |J|=n-1} d^{V^{u(x)}}({\eta}_{x,  J}X^{p^{u(x)}x}) \bigwedge_{J} dlog X_i\]. 
where the convergence means that, for any $m \in {\mathbb{N}}$, $\xi_{x,I} \in {}^{V^m}W(R)$ for almost all $(x,I)$. 
\end{cor}

\begin{proof}
We prove the claim by induction on $n$. Because we proved the claim in the case 
$n =0, 1$, we may assume the claim in the case $n = k$ and prove the claim 
in the case $n = k+1$. An element in 
$W\Lambda_{(R[P],  P)/(R,  {*})}^{k+1}$ is written as a convergent sum of 
elements of the form $ \omega d \log X_l$ or the form $\omega d b(\xi', x')$, 
% \begin{itemize}
% \item $ \omega d \log X_l$,
% \item $ \omega d b(\xi_2, x_2)$, 
%\end{itemize}
where $\omega \in W\Lambda_{(R[P],  P)/(R,  {*})}^{k}$. 
By induction hypothesis, $\omega$ is of the form  
\[ \omega = \sum_{x \in P[\frac{1}{p}], \atop I \subset [0,  r], |I|=k} b(\xi_{x,  I}, x, I). \]
%\[ \omega = \sum_{x \in P[\frac{1}{p}], \atop I \subset [1,  r], |I|=k} {}^{V^{u(x)}}({\eta}_{x, I}X^{p^{u(x)}x} )\bigwedge_{I} dlog X_i + \sum_{x \in P[\frac{1}{p}], \atop J \subset [1,  r], |J|=k-1} d^{V^{u(x)}}({\eta}_{x,  J}X^{p^{u(x)}x}) \bigwedge_{J} dlog X_i\]
Thus we may assume that $\omega = b(\xi_{x,  I}, x, I)$. 

In the first case, the claim follows from the equality 
\[ b(\xi_{x,  I}, x, I)\bigwedge d\log X_l  =\left\{ \begin{array}{ll}
  0 & (l \in I ), \\
   b( \pm \xi ,   x,   I \cup \{ l \}) & (l \notin I). 
 \end{array}  \right. \]

In the second case, 
\[ b(\xi_{x,  I}, x, I)d b(\xi', x') = \left\{ \begin{array}{ll}
  db(\xi_{x,  I}, x) \bigwedge_{i \in I \setminus \{0\}} d \log X_i  \bigwedge d b(\xi', x') & (0 \in I), \\
    b(\xi_{x,  I},   x) \bigwedge_{i \in I } d \log X_i \bigwedge d b(\xi',   x') & (0 \notin I ).  
 \end{array}  \right. \]
On the right hand side, $b(\xi_{x,  I}, x) d b(\xi', x')$ is written as a sum of 
$P$-basic Witt differentials by Proposition \ref{thm16}. 
Also, if we express the element $b(\xi_{x,  I}, x) d b(\xi', x')$ as a sum 
\[ \sum_{x \in P[\frac{1}{p}], 0 \leq n \leq r} b(\xi_{x,  n}, x, n) \]
of $P$-basic Witt differentials, % by using Proposition \ref{thm16}, 
we have the equality 
\[ d b(\xi_{x,  I}, x) d b(\xi', x') = \sum_{x \in P[\frac{1}{p}], 1 \leq n \leq r} b(\xi_{x,  n}, x, \{ 0 ,  n \} ) \]
and so it is a sum of $P$-basic Witt differentials. 
Therefore, $b(\xi_{x,  I}, x, I)d b(\xi', x')$ is expressed as a sum of 
$P$-basic Witt differentials, as required. 
\end{proof}

\subsection{Basic Witt differential in the sense of Langer-Zink}

In this subsection, let $R$ be a ${\mathbb{Z}}_{(p)}$-algebra 
%in which $p$ is nilpotent 
and we consider the relative log de Rham-Witt complex 
$W\Lambda_{(R[P],  P)/(R,  {*})}^{\bullet}$ in the case $P = \mathbb{Z}^r$. 
We prove that any element in it is written as a convergent sum of 
basic Witt differentials in the sense of Langer-Zink \cite{LZ}. 

First we compare the  relative log de Rham-Witt complex 
$W\Lambda_{(R[{\mathbb{Z}}^r],  {\mathbb{Z}}^r)/(R,  {*})}^{\bullet}$ with 
the relative de Rham-Witt complex 
$W\Omega_{R[\mathbb{Z}^r]/R}^{\bullet}$ without log structures. 
The ring $R[\mathbb{Z}^r]$ is isomorphic to the Laurent polynomial ring 
$R[T_1^{\pm1}, \ldots, T_r^{\pm1} ]$. We put $X_i := [T_i] \in W(R[T_1^{\pm1}, \ldots,  T_r^{\pm1}])$ and for $x \in \mathbb{Z}^r$, we put $X^x := [T^x]$. 

\begin{lemma}
We can put a natural structure of log $F$-$V$-procomplex on 
$ \{ W_{m}\Omega_{R[\mathbb{Z}^r]/R}^{\bullet} \}_{m \in \mathbb{N}}$, and 
it is isomorphic as log $F$-$V$-procomplex to 
$ \{ W_m\Lambda_{(R[\mathbb{Z}^r], \mathbb{Z}^r)/(R, {*})}^{\bullet} \} _{m \in \mathbb{N}}$. 
\end{lemma}

\begin{proof}
We define for each $m$ the map 
$d\log: \mathbb{Z}^r \to W_{m}\Omega_{R[\mathbb{Z}^r]/R}^{\bullet}$
by \[\mathbb{Z}^r \ni x \mapsto d\log X^{x} := X^{-x}dX^{x}. \]
Then, the pair $(d,d\log)$ (where $d$ is the 
the derivation $W_m(R) \to W_m\Omega^1_{R[\mathbb{Z}^r]/R}$) 
forms a log derivation and this definition is compatible with respect to $m$. 
Also, we have the 
equality $d \circ d \log = 0$ 
 (where $d$ is the $W_m\Omega^1_{R[\mathbb{Z}^r]/R} \to 
W_m\Omega^2_{R[\mathbb{Z}^r]/R}$) and the equality 
\begin{align*}
^{F}d\log X^x
 & =  ^{F}(X^{-x})^{F}d(X^{x}) 
  =  X^{-px}X^{(p-1)x}dX^{x} 
  =  X^{-px}X^{px}d\log X^{x} 
  =  d\log X^{x}.  
\end{align*} 
Hence, with this log derivation, we can regard 
$ \{ W_{m}\Omega_{R[\mathbb{Z}^r]/R}^{\bullet} \}_{m \in \mathbb{N}}$ as a log $F$-$V$-procomplex. 

By the universality, there exists a unique morphism of log $F$-$V$-procomplexes 
\[ \varphi: \{ W_m\Lambda_{(R[\mathbb{Z}^r], \mathbb{Z}^r)/(R, {*})}^{\bullet} \} _{m \in \mathbb{N}}
\to \{ W_{m}\Omega_{R[\mathbb{Z}^r]/R}^{\bullet} \}_{m \in \mathbb{N}} \] 
and a unique morphism of $F$-$V$-procomplexes 
\[ \psi: \{ W_{m}\Omega_{R[\mathbb{Z}^r]/R}^{\bullet} \}_{m \in \mathbb{N}} \to 
\{ W_m\Lambda_{(R[\mathbb{Z}^r], \mathbb{Z}^r)/(R, {*})}^{\bullet} \} _{m \in \mathbb{N}}. \] 
Moreover, by the definition of $d\log$ given above, we see that $\psi$ is a 
morphism of log $F$-$V$-procomplexes. 
Thus, by the universality again, we have equalities $\psi \circ \varphi = {\rm id}$, $\varphi \circ \psi 
= {\rm id}$. So the lemma is proved. 
\begin{comment}
We prove that $ \{ W_{m}\Omega_{R[\mathbb{Z}^r]/R}^{\bullet} \}_{m \in \mathbb{N}}$ 
is isomorphic to 
$ \{ W_m\Lambda_{(R[\mathbb{Z}^r], \mathbb{Z}^r)/(R, {*})}^{\bullet} \} _{m \in \mathbb{N}}$. 
First, by the remark before 
\cite[Prop. 3.5]{Ma}, we have the isomorphism 
\[ W_m\Lambda_{(R[\mathbb{Z}^r], {*})/(R, {*})}^{\bullet} \cong W_{m}\Omega_{R[\mathbb{Z}^r]/R}^{\bullet}. \] 
Also, by the isomorphism 
$\mbox{Spec}(R[\mathbb{Z}^r] , \mathbb{Z}^r) \xrightarrow{\cong} \mbox{Spec}(R[\mathbb{Z}^r],  *)$
and Proposition-Definition \ref{thm15}, we have the isomorphism  
\[ W_m\Lambda_{(R[\mathbb{Z}^r], {*})/(R, {*})}^{\bullet} \cong W_m\Lambda_{(R[\mathbb{Z}^r], \mathbb{Z}^r)/(R, {*})}^{\bullet}.  \]
Combining these, we obtain the required isomorphism. 
\end{comment}
\end{proof}

Thus, in the following, we will prove that any element in 
$W\Omega_{R[\mathbb{Z}^r]/R}^{\bullet} = W\Omega_{R[T_1^{\pm1}, \ldots, T_r^{\pm1} ]/R}^{\bullet}$ 
is written uniquely as a convergent sum of basic Witt differentials 
in the sense of Langer-Zink. The rest of the argument in this subsection is based on the argument in 
the case of $W\Omega_{R[T_1, \ldots, T_r]/R}^{\bullet}$ treated in \cite{LZ}. (See also \cite[Sec.~4]{Ma}.) 

Let $p^{-\infty}$ be a symbol with rules  
$p \cdot p^{- \infty} := p^{- \infty} , p^{-1} \cdot p^{- \infty} := p^{- \infty}, \mbox{ord}_p{p^{- \infty}} := -\infty$. 
We call a map of sets $k: [1,  r] \to \mathbb{Z}[\frac{1}{p}] \sqcup \{ p^{- \infty } \}$ a weight and we will 
denote its value $k(i)$ at $i$ simply by $k_i$ in the following. We put 
${\rm supp}\,k := \{ i \in [1, r ] ; k_i \neq 0 \} $. For a weight $k$, we define $k^{+}$ by 
\[
  (k^{+})_{i} =\left\{ \begin{array}{ll}
  0 & (k_i = p^{- \infty}), \\
    k_{i}  & (k_i \neq p^{- \infty}). 
 \end{array}  \right.
\]
For each weight $k$, we put a total order ${\rm supp}\,k= \{ i_1, i_2, \ldots, i_s \}$
on the set ${\rm supp}\,k$ in such a way that the inequality 
$\mbox{ord}_p k_{i_{j}} \leq \mbox{ord}_p k_{i_{j+1}} $ holds for any $j$. 
Also, we assume that the total order on ${\rm supp}\,k$ is equal to that on 
${\rm supp}\,p^a k  (a \in \mathbb{Z})$. 

We call a subset $I$ of ${\rm supp}\,k$ an interval if any element $a \in {\rm supp}\,k$ which is bigger than 
some element $b$ in $I$ and smaller than some element $c$ in $I$ with respect to the total order on 
${\rm supp}\,k$ belongs to $I$. 

A tuple 
$\mathcal{P} := (I_{-\infty}, I_0, I_1, \ldots, I_l)$
of intervals in ${\rm supp}\,k $ is called a partition of ${\rm supp}\,k$ if 
$I_{- \infty} = \{ i \in [1, r] ; k_i = p^{-\infty} \} $, ${\rm supp}\,k = I_{- \infty} \sqcup I_0 \sqcup \cdots \sqcup I_l$, $I_1 , \ldots, I_l \neq \emptyset$ and if any element in $I_j$ is smaller than any element in $I_{j+1}$ 
with respect to the total order in ${\rm supp}\,k$ for $j = 0, \cdots , l-1$. 

For a weight $k$ and a nonempty set $I \subset [1,r]$ with 
$k_i \not= p^{-\infty}$ for all $i \in I$, let $t(k_{I}) \in \mathbb{Z}$ be the unique integer 
such that the elements $p^{t(k_I)} k_i \, (i \in I)$ are all integers and at least one of them is prime to $p$. 
Also, we put $u(k_{I}) := \max \{ t(k_I),  0 \}$. When there is no risk of confusion, we will denote 
$t(k_{I}), u(k_I)$ simply by $t(I), u(I)$, respectively. Also, we put $t(\emptyset) = u(\emptyset) := 0$. 

For a triple $(\xi, k, \mathcal{P})$ consisting of a weight $k$, its partition 
$\mathcal{P} = (I_{-\infty}, I_0, \ldots , I_l)$ and 
$\xi =  {}^{V^{u(I)}} \eta \in {}^{V^{u(I)}}W(R)$ (where $I = {\rm supp}\,k^+$), 
we define the element $\hat{e}(\xi, k, \mathcal{P})$ in $W\Omega_{R[\mathbb{Z}^r]/R}^{l+|I_{-\infty}|}$ 
by 
\[ \hat{e}(\xi, k, \mathcal{P}) :=  e(\xi ,  k^{+}, (I_0, \ldots , I_l) ) \cdot \bigwedge_{i \in I_{-\infty}} d\log X_i, \]
where $e(\xi ,  k^{+}, (I_0, \ldots , I_l) )$ is defined analogously to the case of \cite{LZ} 
in the following way (cf. Section 1.1): 
We put $X^{p^{t(I_i)}k_{I_i}} := \prod_{j \in I_i} X_j^{p^{t(I_i)}k_{j}}$ for each $0 \leq i \leq l$. Then, 
\leftline{When $I_0 \neq \emptyset$, }
\[
e(\xi ,  k^{+}, (I_0, \ldots , I_l) ) := {}^{V^{u(I_0)}}(\eta X^{u(I_0)k_{I_0}}) d^{V^{t(I_1)}}(X^{t(I_1)k_{I_1}})\cdots d^{V^{t(I_l)}}(X^{t(I_l)k_{I_l}}). 
 \]
When $I_0 = \emptyset, t(I_1) \geq 1$,  
\[
e(\xi ,  k^{+}, (I_0, \ldots , I_l) ) := d^{V^{t(I_1)}}(\eta X^{t(I_1)k_{I_1}}) d^{V^{t(I_2)}}(X^{t(I_2)k_{I_2}}) \cdots d^{V^{t(I_l)}}(X^{t(I_l)k_{I_l}}). 
 \]
When $I_0 = \emptyset, t(I_1) \leq 0$, 
\[
e(\xi ,  k^{+}, (I_0, \ldots , I_l) ) := \eta d^{V^{t(I_1)}}( X^{t(I_1)k_{I_1}})\cdots d^{V^{t(I_l)}}(X^{t(I_l)k_{I_l}}). 
 \]
Here, when $t(I_i) <0$, we put 
\[ d^{V^{t(I_i)}}( X^{t(I_i)k_{I_i}}) := {}^{F^{-t(I_i)}}d X^{t(I_i)k_{I_i}}.\] 
We denote the image of $\hat{e}(\xi, k, \mathcal{P})$ in 
$W_m \Omega_{R[\mathbb{Z}^r]/R}^{l+|I_{-\infty}|}$ 
or that of $e(\xi ,  k^{+}, (I_0, \ldots , I_l) )$ in 
$W_m \Omega_{R[\mathbb{Z}^r]/R}^{l}$ by the same symbol. 
We call an element in $W \Omega_{R[\mathbb{Z}^r]/R}^{l+|I_{-\infty}|}$ or 
$W_m \Omega_{R[\mathbb{Z}^r]/R}^{l+|I_{-\infty}|}$
of the form $\hat{e}(\xi, k, \mathcal{P})$ a basic Witt differential. 
A weight $k$ is called integral if $k^+_i \in {\mathbb{Z}}$ for any $i$. 

As for the derivation and the action of $V$ on basic Witt differentials 
$\hat{e}(\xi, k, (I_{-\infty}, I_0, \ldots , I_l))$, we can prove the following 
rule in exactly the same way as \cite[Prop.~2.5]{LZ}, \cite[Sec.~4.1]{Ma}:  
\begin{lemma}\label{thm4}
\begin{enumerate}[(1)]
 \item  $ d\hat{e}(\xi, k , (I_{-\infty }, I_0, \ldots , I_l)) $ \[ = \left\{ \begin{array}{ll}
  0 & (I_0 = \emptyset ~ \mbox{or} ~ {\rm supp}\,k^{+} = \emptyset), \\
  \hat{e}(\xi, k , (I_{-\infty }, \emptyset, I_0, \ldots , I_l))  & (I_0 \neq \emptyset , k^{+} \mbox{ is not integral}), \\
    p^{-t(I_0)}\hat{e}(\xi,   k ,   (I_{-\infty},   \emptyset,   I_0,   \ldots ,   I_l))  & (I_0 \neq \emptyset ,   k^{+} \mbox{ is integral}). 
 \end{array}  \right. \]
  \item $ {}^{V}\hat{e}(\xi, k , (I_{-\infty }, I_0, \ldots , I_l)) $ \[ = \left\{ \begin{array}{ll}
  \hat{e}(^{V}\xi, \frac{1}{p} k , (I_{-\infty }, I_0, \ldots , I_l))  & (I_0 \neq \emptyset ~{\rm or}~ \frac{1}{p} k^{+} \mbox{ is  integral}), \\
    \hat{e}(p {}^{V}\xi,   \frac{1}{p} k ,   (I_{-\infty }, I_0,   \ldots ,   I_l))  & (I_0 = \emptyset ,   \frac{1}{p} k^{+} \mbox{ is not integral}).  
 \end{array}  \right. \]
\end{enumerate}
%In particular, for $\hat{e} := \hat{e}(\xi, k , (I_{-\infty }, I_0, \ldots , I_l))$ with 
%$I_0 \neq \emptyset$ and $k^{+}$ not integral, $d\hat{e} = 0$ and \[
%d\hat{e}(\xi, k , (I_{-\infty }, I_1, \ldots , I_l)) = \hat{e}(\xi, k , (I_{-\infty }, \emptyset ,  I_1, \ldots , I_l))
%\]
\end{lemma}

\begin{prop}\label{thm1}
Any element in 
$W\Omega_{R[\mathbb{Z}^r]/R}^{\bullet}$ is written as a convergent sum of 
basic Witt differentials. Namely, any $\omega \in W\Omega_{R[\mathbb{Z}^r]/R}^{\bullet}$
is written as a convergent sum 
\[ \omega = \sum_{k, \mathcal{P}} \hat{e} (\xi_{k, \mathcal{P}} , k , \mathcal{P}), \]
where the definition of the convergence is the same as that in Lemma 
\ref{thm11}. 
\end{prop}

\begin{proof}
In this proof, we denote $W\Omega_{R[\mathbb{Z}^r]/R}^{\bullet}$ by 
$W\Omega_{R[T_{1}^{\pm1}, \ldots, T_{r}^{\pm 1}]/R}^{\bullet}$. Also, let 
$X_i, X^x \, (x \in \mathbb{Z}^r)$ be as in the beginning of this subsection. 
We prove the proposition by reducing to a result in \cite{LZ}. 

Since any element in 
$W\Omega_{R[T_{1}^{\pm1}, \ldots T_{r}^{\pm 1}]/R}^{l}$ 
is written as a convergent sum of elements of the form 
\begin{align}\label{7}
{}^{V^{n_1}}(\eta_1 X^{x_1}) d^{V^{n_2}}(\eta_2 X^{x_2}) \cdots d^{V^{n_l}}(\eta_l X^{x_l}), % \tag{2.3.3.1}
\end{align}
it suffices to prove that the element \eqref{7} has the expression in the statement of the 
proposition, by induction on $l$. 

First, we prove that we can reduce to the case $n_1 \geq n_2 \geq \cdots \geq n_l$. 
By changing the order of $d^{V^{n_j}}(\eta_j X^{x_j}) \, (j = 2, \ldots l)$, we may assume that 
$n_2 \geq \cdots \geq n_l$. If $n_1 \le n_2$, we use the equality 
\[ {}^{V^{n_1}}(\eta_1 X^{x_1}) d^{V^{n_2}}(\eta_2 X^{x_2}) = d({}^{V^{n_1}}(\eta_1 X^{x_1}) {}^{V^{n_2}}(\eta_2 X^{x_2})) - {}^{V^{n_2}}(\eta_2 X^{x_2})d{}^{V^{n_1}}(\eta_1 X^{x_1})\] to rewrite the element \eqref{7} to the form 
\begin{align}\label{8}
 d({}^{V^{n_1}}(\eta_1 X^{x_1}) {}^{V^{n_2}}(\eta_2 X^{x_2})) \cdots d^{V^{n_l}}(\eta_l X^{x_l}) - {}^{V^{n_2}}(\eta_2 X^{x_2})d{}^{V^{n_1}}(\eta_1 X^{x_1}) \cdots d^{V^{n_l}}(\eta_l X^{x_l}). % \tag{2.3.3.2}
\end{align}
If we put 
$\omega := {}^{V^{n_1}}(\eta_1 X^{x_1}){}^{V^{n_2}}(\eta_2 X^{x_2}) d{}^{V^{n_3}}(\eta_3 X^{x_3})\cdots d{}^{V^{n_l}}(\eta_l X^{x_l})$, 
the first term in \eqref{8} is equal to $d\omega$.  
Since $\omega \in W\Omega^{l-1}_{R[T_1^{\pm 1},  \dots,  T_r^{\pm 1}]}$ is written as a convergent sum of 
 in the statement of the proposition by induction hypothesis, so is the element $d\omega$ by Lemma \ref{thm4}. 
Also, since $n_2 \geq n_j \,  (j=1, 2, \ldots, l)$, by changing the order of $d^{V^{n_j}}(\eta_j X^{x_j}) \, (j = 1,3,4, \ldots l)$
we can rewrite the second term in \eqref{8} to an element of the form \eqref{7} with the condition 
$n_1 \geq n_2 \geq \cdots \geq n_l$ (up to the multiplication by $\pm 1$). 
Thus we can reduce to the case $n_1 \geq n_2 \geq \cdots \geq n_l$. 

Then we have 
\[ {}^{V^{n_1}}(\eta_1 X^{x_1}) d^{V^{n_2}}(\eta_2 X^{x_2}) \cdots d^{V^{n_l}}(\eta_l X^{x_l}) 
 = {}^{V^{n_1}} \left( \prod_{i=1}^{l} {}^{F^{n_1-n_i}}\eta_i \cdot X^{x_1} {}^{F^{n_1-n_2}}dX^{x_2} \cdots {}^{F^{n_1-n_l}}dX^{x_l} \right). \]
By Lemma \ref{thm4}, if the element within the bracket on the right hand side is written as a convergent sum of 
basic Witt differentials, then so is the element on the left hand side. So we can work with the element 
within the bracket on the right hand side. If we put 
$\prod_{i=1}^{l} {}^{F^{n_1-n_i}}\eta_i =: \xi$, this element is rewritten as follows: 
\[ \xi X^{x_1} {}^{F^{n_1-n_2}}dX^{x_2}\cdots {}^{F^{n_1-n_l}}dX^{x_l}  = 
\xi  \prod_{i=1}^{l} X^{(p^{n_1-n_i}-1) x_i}dX^{x_2}\cdots dX^{x_l}. \]
Moreover, by using Leibniz rule, we can rewrite the derivation $dX^{x_i}$ of a monomial $X^{x_i}$ involving many variables 
to a finite sum of the product of a polymonial and the derivation $dX_{j}^{n_{j}}$ of a monomial 
$X_{j}^{n_{j}}$ involving only one variable. Using this fact, we can work with an element of the form 
\[ \xi X^{y} dX_{i_1}^{x_1} \ldots dX_{i_l}^{x_l} \]
with $\xi \in W(R), y \in \mathbb{Z}^r, i_1 , \ldots , i_l \in [1, r],  x_1, \ldots x_r \in \mathbb{Z}_{\not= 0}$. 
If we denote the $i_k$-th entry of $y$ by $y_{i_k}$, we have the equality 
\[
X_{i_k}^{y_{i_k}} dX_{i_k}^{x_k} =\left\{ \begin{array}{ll}
   x_k X_{i_k}^{y_{i_k}+x_k-1} dX_{i_k}  & (y_{i_k} + x_k \geq 1), \\
   -x_k X_{i_k}^{y_{i_k}+x_k+1} dX_{i_k}^{-1}  & (y_{i_k} + x_k \leq -1),  
 \\ x_k X_{i_k}^{-1} dX_{i_k} = x_k d\log X_i & (y_{i_k} + x_k = 0).  
 \end{array}  \right.
\]
Using this and renumbering the indices, we can rewrite the element 
$ \xi X^{y} dX_{i_1}^{x_1} \ldots dX_{i_l}^{x_l}$ 
into an element of the form 
\begin{equation}\label{eqelt1}
\xi' \prod_{i=1}^a X_i^{y'_i} \prod_{i=a+1}^b X_i^{y'_i} 
\bigwedge_{i=b+1}^c X_i^{y'_i}dX_i \bigwedge_{i=c+1}^d X_i^{y'_i} dX_i^{-1} 
\bigwedge_{i=d+1}^e d\log X_i 
\end{equation}
with $0 \leq a \leq b \leq c \leq d \leq e \leq r$, $\xi' \in W(R)$ 
and $y'_i \in {\mathbb{Z}}$ such that $y'_i > 0$ for $i \in [1,a] \sqcup [b+1,c]$ 
and $y'_i < 0$ for $i \in [a+1,b] \sqcup [c+1,d]$. 
%$ \xi' X^{y'} dX_{i'_1}^{x'_1} \ldots dX_{i'_{l'}}^{x'_{l'}} \bigwedge_{i' \in ^{\exists}I \subset [1, r]} d\log X_{i'}$, $(\xi \in W_(R), y' \in \mathbb{Z}^r , x_k' \in \{ \pm 1 \}, y'_{i'_k} \cdot x'_{k} \geq 0)$
If we can write the element 
\begin{equation}\label{eqelt2}
\xi' \prod_{i=1}^a X_i^{y'_i} \prod_{i=a+1}^b X_i^{y'_i} 
\bigwedge_{i=b+1}^c X_i^{y'_i}dX_i \bigwedge_{i=c+1}^d X_i^{y'_i} dX_i^{-1} 
\end{equation}
as a convergent sum of basic Witt differentials in $W \Omega^{\bullet}_{R[T_1^{\pm 1}, \ldots, T_d^{\pm 1}]/R}$, 
we can write the element \eqref{eqelt1} also as a convergent sum of basic Witt differentials 
in $W \Omega^{\bullet}_{R[T_1^{\pm 1}, \ldots, T_r^{\pm 1}]/R}$ and so we are done. 
We define the ring homomorphism 
$R[U_1, \ldots U_{d}] \to R[T_{1}^{\pm1}, \ldots T_{d}^{\pm 1}]$ by 
\[
\begin{array}{ll}
  U_{i} \mapsto T_{i} & (i \in [1,a] \sqcup [b+1,c]), \\
  U_{i} \mapsto T_{i}^{-1}  & (i \in [a+1,b] \sqcup [c+1,d]),  
% \\ U_{j} \mapsto T_{j}  & (j \notin \{ i'_1 \ldots, i'_{l'} \})
\end{array} 
\]
and consider the induced map of relative log de Rham-Witt complexes 
\[ 
\varphi: W\Omega_{R[U_1, \ldots, U_{r}]/R}^{\bullet} \to 
W\Omega_{R[T_{1}^{\pm1}, \ldots, T_{r}^{\pm 1}]/R}^{\bullet}.\]
Then the element \eqref{eqelt2} is the image of some element 
$\tilde{\omega} \in W\Omega_{R[U_1, \ldots, U_{d}]/R}^{\bullet}$ by $\varphi$. 
By \cite[Thm.~2.8]{LZ}, $\tilde{\omega}$ is written as a convergent sum of 
basic Witt differentials in $ W\Omega_{R[U_1, \ldots, U_{d}]/R}^{\bullet}$ and it is easy to see 
that the basic Witt differentials in $ W\Omega_{R[U_1, \ldots, U_{d}]/R}^{\bullet}$ is sent to 
those in $W\Omega_{R[T_{1}^{\pm1}, \ldots, T_{r}^{\pm 1}]/R}^{\bullet}$ by $\varphi$. 
Therefore, the element \eqref{eqelt2} is written 
as a convergent sum of basic Witt differentials in $W \Omega^{\bullet}_{R[T_1^{\pm 1}, \ldots, T_d^{\pm 1}]/R}$, 
as required. 
\end{proof}

To prove the uniqueness of the expression in Proposition \ref{thm1}, 
we use the morphism 
\[ \omega_m : W_{m+1} \Omega_{R[\mathbb{Z}^r]/R}^{\bullet} \to \Omega_{R[\mathbb{Z}^r]/R,  {\bf w}_m}^{\bullet} \]
which we recalled in the end of Section 1.1. 
We calculate the image of the element 
$\hat{e}(\xi, k, \mathcal{P}) \in W_{m+1}\Omega_{R[\mathbb{Z}^r]/R}^{\bullet} \allowbreak 
\, (\mathcal{P} = (I_{-\infty}, I_0, \ldots, I_l))$
by $\omega_m$. When $I_{-\infty} = \emptyset$, we have the following result by exactly the same argument as 
\cite[Prop.~2.16]{LZ}. 

\begin{lemma}\label{thm18}
When $I_{-\infty} = \emptyset$, the image of the element $\hat{e} := \hat{e}(\xi, k, \mathcal{P})$ by $\omega_m$ is zero when
$p^m k$ is not integral. If $p^m k$ is integral, \[
\omega_m(\hat{e}) = \left\{
\begin{array}{ll}
  {\bf w}_m(\xi) T^{p^m k_{I_0}} (p^{-{{\rm ord}}_p(p^m k_{I_1})}dT^{p^m k_{I_1}}) \cdots (p^{-{{\rm ord}}_p(p^m k_{I_l})}dT^{p^m k_{I_l}}) & \left( \begin{array}{l}
I_0 \neq \emptyset \\
\mbox{or} ~ k ~\mbox{is integral}
\end{array} \right ), \\
   {\bf w}_{m-u}(\eta)(p^{-{{\rm ord}}_p(p^m k_{I_1})}dT^{p^m k_{I_1}}) \cdots (p^{-{{\rm ord}}_p(p^m k_{I_l})}dT^{p^m k_{I_l}})  & (I_0 = \emptyset). 
\end{array} 
\right. \]
Here, $u = u(k_{I}), I = {\rm supp}\,k, {}^{V^u}\eta =\xi$ and we put 
\[ p^{-{{\rm ord}}_p(p^m k_{I})}dT^{p^m k_{I}} :=T^{(p^m - p^{m-{{\rm ord}}_p(p^m k_{I}) })k_{I} } dT^{p^m k_{I_1} p^{-{{\rm ord}}_p(p^m k_{I})}}
.\]  
\end{lemma}

Also, the image of $d\log X_i$ by $\omega_m$ is calculated as follows: 
 \begin{align*}
\omega_m(d\log X_i)
 & =  \omega_m(X_i^{-1} d X_i) =  {\bf w}_m(X_i^{-1}) \delta(X_i) 
= T_i^{-p^{m}} T_i^{p^m-1}dT = d \log T_i. 
\end{align*} 
Thus we have  
\[ \omega_m (\hat{e}(\xi, k, \mathcal{P}) ) =  \omega_m(e(\xi , k^{+}, \mathcal{P}))\bigwedge_{i \in I_{- \infty} }d\log T_i \]
in general case. 

For a map 
$k: [1,  r] \to \mathbb{Z} \sqcup \{ p^{-\infty} \}$ (which we can regard as a weight if we embed the 
target $\mathbb{Z} \sqcup \{ p^{-\infty} \}$ into $\mathbb{Z}[\frac{1}{p}] \sqcup \{ p^{-\infty} \}$) 
and a partition $\mathcal{P} = (I_{-\infty}, I_0, \cdots , I_l)$ of ${\rm supp}\,k$, we define
the element $\bar{e}(k, \mathcal{P}) \in \Omega_{R[T_{1}^{\pm1}, \ldots T_{r}^{\pm1}]/R}^{l+|I_{-\infty}|}$ 
by 
\[ \bar{e}(k, \mathcal{P}) = T^{k_{I_0}} (p^{-{{\rm ord}}_p (k_{I_1})}dT^{k_{I_1}}) \cdots (p^{-{{\rm ord}}_p(k_{I_l})}dT^{k_{I_l}}) \bigwedge _{ i \in I_{-\infty} }d \log T_i \]
and call it a $p$-basic element. Then we have the following: 

\begin{lemma}\label{thm33}
The $p$-basic elements with all possible $(k, \mathcal{P})$'s form a basis of 
$\Omega_{R[T_{1}^{\pm1}, \ldots T_{r}^{\pm1}]/R}^{\bullet}$ as $R$-module. 
\end{lemma}

\begin{proof} (cf. \cite[Prop.2.1]{LZ}) 
$\Omega_{R[T_{1}^{\pm1}, \ldots T_{r}^{\pm1}]/R}^{l}$ is freely generated by the set 
\[ \mathcal{A} := \{ T^k \bigwedge_{i \in J} d \log T_i \, ; k:[1,  r] \to \mathbb{Z}, J \subset [1,  r], |J|=l\} \]
as $R$-module. If we fix a map $k: [1, r] \to \mathbb{Z}$ and put 
$\mathcal{A}_k := \{ T^k \bigwedge_{i \in J} d \log T_i \, ; J \subset [1,  r], |J|=l\}$, 
$|\mathcal{A}_k| = \binom{r}{l}$. We denote by $R{\mathcal A}_k$ the $R$-submodule of 
$\Omega^l_{R[T_1^{\pm 1},  \dots,  T_r^{\pm 1}]/R}$ generated by ${\mathcal A}_k$. 

On the other hand, for a weight $k':[1, r] \to \mathbb{Z} \sqcup \{p^{-\infty}\}$ and a partition 
${\mathcal P} = (I_{-\infty},  I_0,  \dots,  I_{l'})$ of ${\rm supp}\, k'$,  
\begin{align}\label{9}
\overline{e}(k', {\mathcal P}) & = 
T^{k'_{I_0}}(p^{-{\rm ord}_p(k'_{I_1})}dT^{k'_{I_1}}) \cdots 
(p^{-{\rm ord}_p(k'_{I_{l'}})}dT^{k'_{I_{l'}}}) 
\bigwedge_{i \in I_{-\infty}} d\log T_i \nonumber \\ 
& = \prod_{i=0}^{l'} T^{k'_{I_i}} 
\bigwedge_{i=1}^{l'} d\log T^{k'_{I_i} p^{-{\rm ord}_p(k'_{I_i})}}
\bigwedge_{i \in I_{-\infty}} d\log T_i.  % \tag{2.3.5.1}
\end{align}
This element belongs to $R{\mathcal A}_{k''}$ for some $k'':[0, 1] \to \mathbb{Z}$. 
For the fixed $k: [1, r] \to \mathbb{Z}$, we calculate the number of pairs $(k',  {\mathcal P})$ such that 
$\overline{e}(k', {\mathcal P})$ belongs to $R{\mathcal A}_k$. 

The element \eqref{9} belongs to $R{\mathcal A}_k$ if and only if 
$(k')^+ = k$ and $|I_{-\infty}| + l' = l$. (In particular, $\alpha := {\rm supp}\, (k')^+$ is uniquely determined.)  
Thus we can calculate the number of pairs $(k',  {\mathcal P})$ in question 
% with $\overline{e}(k', {\mathcal P}) \in R{\mathcal A}_k$ 
in the following way: 
\begin{enumerate}

\item[(a)] First we fix $\beta := |I_{-\infty}|$. We calculate the number of subsets $I_{-\infty}$ of order $\beta$ in 
the set $[1, r] \setminus {\rm supp} \,  k$ and the number of partitions 
$(I_0,  \dots,  I_{l'})$ of ${\rm supp}\,k$ with $l' := l - \beta$, and multiply them. 

\item[(b)] We take a sum of the numbers calculated in (a) for all possible $\beta$'s. 

\end{enumerate}
In (a), the number of subsets $I_{-\infty}$ is $\binom{r - \alpha}{\beta}$. 
Also, since a partition $(I_0,  \dots,  I_{l'})$ is uniquely determined by fixing the smallest elements of 
$I_1,  \dots,  I_{l'}$, the number of partitions $(I_0,  \dots,  I_{l'})$ is equal to $\binom{\alpha}{l-\beta}$. 
Thus, by (b), the number of pairs in question is equal to $\sum_{\beta}\binom{r - \alpha}{\beta}\binom{\alpha}{l-\beta}$ 
and it is equal to $\binom{r}{l}$. Therefore, it suffices to prove that the $p$-basic elements in $R{\mathcal A}_k$
generates $R{\mathcal A}_k$ as $R$-module. To do so, it suffices to prove that any element of the form 
$T^k \wedge_{i \in J} d\log T_i$ is written as an $R$-linear combination of $p$-basic elements. 

By renumbering the indices, we may assume that 
${\rm supp}\, k = [1, s]$. Then we have 
$T^k \bigwedge_{i \in J} d\log T_i \allowbreak 
= (T^k \bigwedge_{i \in J \cap [1, s]} d\log T_i) \bigwedge_{i \in J \setminus [1, s]} d\log T_i$. 
If we prove that $T^k \bigwedge_{i \in J \cap [1, s]} d\log T_i$ is written as an $R$-linear combination of 
$p$-basic elements in $\Omega^{\bullet}_{R[T_1^{\pm 1},  \dots,  T_s^{\pm 1}]/R}$, 
we can prove that $T^k \wedge_{i \in J} d\log T_i$ is written as an $R$-linear combination of $p$-basic elements 
in $\Omega^{\bullet}_{R[T_1^{\pm 1},  \dots,  T_r^{\pm 1}]/R}$, by multiplying with 
$\bigwedge_{i \in J \setminus [1, s]} d\log T_i$. 
Thus we may assume that 
$J \subseteq {\rm supp}\, k$ to prove the claim in the previous paragraph. 

We define the morphism 
$\varphi: \Omega_{R[U_1, \ldots U_{r}]/R}^{l} \to \Omega_{R[T_{1}^{\pm1}, \ldots T_{r}^{\pm 1}]/R}^{l}$ 
to be the one induced by the ring homomorphism $R[U_1, \ldots U_{r}] \to R[T_{1}^{\pm1}, \ldots T_{r}^{\pm 1}]$ 
defined as follows: 
\[\begin{array}{ll}
  U_{i} \mapsto T_{i} & (k_i \geq 0 ), \\
  U_{i} \mapsto T_{i}^{-1}  & (k_i \leq -1). 
\end{array} 
\]
Then the element $T^k \bigwedge_{i \in J} d \log T_i$ is in the image of $\varphi$. 
Since any element in $\Omega_{R[U_1, \ldots U_{r}]/R}^{l}$ can be written as an $R$-linear combination of 
$p$-basic elements by \cite[Prop.~2.1]{LZ} and $p$-basic elements in $\Omega_{R[U_1, \ldots U_{r}]/R}^{l}$ 
are sent by $\varphi$ to $p$-basic elements in $\Omega_{R[T_{1}^{\pm1}, \ldots T_{r}^{\pm 1}]/R}^{l}$, 
we conclude that 
$T^{k} \bigwedge_{i \in J}d \log T_i$
is written as an $R$-linear combination of 
$p$-basic elements in $\Omega_{R[T_{1}^{\pm1}, \ldots T_{r}^{\pm 1}]/R}^{l}$. 
So we are done. 
\end{proof}

Now we can prove the following uniqueness result.

\begin{prop}\label{thm20}
Any element $\omega$ in 
$W_m\Omega_{R[\mathbb{Z}^r]/R}^{\bullet}$ is written uniquely as the finite sum 
\[
\omega = \sum_{k, \mathcal{P}} \hat{e} (\xi_{k, \mathcal{P}} , k , \mathcal{P} ) \qquad (\xi_{k, \mathcal{P}} 
\in {}^{V^{u(k)}}W_{m-u(k)}(R)),  \]
where $k$ runs through weights with $p^{m-1}k$ integral. 

Also, any element $\omega$ in $W\Omega_{R[\mathbb{Z}^r]/R}^{\bullet}$
is written uniquely as the convergent sum 
\[ \omega = \sum_{k, \mathcal{P}} \hat{e} (\xi_{k, \mathcal{P}} , k , \mathcal{P}) \]
of basic Witt differentials. Here the convergence is defined in the same way as that in 
Proposition \ref{thm16}. 
\end{prop}

\begin{proof}(cf. \cite[Prop. 2.17]{LZ}) 
By Proposition \ref{thm1}, it suffices to prove the uniqueness. Also, it suffices to prove the former assertion. 
Suppose that we have an equality in $W_m\Omega_{R[\mathbb{Z}^r]/R}^{\bullet}$ 
\[0 = \sum_{k, \mathcal{P}} \hat{e} (\xi_{k, \mathcal{P}} , k , \mathcal{P} ).\]
We consider the images of it by the maps $\omega_0, \ldots, \omega_{m-1}$ which we recalled in 
the end of Section 1.1. By Lemma \ref{thm18} and the paragraph after it, the images of 
basic Witt diffrentials are $p$-basic elements and by Lemma \ref{thm33}, 
they form a basis of $\Omega_{R[T_{1}^{\pm1}, \ldots T_{r}^{\pm1}]/R}^{l}$
as $R$-module. Thus we see that ${\bf w}_{i-u(k)}( \xi_{k, \mathcal{P}}) = 0$ for $u(k) \leq i \leq m-1$. 
When $R$ is $p$-torsion free, this implies that $\xi_{k, \mathcal{P}} = 0$. Thus we are done in 
the case $R$ is $p$-torsion free. 

When $R$ is not necessarily $p$-torsion free, we prove the proposition by the same argument as 
the latter half of the proof in \cite[Prop.~2.17]{LZ}, which we explain now. Take a $p$-torsion free 
${\mathbb Z}_{(p)}$-algebra 
$\tilde{R}$ with $R = \tilde{R}/\mathfrak{a}$ for some ideal $\mathfrak{a}$. Also, let 
\[ W_m\Omega_{\mathfrak{a}\tilde{R}[\mathbb{Z}^r]/\tilde{R}}^{\bullet} \subset W_m\Omega_{\tilde{R}[\mathbb{Z}^r]/\tilde{R}}^{\bullet}\] 
be the set of elements of the form 
\[ \sum_{\xi, k ,  \mathcal{P}} \hat{e}(\xi_{k,  \mathcal{P}} , k , \mathcal{P}) \]
with each $\xi_{k, \mathcal{P}}$ in $W_{m-u(k)}(\mathfrak{a})$. 
We see that $\{ W_m\Omega_{\mathfrak{a}\tilde{R}[\mathbb{Z}^r]/\tilde{R}}^{\bullet} \}_m$ is stable by $F, V, d$. 
We consider the quotient 
\[E_m^{\bullet} := W_m\Omega_{\tilde{R}[\mathbb{Z}^r]/\tilde{R}}^{\bullet}/W_m\Omega_{\mathfrak{a}\tilde{R}[\mathbb{Z}^r]/\tilde{R}}^{\bullet}.\] 
%\[ 0 \to W_m\Omega_{\mathfrak{a}\tilde{R}[\mathbb{Z}^r]/\tilde{R}}^{\bullet} \to W_m\Omega_{\tilde{R}[\mathbb{Z}^r]/\tilde{R}}^{\bullet} \to E_m^{\bullet} \to 0 \]
By definition, $E_m^0= W_m(R[\mathbb{Z}^r])$. Since $\{ E_m^{\bullet} \}_m$
is an $F$-$V$-procomplex over the $R$-algebra $R[\mathbb{Z}^r]$, we have the unique morphism 
$\{ W_m \Omega_{R[\mathbb{Z}^r]/R}^{\bullet} \}_m \to \{ E_m^{\bullet} \}_m$ which makes 
the following diagram commutative: 
\[ \xymatrix{
  W_m \Omega_{R[\mathbb{Z}^r]/R}^{\bullet} \ar[r]%^(.6){\exists}
&  E_m^{\bullet}. \\
  W_m \Omega_{\tilde{R}[\mathbb{Z}^r]/\tilde{R}}^{\bullet} \ar[u] \ar[ru]
  }
\]
Now suppose that we have an equality 
\[ 0 = \sum_{k, \mathcal{P}} \hat{e} (\xi_{k, \mathcal{P}} , k , \mathcal{P}) \]
in $W_m\Omega_{R[\mathbb{Z}^r]/R}^{\bullet}$. We obtain a lifting of the terms on the right hand side 
by lifting each coefficient $\xi_{k,  \mathcal{P}}$ to some $\tilde{\xi}_{k,  \mathcal{P}} \in W_{m-u(k)}(\tilde{R})$. 
If we send these lifted terms to $E_m^{\bullet}$, their sum is zero. 
Since $\tilde{R}$ is $p$-torsion free, every element in 
$W_m\Omega_{\tilde{R}[\mathbb{Z}^r]/\tilde{R}}^{\bullet}$ is written uniquely as 
a sum of basic Witt differentials as in the statement of the proposition, we conclude that 
$\tilde{\xi}_{k,  \mathcal{P}} \in W_{m - u(k)}(\mathfrak{a})$
for all $k, \mathcal{P}$. Hence $\xi_{k,  \mathcal{P}} = 0$ for all $k, \mathcal{P}$ and so we are done. 
\end{proof}

By Proposition \ref{thm20}, we have the following immediate corollary. 

\begin{cor}\label{thm50}
Any element in the kernel of the map 
$W\Omega_{R[\mathbb{Z}^r]/R}^{\bullet} \to W_m\Omega_{R[\mathbb{Z}^r]/R}^{\bullet}$ is written uniquely as 
a convergent sum of basic Witt differentials $\hat{e} (\xi,  k, \mathcal{P})$ with $  \xi \in {}^{V^m} W(R)$. 
\end{cor}

\subsection{Decomposition of relative log de Rham-Witt complex (I)}

Let $R$ be a ${\mathbb Z}_{(p)}$-algebra  
% in which $p$ is nilpotent 
and let $P$ be an fs monoid 
with $P^{\rm gp} \cong {\mathbb{Z}}^r$ torsion free. 
In this subsection, we introduce the notion of $x$-part 
$W_m \Lambda_{(R[P], P) /(R, *),x}^{\bullet}$
of the relative log de Rham-Witt complex $W_m \Lambda_{(R[P], P) /(R, *)}^{\bullet}$ 
for $x \in P[\frac{1}{p}]$ and prove the 
direct sum decomposition 
\begin{equation*}\label{eq:decompdecomp}
W_m\Lambda^{\bullet}_{(R[P],P)/(R,*)} 
= \bigoplus_{x \in P[\frac{1}{p}]} W_m\Lambda^{\bullet}_{(R[P],P)/(R,*), x}. 
\end{equation*}

\begin{definition}\label{thm21}
Let $R, P$ be as above. For $x \in P[\frac{1}{p}]$,  
we denote by $W\Lambda_{(R[P],  P)/(R,  {*}), x}^{n}$ the set of elements 
$\omega$ in $W\Lambda_{(R[P],  P)/(R,  {*})}^{n}$ which is written in the form 
\begin{equation}\label{eq:xpartelt}
\omega =  \sum_{ I \subset [0, r], |I|=n} b(\xi_I, x, I).
\end{equation}
%\[ \omega =  \sum_{ I \subset [1, r], |I|=n} {}^{V^{u(x)}}({\eta}_{x, I}X^{p^{u(x)}x} )\bigwedge_{I} d\log X_i + \sum_{ J \subset [1,  r], |J|=n-1} d^{V^{u(x_J)}}({\eta}_{x_J}X^{p^{u(x_J)}x_J} )\bigwedge_{J} d\log X_i \]
Also, we put 
\[ W_m\Lambda_{(R[P],  P)/(R,  {*}), x}^{n} := \mbox{Im}(W\Lambda_{(R[P],  P)/(R,  {*}), x}^{n} \to 
W_m\Lambda_{(R[P],  P)/(R,  {*})}^{n}). \]
We call them as the $x$-part of $W\Lambda_{(R[P],  P)/(R,  {*})}^{n} , W_m\Lambda_{(R[P],  P)/(R,  {*})}^{n}$, respectively. 
\end{definition}

By definition, the truncation map $W\Lambda_{(R[P],  P)/(R,  {*})}^{n} \to W_m\Lambda_{(R[P],  P)/(R,  {*})}^{n}$
induces the morphism $W\Lambda_{(R[P],  P)/(R,  {*}),x}^{n} \to W_m\Lambda_{(R[P],  P)/(R,  {*}),x}^{n}$
between $x$-parts. 

\begin{lemma}\label{thm382}
The derivation $d$ preserves the $x$-part.
\end{lemma}

\begin{proof}
If $\omega$ is of the form \eqref{eq:xpartelt}, 
$d\omega =  \sum_{ 0 \notin I \subset [1, r], |I|=n} b(\xi_I, x, I \cup \{ 0 \} ).$
\end{proof}
%\[ d\omega =  \sum_{ I \subset [1,  r], |I|=n} d{}^{V^{u(x)}}({\eta}_{x, I}X^{p^{u(x)}x} )\bigwedge_{I} d\log X_i \]

By Definition \ref{thm21}, the notion of $x$-part 
$W\Lambda_{(R[P^{\rm gp}],  P^{\rm gp})/(R,  {*}),  x}^{\bullet}$ of 
$W\Lambda_{(R[P^{\rm gp}],  P^{\rm gp})/(R,  {*})}^{\bullet}$ is defined for 
$x \in P^{\rm gp}[\frac{1}{p}] \cong {\mathbb{Z}}[\frac{1}{p}]^r$. 
We prove the relation of it with the basic Witt differentials defined in the previous subsection. 

\begin{lemma}\label{thm2}
Let $P$ be as above and let $x \in P^{\rm gp}[\frac{1}{p}] \cong {\mathbb{Z}[\frac{1}{p}]}^r$. Then 
any element $\omega$ in $W\Lambda_{(R[P^{\rm gp}],  P^{\rm gp})/(R,  {*}),  x}^{\bullet}$
is written uniquely as a sum of basic Witt differentials whose weight $k$ satisfies 
the equality $k^{+} = x$. Namely, $\omega$ is written uniquely in the following form: 
\[ \omega = \sum_{k^{+} = x, \mathcal{P}} \hat{e}(\xi_{k,  \mathcal{P}}, k, \mathcal{P}). \]
\end{lemma}

\begin{proof}
By Proposition \ref{thm20}, it suffices to prove the existence of the required expression. 
Since any element in $\omega \in W\Lambda_{(R[P^{\rm gp}], P^{\rm gp})/(R,  {*}),  x}^{\bullet}$
is written as the sum  
%\[ \omega =  \sum_{ I \subset [1,  r], |I|=n} {}^{V^{u(x)}}({\eta}_{x, I}X^{p^{u(x)}x} )\bigwedge_{I} d\log X_i + \sum_{ J \subset [1,  r], |J|=n-1} d^{V^{u(x_J)}}({\eta}_{x_J}X^{p^{u(x_J)}x_J} )\bigwedge_{J} d\log X_i \]
\[ \omega =  \sum_{ I \subset [0, r], |I|=n} b(\xi_I, x, I)\]
of $P^{\rm gp}$-basic Witt differentials, 
it suffices to prove the lemma for a $P^{\rm gp}$-basic Witt differential $b(\xi, x, I)$. 
In fact, we prove the following slightly stronger claim: 
\medskip 

\noindent
{\bf claim 1.} Any $P^{\rm gp}$-basic Witt differential $b(\xi, x, I)$ is written 
(uniquely) in the form  
\[ b(\xi, x, I) = \sum_{k^{+} = x, \mathcal{P}} \hat{e}(\xi_{k,  \mathcal{P}}, k, \mathcal{P}) \]
with each $\xi_{k, \mathcal{P}}$ in the set ${\mathbb{Z}}_{(p)} \xi$. 
\medskip 

If $0 \in I$, we have the equality $b(\xi_I, x , I) = db(\xi_I, x , I \setminus \{ 0 \} )$. 
Then we see by Lemma \ref{thm4} that, 
if claim 1 holds for the element $b(\xi_I, x , I \setminus \{ 0 \} )$, it holds also for 
the element $b(\xi_I, x , I)$. 
Thus we may assume that $0 \notin I$ to prove claim 1. 

Next we consider the following claim: 
\medskip 

\noindent
{\bf claim 2.} Let $\hat{e}(\xi, k, \mathcal{P})$ be a basic Witt differential with $k^{+} =x$. 
Then, for each $i \in [1,  r]$, $\hat{e}(\xi, k, \mathcal{P})d \log X_i$
is written in the form 
\[ \hat{e}(\xi, k, \mathcal{P})d \log X_i = \sum_{k^{+} = x, \mathcal{P}} \hat{e}(\xi_{k,  \mathcal{P}}, k, \mathcal{P}) \]
with each $\xi_{k, \mathcal{P}}$ in the set ${\mathbb{Z}}_{(p)} \xi$.
%as a sum of basic Witt differentials with $k^{+} =x$ and with 
%coefficients in $\mathbb{Z}_{(p)} \xi$. 
\medskip 

Then claim 2 implies claim 1: 
%that any element of the form $b(\xi_I, x, I)$ with $0 \notin I$ 
%is written as a sum of basic Witt differentials with $k^{+} =x$, thus implies the lemma: 
Indeed, if $|I|=0$ (hence $I = \emptyset$), $b(\xi, x, \emptyset) = \hat{e}(\xi, k, ({\rm supp}\,k))$ with $k = x$. 
If $|I| > 0$ with $0 \not= i \in I$, we have the equality 
$b(\xi, x, I) = b( \pm \xi, x, I \setminus \{i \}) d \log X_i$ and so 
claim 2 implies that the induction on $|I|$ works. 

So it suffices to prove claim 2. So let $\hat{e}(\xi, k, \mathcal{P})$ 
be a basic Witt differential with $k^{+} =x$, let $i \in [1,  r]$ and put $\mathcal{P} := (I_{-\infty}, I_0, \ldots , I_n)$. 

\begin{enumerate}[(1)]

\item The case $k_i = p^{-\infty}$. 

In this case, $\hat{e}(\xi, k, \mathcal{P})d \log X_i = 0$. 

\item The case $k_i = 0$. 

In this case, $\hat{e}(\xi, k, {\mathcal P}) d\log X_i = \hat{e}((-1)^a\xi, k', {\mathcal P}')$, where 
\begin{align*}
& k'_j = 
\begin{cases}
p^{-\infty} & (j=i), \\ 
k_j & (j \not= i),
\end{cases} \\ 
& {\mathcal P}' = (I'_{-\infty}, I'_0, \dots, I'_{n}), 
\quad 
I'_j = 
\begin{cases}
I_{-\infty} \cup \{i\} & (j=-\infty), \\ 
I_j & (0 \leq j \leq n),
\end{cases} \\ 
& \bigwedge_{j \in I_{-\infty}} d\log X_j \bigwedge d\log X_i = 
(-1)^a \bigwedge_{j \in I'_{-\infty}} d\log X_j. 
\end{align*}
Thus $k'^{+} =x$. 

\item The case $i \in I_a \, (0 \leq a \leq r)$. 

We denote the term in $\hat{e}(\xi,  k,  {\mathcal P})$ defined from 
$I_a$ by $\omega$, namely, we put 
\[ \omega = \left\{
\begin{array}{ll}
{}^{V^{u(I_0)}}(\eta X^{p^{u(I_0)}k_{I_0}}) & (a=0), \\
d^{V^{t(I_1)}}(\eta X^{p^{t(I_1)}k_{I_1}}) & (a=1,  I_0 = \emptyset,  t(I_1) \geq 1), \\ 
\eta d^{V^{t(I_1)}}(X^{p^{t(I_1)}k_{I_1}}) & (a=1,  I_0 = \emptyset,  t(I_1) \leq 0), \\ 
d^{V^{t(I_a)}}(X^{p^{t(I_a)}k_{I_a}}) & (\mbox{otherwise}),  
\end{array} \right.
\] where $\xi = {}^{V^{u(I)}}\eta,  I = {\rm supp}\, k$. 
We denote the product of the terms in $\hat{e}(\xi,  k,  {\mathcal P})$ 
which are on the left (resp.~right) of the term $\omega$ by 
$\omega_1$ (resp. $\omega_2$). (If there is no such term, we put $\omega_i := 1$.) 
Then $\hat{e}(\xi,  k,  {\mathcal P}) = \omega_1 \omega \omega_2$. 
We define the subintervals $I_{a0},  I_{a1}$ of $I_a$ by $I_{a0} := \{j \in I_a \, ;\,  j < i\},  I_{a1} := \{j \in I_a \, ;\,  j > i\}$
and put $n_i := t(\{i\})$, $l_i := p^{n_i}k_i$. 

We calculate $\omega d\log X_i$. First, when $a=0$, 
{\allowdisplaybreaks{
\begin{eqnarray*}
& \quad &  \omega d\log X_i \\ 
& = & {}^{V^{u(I_0)}}(\eta X^{p^{u(I_0)}k_{I_0}})d\log X_i \\ 
& = & {}^{V^{u(I_0)}}(\eta X^{p^{u(I_0)}k_{I_0}} d\log X_i) \\ 
& = & {}^{V^{u(I_0)}}(l_i^{-1} \eta X^{p^{u(I_{0})}k_{I_{00}}} X_i^{p^{u(I_{0})-n_i}l_i} X^{p^{u(I_{0})}k_{I_{01}}}d\log X_i^{l_i}) \\
& = & {}^{V^{u(I_0)}}(l_i^{-1} \eta X^{p^{u(I_{0})}k_{I_{00}}} X_i^{(p^{u(I_{0})-n_i}-1)l_i} X^{(p^{u(I_{0})}- p^{n_i})k_{I_{01}}}X^{p^{n_i}k_{I_{01}}}dX_i^{l_i}) \\ 
& = & {}^{V^{u(I_0)}}(l_i^{-1} \eta X^{p^{u(I_{0})}k_{I_{00}}} X_i^{(p^{u(I_{0})-n_i}-1)l_i} X^{(p^{u(I_{0})}- p^{n_i})k_{I_{01}}} (d(X^{p^{n_i}k_{I_{01}}}X_i^{l_i}) - X_i^{l_i}dX^{p^{n_i}k_{I_{01}}})) \\ 
& = & {}^{V^{u(I_0)}} (l_i^{-1} \eta X^{p^{u(I_{0})}k_{I_{00}}} ({}^{F^{u(I_{0})-n_i}}d(X_i^{l_i}X^{p^{n_i}k_{I_{01}}}) \\ 
& \quad & \hspace{4cm} - p^{n_i-t(I_{01})} X_i^{p^{u(I_{0})-n_i}l_i}{}^{F^{u(I_0)-t(I_{01})}}dX^{p^{t(I_{01})}k_{I_{01}}})) \\ 
& = & {}^{V^{u(I_0)}}(\frac{\eta}{l_i}X^{p^{u(I_{0})}k_{I_{00}}}{}^{F^{u(I_{0})-n_i}}dX^{p^{n_i}k_{(\{i\} \cup I_{01})}}) \\ 
& \quad & \hspace{4cm} - {}^{V^{u(I_0)}}(\frac{\eta p^{n_i-t(I_{01})}}{l_i}X^{p^{u(I_{0})}k_{(I_{00} \cup \{i\})}} {}^{F^{u(I_0)-t(I_{01})}}dX^{p^{t(I_{01})}k_{I_{01}}}) \\ 
& = & {}^{V^{u(I_0)}}(\frac{\eta}{l_i}X^{p^{u(I_{0})}k_{I_{00}}}) d^{V^{n_i}}(X^{p^{n_i}k_{(\{i\} \cup I_{01})}}) \\ 
& \quad & \hspace{4cm} - {}^{V^{u(I_0)}}(\frac{\eta p^{n_i-t(I_{01})}}{l_i}X^{p^{u(I_{0})}k_{(I_{00} \cup \{i\})}}) d^{V^{t(I_{01})}}X^{p^{t(I_{01})}k_{I_{01}}}. 
\end{eqnarray*}}}%
When $I_{00} = \emptyset$, the first term is rewritten as
\begin{equation*}
\begin{cases}
d^{V^{n_i}}(\frac{\eta p^{u(I)}}{l_i} X^{p^{n_i}k_{(\{i\} \cup I_{01})}}) & (\text{if $n_i \geq 1$}), \\  
\frac{\eta p^{u(I)}}{l_i} d^{V^{n_i}}(X^{p^{n_i}k_{(\{i\} \cup I_{01})}}) & (\text{if $n_i \leq 0$}),
\end{cases}
\end{equation*}
where $I = {\rm supp}\,k$. 
%\[ d^{V^{n_i}}(\frac{\eta p^{u(I)}}{l_i} X^{p^{n_i}k_{(\{i\} \cup I_{01})}}) \quad \mbox{(where $I = {\rm supp}\, k$).} \]
Also, noting the equality of the form 
\[ d^{V^{a}}(\eta X^{b}) d\log X_i = d({}^{V^{a}}(\eta X^{b} d \log X_i) ), \]
we see by a similar calculation as above that, when 
$a=1,  I_0 = \emptyset,  t(I_1) \geq 1$, 
\begin{eqnarray*}
&\quad & \omega d\log X_i = d {}^{V^{t(I_1)}}(\frac{\eta}{l_i}X^{p^{t(I_{1})}k_{I_{10}}}) d^{V^{n_i}}(X^{p^{n_i}k_{(\{i\} \cup I_{11})}}) \\ 
&\quad & \hspace{4cm} - d{}^{V^{t(I_1)}}(\frac{\eta p^{n_i-t(I_{11})}}{l_i}X^{p^{t(I_{1})}k_{(I_{10} \cup \{i\})}}) d^{V^{t(I_{11})}}X^{p^{t(I_{11})}k_{I_{11}}},  
\end{eqnarray*}
when $a=1,  I_0 = \emptyset,  t(I_1) \leq 0$, 
\begin{eqnarray*}
&\quad & \omega d\log X_i = \frac{\eta}{l_i} d {}^{V^{t(I_1)}}(X^{p^{t(I_{1})}k_{I_{10}}}) d^{V^{n_i}}(X^{p^{n_i}k_{(\{i\} \cup I_{11})}}) \\ 
&\quad & \hspace{4cm} - \frac{\eta p^{n_i-t(I_{11})}}{l_i} d{}^{V^{t(I_1)}}(X^{p^{t(I_{1})}k_{(I_{10} \cup \{i\})}}) d^{V^{t(I_{11})}}X^{p^{t(I_{11})}k_{I_{11}}} 
\end{eqnarray*}
and in the other cases, 
\begin{eqnarray*}
&\quad & \omega d\log X_i = \frac{1}{l_i} d {}^{V^{t(I_a)}}(X^{p^{t(I_{a})}k_{I_{a0}}}) d^{V^{n_i}}(X^{p^{n_i}k_{(\{i\} \cup I_{a1})}}) \\ 
&\quad & \hspace{4cm} - 
\frac{p^{n_i-t(I_{a1})}}{l_i} d{}^{V^{t(I_a)}}(X^{p^{t(I_{a})}k_{(I_{a0} \cup \{i\})}}) d^{V^{t(I_{a1})}}X^{p^{t(I_{a1})}k_{I_{a1}}}. 
\end{eqnarray*}
Thus we obtain the equality 
\[ \hat{e}(\xi, k , \mathcal{P}) d\log X_i = \hat{e}(\xi_1, k , \mathcal{P}^1) - \hat{e}(\xi_2, k , \mathcal{P}^2), \]
where \[ \xi_1 = \left\{ \begin{array}{ll}
\frac{\xi p^{u(I)}}{l_i}(-1)^{{\rm deg}\omega_2} & (a=0, I_{00} = \emptyset),  \\
\frac{\xi}{l_i}(-1)^{{\rm deg}\omega_2} & (\mbox{otherwise})
 \end{array}  \right. \]
(where $I = {\rm supp}\,k$), $\xi_2 = \frac{\xi p^{n_i - t(I_{a1})}}{l_i}(-1)^{{\rm deg}\omega_2} $, and if we put 
$\mathcal{P}^{j} := (I_{-\infty}^j, I_0^{j}, \ldots, I_{n+1}^{j}) \quad(j=1,  2)$, 
\begin{align*}
& 
I_k^1 = 
\begin{cases}
I_k & (k < a), \\ 
I_{a0} & (k=a), \\ 
\{i\} \cup I_{a1} & (k=a+1), \\ 
I_{k-1} & (k > a+1), 
\end{cases}
\qquad 
I_k^2 = 
\begin{cases}
I_k & (k < a), \\ 
I_{a0} \cup \{i\} & (k=a), \\ 
I_{a1} & (k=a+1), \\ 
I_{k-1} & (k > a+1). 
\end{cases}
\end{align*}
Here, if $I_{a0} = \emptyset$ (resp. $I_{a1} = \emptyset$), we put 
$\hat{e}(\xi_1, k , \mathcal{P}^1) = 0$ (resp. $\hat{e}(\xi_2, k , \mathcal{P}^2) = 0$). 
Since the weight of $\hat{e}(\xi_1, k , \mathcal{P}^1) , \hat{e}(\xi_2, k , \mathcal{P}^2)$ are equal to $k$, claim 2 is proved. 
\end{enumerate}% 
\end{proof}

For $x \in P[\frac{1}{p}] (\subset P^{\rm gp}[\frac{1}{p}] \cong {\mathbb Z}[\frac{1}{p}]^r)$, 
%. and denote its image in 
%$P^{\rm gp}[\frac{1}{p}] \cong \mathbb{Z}[\frac{1}{p}]^r$ by $(x_i)_i$. 
we will define the following morphisms: 
\[ f: W\Lambda_{(R[P],  P)/(R,  {*}), x}^{\bullet} \to W\Lambda_{(R[P^{\rm gp}],  P^{\rm gp})/(R,  {*}), x}^{\bullet}, \]
 \[ g: W\Lambda_{(R[P^{\rm gp}],  P^{\rm gp})/(R,  {*}), x}^{\bullet} \to W\Lambda_{(R[P],  P)/(R,  {*}), x}^{\bullet}. \]
The morphism $f$ is defined to be the morphism between $x$-parts of the morphism of 
relative log de Rham-Witt complexes 
\begin{equation}\label{eq:h}
h: W\Lambda_{(R[P],  P)/(R,  {*})}^{\bullet} \to W\Lambda_{(R[P^{\rm gp}],  P^{\rm gp})/(R,  {*})}^{\bullet} 
\end{equation}
induced by the monoid homomorphism $P \to P^{\rm gp}$. 

On the other hand, by Lemma \ref{thm2}, any element in 
$W\Lambda_{(R[P^{\rm gp}],  P^{\rm gp})/(R,  {*}), x}^{\bullet}$ is written uniquely as a convergent sum of 
basic Witt differentials $\hat{e}(\xi, k ,  \mathcal{P})$ with $k^{+} = x$. 
Thus the morphism $g$ is defined if we define the image of each $ \hat{e}(\xi, k ,  \mathcal{P})$. 
Put $\mathcal{P}= (I_{-\infty}, I_0, \ldots,  I_l)$. Then, when 
$I_0 \neq \emptyset$, we have the equality 
\begin{eqnarray*}
\hat{e}(\xi, k ,  \mathcal{P})
& = & ^{V^{u(I_0)}}(\eta_1 X^{p^{u(I_0)}k_{I_0}}) d^{V^{t(I_1)}}(X^{p^{t(I_1)}k_{I_1}})\cdots \\
&= & \sum_{i \in I_1} {}^{V^{u(I_0)}}(\eta k_i p^{t(I_1)} X^{p^{u(I_0)}k_{I_0 \cup I_1}} )d\log X_i \cdots,  
\end{eqnarray*}
where, when $t(I) < 0$, we put $d^{V^{t(I)}}(X^{x}) : = {}^{F^{-t(I)}}dX^{x}$. 
(We omitted to write the terms defined by $I_j$ for $j \geq 2$.) 
This calculation is valid also in the case 
$I_0 = \emptyset,  u(I_1) = 0$, by our convention that $u(\emptyset) =0$. 

By continuing this calculation to the terms defined by  
$I_2, \dots, I_l$, we obtain the following equality when $I_0 \neq \emptyset$ or $I_0 = \emptyset, u(I_1)=0$: 
\begin{align} \label{1}
\hat{e}(\xi, k ,  \mathcal{P}) = \sum_{i_j \in  I_j , \atop j=1,  2, \ldots, l} {}^{V^{u(I_0)}} \left( \eta \prod_{j=1}^{l} k_j p^{t(I_j)} X^{p^{u(I_0)}k^{+}} \right) \bigwedge_{j=1}^{l} d\log X_{i_j} \bigwedge_{j \in I_{-\infty}} d\log X_j. % \tag{2.4.1}.
\end{align}
Also, by the equality 
\begin{align} \label{2}
d\hat{e}(\xi, k ,  \mathcal{P}) = \sum_{i_j \in  I_j , \atop j=1,  2, \ldots, l} d^{V^{u(I_0)}} \left( \eta \prod_{j=1}^{l} k_j p^{t(I_j)} X^{p^{u(I_0)}k^{+}} \right) \bigwedge_{j=1}^{l} d\log X_{i_j} \bigwedge_{j \in I_{-\infty}} d\log X_j % \tag{2.4.2}
\end{align}
which we obtain by applying $d$ to \eqref{1} and by Lemma \ref{thm4}, 
we obtain the following equality when 
$I_0 = \emptyset , u(I_1) \geq 1$: 
\begin{align} \label{3}
\hat{e}(\xi, k ,  \mathcal{P}) = \sum_{i_j \in  I_j , \atop j=2, \ldots, l} d^{V^{u(I_1)}} \left( \eta \prod_{j=2}^{l} k_{i_j} p^{t(I_j)} X^{p^{u(I_0)}k^+} \right) \bigwedge_{j=2}^{l} d\log X_{i_j} \bigwedge_{j \in I_{-\infty}} d\log X_j.  % \tag{2.4.3}
\end{align}

We define the image of the element $\hat{e}(\xi, k ,  \mathcal{P})$ by $g$ as the element 
we obtain by replacing $p^{u(I_0)}k^{+}$ by $p^{u(x)}x$ in the expression on the right hand side of \eqref{1}, \eqref{3}. 
Namely, when $I_0 \neq \emptyset$, 
\begin{align} \label{4}
\hat{e}(\xi, k ,  \mathcal{P}) \mapsto \sum_{i_j \in  I_j , \atop j=1,  2, \ldots, l} {}^{V^{u(I_0)}} \left( \eta \prod_{j=1}^{l} k_j p^{t(I_j)} X^{p^{u(x)}x} \right) \bigwedge_{j=1}^{l} d\log X_{i_j} \bigwedge_{j \in I_{-\infty}} d\log X_j , % \tag{2.4.4}
\end{align} 
when $I_0 = \emptyset , u(I_1) \geq 1$, 
\begin{align} \label{5}
\hat{e}(\xi, k ,  \mathcal{P}) \mapsto \sum_{i_j \in  I_j , \atop j=2, \ldots, l} d^{V^{u(I_1)}} \left( \eta \prod_{j=2}^{l} k_{i_j} p^{t(I_j)} X^{p^{u(x)}x} \right) \bigwedge_{j=2}^{l} d\log X_{i_j} \bigwedge_{j \in I_{-\infty}} d\log X_j , % \tag{2.4.5}
\end{align}
and when $I_0 = \emptyset, u(I_1) = 0$, 
\begin{align} \label{6}
\hat{e}(\xi, k ,  \mathcal{P}) \mapsto \sum_{i_j \in I_j , \atop j=1,  2, \ldots, l} \left( \eta \prod_{j=1}^{l} k_{i_j} p^{t(I_j)} X^{p^{u(x)}x} \right) \bigwedge_{j=1}^{l} d\log X_{i_j} \bigwedge_{j \in I_{-\infty}} d\log X_j. % \tag{2.4.6}
\end{align}

\begin{prop}\label{thm7}
The morphisms $f,  g$ are isomorphisms.  
\end{prop}

\begin{proof}
We prove the proposition in 6 steps. 
\medskip 

\noindent
{\bf Step 1.} For any $\omega \in W\Lambda_{(R[P],  P)/(R,  {*}), x}^{\bullet}$, $f(d\omega) = df(\omega)$. 

In fact, the morphism $W\Lambda_{(R[P],  P)/(R,  {*})}^{\bullet} \to W\Lambda_{(R[P^{\rm gp}],  P^{\rm gp})/(R,  {*})}^{\bullet}$
which we defined by the universality respects $d$, and so the morphism between $x$-parts of them 
also respects $d$. 
\medskip 

\noindent 
{\bf Step 2.} For any $\omega \in W\Lambda_{(R[P^{\rm gp}],  P^{\rm gp})/(R,  {*}), x}^{\bullet}$, 
$g(d\omega) = dg(\omega)$. 

To show this, we may assume that $\omega = \hat{e}(\xi, k , \mathcal{P})$ 
is a basic Witt differential with $k^+ = x$. Put 
$\mathcal{P} = (I_{-\infty}, I_0, \ldots, I_n), I = {\rm supp}\,k$. 

\begin{enumerate}[(I)]

\item The case $I_0 = \emptyset$. 

In this case, $d\omega =0$. So we will prove that $dg(\omega) =0$. 

\begin{enumerate}[(i)]

\item The case $u(I_1) \geq 1$. 

$g(\omega)$ is defined by \eqref{5}, and it is clear that $dg(\omega) =0$ in this case. 

\item The case $u(I_1) = 0$. 

$g(\omega)$ is defined by \eqref{6}. Then 
\begin{eqnarray*}
dg(\hat{e}(\xi, k ,  \mathcal{P})) 
& = & \sum_{i_j \in I_j , \atop j=1,  2, \ldots, l} d\left( \xi \prod_{j=1}^{l} k_{i_j} p^{t(I_j)} X^{p^{u(x)}x} \right) \bigwedge_{j=1}^{l} d\log X_{i_j} \bigwedge_{j \in I_{-\infty}} d\log X_j \\
& = & d(\xi X^{p^{u(x)}x}) \bigwedge_{j=1}^{l} d \log X^{p^{t(I_j)}k_{I_j}} \bigwedge_{j \in I_{-\infty}} d\log X_j \\
& = & \xi X^{p^{u(x)}x} \left( \sum_{j=1}^{l} p^{u(x) - t(I_j)} d \log X^{p^{t(I_j)}k_{I_j}} \right) \bigwedge_{j=1}^{l} d \log X^{p^{t(I_j)}k_{I_j}} \bigwedge_{j \in I_{-\infty}} d\log X_j \\
& = & 0. 
\end{eqnarray*}
Thus $dg(\omega) =0$. 

\end{enumerate}

\item The case $ I_0 \neq \emptyset$. 

$g(\omega)$ is defined by \eqref{4}. 

\begin{enumerate}[(i)]

\item The case $u(I) \geq 1$. 

$g(d\omega)$ is defined by \eqref{5}: Indeed, by Lemma \ref{thm4},  
\[d\hat{e}(\xi, k , \mathcal{P}) = \hat{e}(\xi, k, (I_{-\infty}, \emptyset, I_0, \ldots, I_n)).\]
In this case, it is clear that 
$dg(\omega) = g(d\omega)$. 

\item The case $u(I) = 0$. 

$g(d\omega)$ is defined by \eqref{6}: Indeed, by Lemma \ref{thm4}, 
\[ d\hat{e}(\xi, k , \mathcal{P}) = \hat{e}(p^{-t(I)}\xi, k, (I_{-\infty}, \emptyset, I_0, \ldots, I_n)). \]
Hence 
{\allowdisplaybreaks{
\begin{eqnarray*}
dg(\hat{e}(\xi, k ,  \mathcal{P})) 
& = & \sum_{i_j \in I_j , \atop j=1,  2, \ldots, l} d\left( \xi \prod_{j=1}^{l} k_{i_j} p^{t(I_j)} X^{p^{u(x)}x} \right) \bigwedge_{j=1}^{l} d\log X_{i_j} \bigwedge_{j \in I_{-\infty}} d\log X_j \\
& = & d(\xi X^{p^{u(x)}x}) \bigwedge_{j=1}^{l} d \log X^{p^{t(I_j)}k_{I_j}} \bigwedge_{j \in I_{-\infty}} d\log X_j \\
& = & \xi X^{p^{u(x)}x} \left( \sum_{j=0}^{l} d \log X^{p^{u(I_0)}k_{I_j}} \right) \bigwedge_{j=1}^{l} d \log X^{p^{t(I_j)}k_{I_j}} \bigwedge_{j \in I_{-\infty}} d\log X_j \\
& = & \xi X^{p^{u(x)}x} \left(d \log X^{p^{u(I_0)}k_{I_0}} \right) \bigwedge_{j=1}^{l} d \log X^{p^{t(I_j)}k_{I_j}} \bigwedge_{j \in I_{-\infty}} d\log X_j \\
& = & g(d\omega)
\end{eqnarray*}}}%
and so $g(d\omega) = dg(\omega) $. 

\end{enumerate}

\end{enumerate}

\noindent
{\bf Step 3.} For any $\omega \in W\Lambda_{(R[P],  P)/(R,  {*}), x}^{\bullet}$ and any $i \in [1,  r]$, 
$f(\omega \bigwedge d\log X_i) = f(\omega) \bigwedge d \log X_i$. 

In fact, $d\log X_i$ is defined as the image of the $i$-th standard basis by the map 
$d\log: {\mathbb{Z}}^r \cong P^{\rm gp} \to W\Lambda_{(R[P],  P)/(R,  {*})}^{1}$. 
Hence the morphism 
$W\Lambda_{(R[P],  P)/(R,  {*})}^{1} \to W\Lambda_{(R[P^{\rm gp}],  P)/(R,  {*})}^{1}$ 
defined by the universality sends 
$d \log X_i$ to $d \log X_i$. 
\medskip 

\noindent 
{\bf Step 4.} For any $\omega \in W\Lambda_{(R[P^{\rm gp}],  P)/(R,  {*}), x}^{\bullet}$ and any $i \in [1,  r]$, 
$g(\omega \bigwedge d\log X_i) = 
g(\omega) \bigwedge d\log X_i$. 

To show this,
we may assume that $\omega = \hat{e}(\xi, k, {\mathcal P}) \in W\Lambda^n_{(R[P^{\rm gp}],  P^{\rm gp})/(R,  *), x}$
is a basic Witt differential with $k^+ = x$, and that the claim holds for elements in 
$W\Lambda^{n-1}_{(R[P^{\rm gp}],  P^{\rm gp})/(R,  *), x}$. 
We put ${\mathcal P} = (I_{-\infty}, I_0, \dots, I_n)$, 
$I = {\rm supp}\,  k$ and take $i \in [1,  r]$. 

When $I_0 = \emptyset$, $u(I) \geq 1$, there exists a 
basic Witt differential $\hat{e}' \in W\Lambda^{n-1}_{(R[P^{\rm gp}],  P^{\rm gp})/(R,  *), x}$ 
with $\omega = d\hat{e}'$, by Proposition \ref{thm4}. Then 
\begin{align*}
& g(\omega \bigwedge d\log X_i) = 
g(d\hat{e}' \bigwedge d\log X_i) 
= g(d(\hat{e}' \bigwedge d\log X_i)) 
= d(g(\hat{e}' \bigwedge d\log X_i)) \\ 
= \,  \, & d(g(\hat{e}') \bigwedge d\log X_i) 
= d(g(\hat{e}')) \bigwedge d\log X_i 
= g(d(\hat{e}')) \bigwedge d\log X_i
= g(\omega) \bigwedge d\log X_i, 
\end{align*}
where the third and the sixth equalities follows from 
Step 2 and the fourth equality follows from the claim for the element 
$\hat{e}' \in W\Lambda^{n-1}_{(R[P^{\rm gp}],  P^{\rm gp})/(R,  *), x}$. 
Thus we may exclude the case 
$I_0 = \emptyset$, $u(I) \geq 1$ in the following argument. 

We prove the claim by using the calculation given in the proof of lemma \ref{thm2}. 
\begin{enumerate} 
\item[(I)] The case $k_i = p^{-\infty}$.  

$\hat{e}(\xi,  k,  {\mathcal P}) \wedge d\log X_i = 0$
 and $g(\omega) d\log X_i = 0$ in this case. 
\item[(II)] The case $k_i = 0$. 
 
$\hat{e}(\xi,  k,  {\mathcal P}) \bigwedge d\log X_i = \hat{e}((-1)^a\xi, k', {\mathcal P}')$, 
where 
\begin{align*}
& k'_j = 
\begin{cases}
p^{-\infty} & (j=i), \\ 
k_j & (j \not= i),
\end{cases} \\ 
& {\mathcal P}' = (I'_{-\infty}, I'_0, \dots, I'_{n}), 
\quad 
I'_j = 
\begin{cases}
I_{-\infty} \cup \{i\} & (j=-\infty), \\ 
I_j & (0 \leq j \leq n), 
\end{cases}, \\ 
& \bigwedge_{j \in I_{-\infty}} d\log X_j \bigwedge d\log X_i = 
(-1)^a \bigwedge_{j \in I'_{-\infty}} d\log X_j.  
\end{align*}
Hence 
{\allowdisplaybreaks{
\begin{align*}
& g(\hat{e}(\xi,  k,  {\mathcal P}) \bigwedge d\log X_i) 
= g(\hat{e}((-1)^a\xi, k', {\mathcal P}')) \\ 
= \, &(-1)^a \sum_{i_j \in I'_j \atop j=1,  2,  \dots,  n}
{}^{V^{k(I'_0)}}\left( 
\eta \prod_{j=1}^n k'_{i_j} p^{t(I'_j)} X^{u(x)x} \right) 
\bigwedge_{j=1}^n d\log X_j \bigwedge_{j \in I'_{-\infty}} d\log X_j 
\\ = \, & 
\sum_{i_j \in I_j \atop j=1,  2,  \dots,  n}
{}^{V^{k(I_0)}}\left( 
\eta \prod_{j=1}^n k_{i_j} p^{t(I_j)} X^{u(x)x} \right) 
\bigwedge_{j=1}^n d\log X_j \bigwedge_{j \in I_{-\infty}} d\log X_j 
\bigwedge d\log X_i \\ 
= \, & 
g(\hat{e}(\xi,  k,  {\mathcal P})) \bigwedge d\log X_i
\end{align*}}} and so the claim holds. 
\item[(III)] The case 
$i \in I_a \,  (a=0,  1,  \dots, n)$. 

If we use the symbols in the proof of Lemma \ref{thm2}, 
we have the equality 
$$ \hat{e}(\xi,  k,  {\mathcal P}) = 
\hat{e}(\xi_1,  k,  {\mathcal P}^1) - 
\hat{e}(\xi_2,  k,  {\mathcal P}^2), $$
where 
\begin{align*}
& \xi_1 = 
\begin{cases}
\frac{\xi p^{u(I)}}{l_i} (-1)^{\deg \omega_2} & (a=0, I_{00} = \emptyset), \\ 
\frac{\xi}{l_i}(-1)^{\deg \omega_2} & (\text{otherwise}),
\end{cases}
\quad 
\xi_2 = \frac{\xi p^{n_i-t(I_{a1})}}{l_i} (-1)^{\deg \omega_2}, \\ 
& {\mathcal P}^j = (I_{-\infty}^j, I_0^j, \dots, I_{n+1}^j), \\ 
& 
I_k^1 = 
\begin{cases}
I_k & (k < a), \\ 
I_{a0} & (k=a), \\ 
\{i\} \cup I_{a1} & (k=a+1), \\ 
I_{k-1} & (k > a+1), 
\end{cases}
\qquad 
I_k^2 = 
\begin{cases}
I_k & (k < a), \\ 
I_{a0} \cup \{i\} & (k=a), \\ 
I_{a1} & (k=a+1), \\ 
I_{k-1} & (k > a+1). 
\end{cases}
\end{align*}
Here, if $I_{a0} = \emptyset$ (resp. if $I_{a1} = \emptyset$), 
we put $\hat{e}(\xi_1,  k,  {\mathcal P}^1) = 0$, 
(resp. $\hat{e}(\xi_2,  k,  {\mathcal P}^2) \allowbreak = 0$). 

In the following, we put $\hat{e}_1 := \hat{e}(\xi_1,  k,  {\mathcal P}^1)$, 
$\hat{e}_2 := \hat{e}(\xi_2,  k,  {\mathcal P}^2)$. We prove that   
$$g(\hat{e}_1) - g(\hat{e}_2) = g(\hat{e}(\xi,  k,  {\mathcal P})) \bigwedge 
d\log X_i. $$
\begin{enumerate}
\item[(i)] The case $a \not= 0$.  
 
When $I_{a0}, I_{a1} \not= \emptyset$, 
\begin{align*}
& g(\hat{e}_1) = 
\sum_{i_j \in I^1_j \atop j=1,  \dots,  n+1} 
{}^{V^{u(I_0)}}\left( 
\eta_1 \prod_{j=1}^{n+1} k_{i_j} p^{t(I_j^1)} X^{u(x)x} \right) 
\bigwedge_{j=1}^{n+1} d\log X_{i_j} 
\bigwedge_{j \in I_{-\infty}}d\log X_j, \\ 
& g(\hat{e}_2) = 
\sum_{i_j \in I^2_j \atop j=1,  \dots,  n+1} 
{}^{V^{u(I_0)}}\left( 
\eta_2 \prod_{j=1}^{n+1} k_{i_j} p^{t(I_j^2)} X^{u(x)x} \right) 
\bigwedge_{j=1}^{n+1} d\log X_{i_j} 
\bigwedge_{j \in I_{-\infty}}d\log X_j. 
\end{align*}
The terms with $i_{a+1} \not= i$ on the right hand side of the first 
equality correspond bijectively to the terms with $i_{a} \not= i$ 
on the right hand side of the second equality, and we have 
the equality of the coefficients of the corresponding terms  
$$ \eta_1 \prod_{j=1}^{n+1} k_{i_j} p^{t(I_j^1)} = 
\eta_2 \prod_{j=1}^{n+1} k_{i_j} p^{t(I_j^2)}:  $$
This equality is reduced to the equality 
$$ \prod_{j=1}^{n+1} k_{i_j} p^{t(I_j^1)} = 
p^{n_i - t(I_{a1})} \prod_{j=1}^{n+1} k_{i_j} p^{t(I_j^2)}, $$
which holds because 
\begin{align*}
\text{(LHS)} & = k_{i_{a+1}} p^{n_i}
\prod_{1 \leq j \leq n+1 \atop j \not= a+1} 
k_{i_j} p^{t(I_j^1)} 
= k_{i_{a+1}} p^{n_i-t(I_{a1})} p^{t(I_{a1})} 
\prod_{1 \leq j \leq n+1 \atop j \not= a+1} 
k_{i_j} p^{t(I_j^2)} \\ 
& = p^{n_i-t(I_{a1})} k_{i_{a+1}} p^{t(I_{a+1}^2)} 
\prod_{1 \leq j \leq n+1 \atop j \not= a+1} 
k_{i_j} p^{t(I_j^2)}
= \text{(RHS)}. 
\end{align*}
Therefore, in the calculation of  
$ g(\hat{e}_1) - g(\hat{e}_2)$, the terms we considered above 
cancel out. Hence,  
{\allowdisplaybreaks{
\begin{align*}
& \,  \,  \,  \, g(\hat{e}_1) - g(\hat{e}_2) \\ 
= \, & 
\sum_{{i_j \in I^1_j \atop j=1,  \dots,  n+1} \atop i_{a+1}=i} 
{}^{V^{u(I_0)}}\left( 
\eta_1 \prod_{j=1}^{n+1} k_{i_j} p^{t(I_j^1)} X^{u(x)x} \right) 
\bigwedge_{j=1}^{n+1} d\log X_{i_j} 
\bigwedge_{j \in I_{-\infty}}d\log X_j \\ 
& \hspace{1cm} - 
\sum_{{i_j \in I^2_j \atop j=1,  \dots,  n+1} \atop i_a = i} 
{}^{V^{u(I_0)}}\left( 
\eta_2 \prod_{j=1}^{n+1} k_{i_j} p^{t(I_j^2)} X^{u(x)x} \right) 
\bigwedge_{j=1}^{n+1} d\log X_{i_j}
\bigwedge_{j \in I_{-\infty}}d\log X_j \\ 
= \, & 
\sum_{{i_j \in I_j \atop j=1,  \dots,  n
} \atop i_a \in I_{a0}} 
{}^{V^{u(I_0)}}\left( (-1)^{\deg \omega_2}
\eta_1 l_i \prod_{j=1}^{n} k_{i_j} p^{t(I_j)} X^{u(x)x} \right) 
\bigwedge_{j=1}^{n} d\log X_{i_j}
\bigwedge_{j \in I_{-\infty}}d\log X_j \bigwedge d\log X_i \\ 
& \hspace{1cm} + 
\sum_{{i_j \in I_j \atop j=1,  \dots,  n} \atop i_a \in I_{a1}} 
{}^{V^{u(I_0)}}\left( (-1)^{\deg \omega_2}
\eta_2 l_i p^{t(I_{a1})-n_i} 
\prod_{j=1}^{n} k_{i_j} p^{t(I_j)} X^{u(x)x} \right) 
\bigwedge_{j=1}^{n} d\log X_{i_j} \\ & \hspace{10cm} 
\bigwedge_{j \in I_{-\infty}}d\log X_j 
\bigwedge d\log X_i \\ 
= \, & 
\sum_{i_j \in I_j \atop j=1,  \dots,  n} 
{}^{V^{u(I_0)}}\left(
\eta \prod_{j=1}^{n} k_{i_j} p^{t(I_j)} X^{u(x)x} \right) 
\bigwedge_{j=1}^{n} d\log X_{i_j} 
\bigwedge_{j \in I_{-\infty}}d\log X_j \bigwedge d\log X_i \\ 
= \, & 
g(\hat{e}(\xi,  k,  {\mathcal P})) \bigwedge d\log X_i, 
\end{align*}}}%
as required. When 
$I_{a0} = \emptyset$ or $I_{a1} = \emptyset$, 
a similar (and simpler) calculation shows the equality 
$g(\hat{e}_1) - g(\hat{e}_2) = g(\hat{e}(\xi,  k,  {\mathcal P})) \bigwedge 
d\log X_i$. 
\item[(ii)] The case $a=0, I_{00} \not= \emptyset$. 

When $I_{01} \not= \emptyset$, 
\begin{align*}
& g(\hat{e}_1) = 
\sum_{i_j \in I^1_j \atop j=1,  \dots,  n+1} 
{}^{V^{u(I_0)}}\left( 
\eta_1 \prod_{j=1}^{n+1} k_{i_j} p^{t(I_j^1)} X^{u(x)x} \right) 
\bigwedge_{j=1}^{n+1} d\log X_{i_j} 
\bigwedge_{j \in I_{-\infty}}d\log X_j, \\ 
& g(\hat{e}_2) = 
\sum_{i_j \in I^2_j \atop j=1,  \dots,  n+1} 
{}^{V^{u(I_0)}}\left( 
\eta_2 \prod_{j=1}^{n+1} k_{i_j} p^{t(I_j^2)} X^{u(x)x} \right) 
\bigwedge_{j=1}^{n+1} d\log X_{i_j}
\bigwedge_{j \in I_{-\infty}}d\log X_j. 
\end{align*}
By a similar argument to that in (i), 
we see that the sum of the terms with $i_{1} \not= i$ 
in the right hand side of the first equality is equal to 
the right hand side of the equality. Hence,  
{\allowdisplaybreaks 
\begin{align*}
& g(\hat{e}_1) - g(\hat{e}_2) \\ 
= \, & 
\sum_{{i_j \in I^1_j \atop j=1,  \dots,  n+1} \atop i_{1}=i} 
{}^{V^{u(I_0)}}\left( 
\eta_1 \prod_{j=1}^{n+1} k_{i_j} p^{t(I_j^1)} X^{u(x)x} \right) 
\bigwedge_{j=1}^{n+1} d\log X_{i_j}
\bigwedge_{j \in I_{-\infty}}d\log X_j \\ 
= \, & 
\sum_{i_j \in I_j \atop j=1,  \dots,  n} 
{}^{V^{u(I_0)}}\left( (-1)^{\deg \omega_2}
\eta_1 l_i \prod_{j=1}^{n} k_{i_j} p^{t(I_j)} X^{u(x)x} \right) 
\bigwedge_{j=1}^{n} d\log X_{i_j}
\bigwedge_{j \in I_{-\infty}}d\log X_j \bigwedge d\log X_i \\ 
= \, & 
\sum_{i_j \in I_j \atop j=1,  \dots,  n} 
{}^{V^{u(I_0)}}\left(
\eta \prod_{j=1}^{n} k_{i_j} p^{t(I_j)} X^{u(x)x} \right) 
\bigwedge_{j=1}^{n} d\log X_{i_j}
\bigwedge_{j \in I_{-\infty}}d\log X_j \bigwedge d\log X_i \\ 
= \, & 
g(\hat{e}(\xi,  k,  {\mathcal P})) \bigwedge d\log X_i, 
\end{align*}}%  
as required. 
When $I_{01} = \emptyset$, 
a similar (and simpler) calculation shows the equality 
$g(\hat{e}(\xi,  k,  {\mathcal P}) \allowbreak \bigwedge \allowbreak d\log X_i) = 
g(\hat{e}(\xi,  k,  {\mathcal P})) \bigwedge d\log X_i.$
\item[(iii)] The case $a = 0$, $I_{00} = \emptyset$. 

$g(\hat{e}_1)$ is calculated as follows: 
{\allowdisplaybreaks{
\begin{align*}
& g(\hat{e}_1) = 
\sum_{i_j \in I^1_j \atop j=2,  \dots,  n+1} 
d^{V^{u(I_1^1)}}\left( 
\eta_1 \prod_{j=2}^{n+1} k_{i_j} p^{t(I_j^1)} X^{u(x)x} \right) 
\bigwedge_{j=2}^{n+1} d\log X_{i_j} 
\bigwedge_{j \in I_{-\infty}}d\log X_j \\ 
= \, & 
\sum_{i_j \in I^1_j \atop j=2,  \dots,  n+1} 
d^{V^{u(I)}}\left( 
\frac{\eta p^{u(I)}}{l_i} (-1)^{\deg \omega_2} 
\prod_{j=2}^{n+1} k_{i_j} p^{t(I_j^1)} X^{u(x)x} \right) 
\bigwedge_{j=2}^{n+1} d\log X_{i_j}
\bigwedge_{j \in I_{-\infty}}d\log X_j \\ 
= \, & 
\sum_{i_j \in I^1_j \atop j=2,  \dots,  n+1} 
{}^{V^{u(I)}}\left( 
\frac{\eta}{l_i} (-1)^{\deg \omega_2} 
\prod_{j=2}^{n+1} k_{i_j} p^{t(I_j^1)} X^{u(x)x} 
\sum_{t=1}^r p^{u(x)}x_t d\log X_t \right) \\ 
& \hspace{8cm} 
\bigwedge_{j=2}^{n+1} d\log X_{i_j}
\bigwedge_{j \in I_{-\infty}}d\log X_j \\ 
= \, & 
\sum_{i_j \in I^1_j \atop j=2,  \dots,  n+1} \sum_{t=1}^r 
{}^{V^{u(I)}}\left( 
\frac{\eta}{l_i} (-1)^{\deg \omega_2} 
\prod_{j=2}^{n+1} k_{i_j} p^{t(I_j^1)} X^{u(x)x} 
\cdot p^{u(x)}x_t \right) \bigwedge d\log X_t \\ 
& \hspace{8cm} 
\bigwedge_{j=2}^{n+1} d\log X_{i_j}
\bigwedge_{j \in I_{-\infty}}d\log X_j. 
\end{align*}}}
In the sum of the rightmost side, 
the terms with 
$t \notin {\rm supp}\,  k$ are $0$ because $x_t = 0$. 
The terms with 
$t \in I_{-\infty}^1 = I_{-\infty}$ are also $0$ because 
they contain $d\log X_t$ twice. 
For the terms with 
$t \in I_j^1 = I_{j-1} \, (2 \leq j \leq n+1)$, 
they are equal to $0$ if $t=i_j$, because they contain 
$d\log X_t = d\log X_{i_j}$ twice. If 
$t \not= i_j$, the terms cancel out with the terms with 
$t$ and $i_j$ interchanged. Thus, we have only to 
consider the remaining terms, namely, the terms with 
$t \in I_0^1 = \{i\} \cup I_{01}$. The sum of the terms with 
$t = i$ is calculated as follows: 
{\allowdisplaybreaks \begin{align*}
& 
\sum_{i_j \in I^1_j \atop j=2,  \dots,  n+1}
{}^{V^{u(I)}}\left( 
\frac{\eta}{l_i} (-1)^{\deg \omega_2} 
\prod_{j=2}^{n+1} k_{i_j} p^{t(I_j^1)} X^{u(x)x} 
\cdot p^{u(x)}x_i \right) \bigwedge d\log X_i \\ 
& \hspace{8cm} 
\bigwedge_{j=2}^{n+1} d\log X_{i_j}
\bigwedge_{j \in I_{-\infty}}d\log X_j \\ 
= \, & 
\sum_{i_j \in I_j \atop j=1,  \dots,  n}
{}^{V^{u(I)}}\left( 
\eta
\prod_{j=1}^{n} k_{i_j} p^{t(I_j^1)} X^{u(x)x} 
\right)
\bigwedge_{j=1}^{n} d\log X_{i_j} 
\bigwedge_{j \in I_{-\infty}}d\log X_j \bigwedge d\log X_i \\ 
= \, & g(\hat{e}(\xi,  k,  {\mathcal P})) \bigwedge d\log X_i. 
\end{align*}}
Also, the sum of the terms with $t \in I_{01}$ is calculated as follows: 
{\allowdisplaybreaks{
\begin{align*}
& 
\sum_{i_j \in I^1_j \atop j=2,  \dots,  n+1} \sum_{t \in I_{01}} 
{}^{V^{u(I)}}\left( 
\frac{\eta}{l_i} (-1)^{\deg \omega_2} 
\prod_{j=2}^{n+1} k_{i_j} p^{t(I_j^1)} X^{u(x)x} 
\cdot p^{u(x)}x_t \right) \bigwedge d\log X_t \\ 
& \hspace{8cm} 
\bigwedge_{j=2}^{n+1} d\log X_{i_j}
\bigwedge_{j \in I_{-\infty}}d\log X_j \\ 
= \, & 
\sum_{i_j \in I^2_j \atop j=1,  \dots,  n+1} 
{}^{V^{u(I)}}\left( 
\frac{\eta p^{n_i - t(I_{01})}}{l_i} (-1)^{\deg \omega_2} 
\prod_{j=1}^{n+1} k_{i_j} p^{t(I_j^2)} X^{u(x)x} \right) \\ 
& \hspace{8cm} 
\bigwedge_{j=1}^{n+1} d\log X_{i_j}
\bigwedge_{j \in I_{-\infty}}d\log X_j \\ 
=\, & g(\hat{e}_2). 
\end{align*}}}%
By these calculations, we see the equality 
$g(\hat{e}_1) - g(\hat{e}_2) = 
g(\hat{e}(\xi,  k,  {\mathcal P})) \bigwedge d\log X_i$. 
\end{enumerate}
\end{enumerate}

\noindent 
{\bf Step 5.} $g \circ f = {\rm id}$. 

For $\omega \in W\Lambda_{(R[P],  P)/(R,  {*}), x}^{\bullet}$ and $i \in [1,  r]$, we have 
\[ g \circ f (\omega \bigwedge d \log X_i)= g \left( f (\omega) \bigwedge d \log X_i \right) = g \circ f (\omega) \bigwedge d \log X_i. \] 
Also, for a $P$-basic Witt differential $b(\xi,  x)$ of degree $0$, we have 
\[ g \circ f(db(\xi, x)) = dg \circ f(b(\xi, x)). \]
Thus, by Proposition \ref{thm23} and the definition of $P$-basic Witt differentials, 
it suffices to prove the equality  $g \circ f (\omega)  = \omega$ when $\omega$ is a  
$P$-basic Witt differential of degree $0$, and it is trivial. 
\medskip 

\noindent
{\bf Step 6.} $f \circ g = {\rm id}$. 

By Proposition \ref{thm20}, it suffices to prove the equality $f \circ g(\hat{e}) =\hat{e}$ for a 
basic Witt differential $\hat{e} :=\hat{e}(\xi,  k ,  \mathcal{P})$. 
Since $f \circ g(\hat{e})$ is obtained by replacing $p^{u(x)}x$ in the definition of $g(\hat{e})$ by 
$p^{u(I)}k^{+}$ (where $I = {\rm supp}\,k$), the claim follows from the definition of $g$ and 
the equations \eqref{1}, \eqref{2}, \eqref{3}. 
\end{proof}

\begin{prop}\label{thm31}
For $x \in P[\frac{1}{p}]$, the isomorphisms $f, g$ induce the isomorphisms of $x$-parts of truncated 
relative log de Rham-Witt complexes 
\[ f_m : W_m\Lambda_{(R[P],  P)/(R,  *),  x}^{\bullet} \to W_m\Lambda_{R[P^{\rm gp}], P^{\rm gp})/(R, *), x}^{\bullet}, \]
\[ g_m : W_m\Lambda_{R[P^{\rm gp}], P^{\rm gp})/(R, *), x}^{\bullet} \to W_m\Lambda_{(R[P],  P)/(R,  *),  x}^{\bullet}. \]
\end{prop}

\begin{proof}
First we define the morphism 
$f_m:W_m\Lambda_{(R[P],  P)/(R,  *),  x}^{\bullet} \to W_m\Lambda_{(R[P^{\rm gp}], P^{\rm gp})/(R, *), x}^{\bullet} $. By the universality (Proposition-Definition \ref{thm30})
of the relative log de Rham-Witt complex, we have the morphism 
\[ h_m : W_m\Lambda_{(R[P],  P)/(R,  *)}^{\bullet} \to W_m\Lambda_{(R[P^{\rm gp}], P^{\rm gp})/(R, *)}^{\bullet}. \]
(This is the truncated version of the morphism $h$ in \eqref{eq:h}.) 
For $\omega \in W_m\Lambda_{(R[P],  P)/(R,  *),  x}^{\bullet} \subset W_m\Lambda_{(R[P],  P)/(R,  *)}^{\bullet} $, we would like to define $f_m$ by 
$f_m(\omega):= h_m(\omega)$. In order to show that it is well-defined, 
we will check that $h_m(\omega)$ belongs to 
$W_m\Lambda_{(R[P^{\rm gp}], P^{\rm gp})/(R, *),  x}^{\bullet}$. 
Take a lift $\tilde{\omega}$ of 
$\omega$ to $W\Lambda_{(R[P],  P)/(R,  *),  x}^{\bullet}$. 
Then the image of it by the composite 
\[ W\Lambda_{(R[P],  P)/(R,  *),  x}^{\bullet} \xrightarrow{f} 
W\Lambda_{(R[P^{\rm gp}], P^{\rm gp})/(R, *),  x}^{\bullet} \to 
W_m\Lambda_{(R[P^{\rm gp}], P^{\rm gp})/(R, *),  x}^{\bullet} \] 
is equal to $h_m(\omega)$, because $f$ is the map induced by 
the map $h: W\Lambda_{(R[P],  P)/(R,  *)}^{\bullet} \to 
W\Lambda_{(R[P^{\rm gp}], P^{\rm gp})/(R, *)}^{\bullet}
$ in \eqref{eq:h} and the maps $h, h_m$ are compatible.  
%compatible via truncations of the form $W\Lambda_{-/-}^{\bullet} \to 
%W_m\Lambda_{-/-}^{\bullet}$. 
Hence $h_m(\omega)$ belongs to $W_m\Lambda_{(R[P^{\rm gp}], P^{\rm gp})/(R, *),  (x_i)_i}^{\bullet}$ and so the map $f_m$ is well-defined. 

Next we define the morphism 
$g_m:W_m\Lambda_{(R[P^{\rm gp}], P^{\rm gp})/(R, *),  x}^{\bullet} \to W_m\Lambda_{(R[P],  P)/(R,  *),  x}^{\bullet} $. For an element 
$\omega$ in $W_m\Lambda_{(R[P^{\rm gp}], P^{\rm gp})/(R, *),  x}^{\bullet}$, 
take its lift $\tilde{\omega}$ in 
$W\Lambda_{(R[P^{\rm gp}], P^{\rm gp})/(R, *),  x}^{\bullet}$
and we would like to define $g_m(\omega)$ by 
the image of $g(\tilde{\omega})$ in 
$W_m\Lambda_{(R[P], P)/(R, *),  x}^{\bullet}$. 
If we take another lift $\tilde{\omega}'$ of $\omega$ in 
$W\Lambda_{(R[P^{\rm gp}], P^{\rm gp})/(R, *),  x}^{\bullet}$, 
we can express the difference $\tilde{\omega} -\tilde{\omega}'$ uniquely as 
\[ \tilde{\omega}_1 -\tilde{\omega}_2 =\sum_{k^+=x, \mathcal{P}} \hat{e} (\xi_{k,  \mathcal{P}},  k,  \mathcal{P})\] by Lemma \ref{thm2}, and each 
$\xi_{k,  \mathcal{P}}$ belongs to ${}^{V^{m}}W(R)$ by Corollary \ref{thm50}. 
Hence we see that the image of the right hand side by $g$ is zero in $W_m\Lambda^{\bullet}_{(R[P],P)/(R,*),x}$
and so the morphism $g_m$ is well-defined. 

We see that $g_m$ is the inverse of $f_m$ by construction. So $f_m, g_m$ are isomorphisms. 
\end{proof}

Now we prove the decomposition of $W_m\Lambda_{(R[P],  P)/(R,  {*})}^{k}$ into 
$x$-parts ($x \in P[\frac{1}{p}]$), which is the main result of this subsection: 

\begin{prop}\label{thm22}
Any element in 
$W\Lambda_{(R[P],  P)/(R,  {*})}^{k}$ is uniquely written as a convergent sum of 
elements in $W\Lambda_{(R[P],  P)/(R,  {*}),  x}^{k} \, (x \in P[\frac{1}{p}])$. 
Also, we have the canonical direct sum decomposition 
\[
W_m\Lambda_{(R[P],  P)/(R,  {*})}^{k} = \bigoplus_{x \in P[\frac{1}{p}]} W_m\Lambda_{(R[P],  P)/(R,  {*}), x}^{k}.\]
\end{prop}

\begin{proof}  
By Corollary \ref{thm23}, it suffices to prove the uniqueness of the expression. 
So it suffices to prove that, if we have an equality
\[ 0 = \sum_{x \in P[\frac{1}{p}]} \omega_x \] 
 in $W_m\Lambda_{(R[P],  P)/(R,  {*})}^{k}$ 
where the right hand side is a finite sum 
with $\omega_x \in W_m\Lambda_{(R[P],  P)/(R,  {*}), x}^{k}$, 
all $\omega_x$'s are equal to $0$. 

Consider the following commutative diagram
\[ \xymatrix{
\bigoplus_{x \in P[\frac{1}{p}]} W_m\Lambda_{(R[P],  P)/(R,  {*}), x}^{\bullet} \ar[r]^-{\oplus f_m}_-{\cong} \ar[d]^{{\rm sum}} & \bigoplus_{x \in P[\frac{1}{p}]} W_m \Lambda_{(R[P^{\rm gp}],  P^{\rm gp})/(R,  {*}), x}^{\bullet} \ar[d]^{{\rm sum}} \\
  W_m\Lambda_{(R[P],  P)/(R,  {*})}^{\bullet}  \ar[r]^-h & W_m\Lambda_{(R[P^{\rm gp}],  P^{\rm gp})/(R,  {*})}^{\bullet}, 
  }
\]
where $\oplus f_m$ is the direct sum of the isomorphisms $f_m$ in 
Proposition \ref{thm31}, $h$ is as in \eqref{eq:h} and ${\rm sum}$ are the morphisms of taking sum.  
By assumption, the image of the element $(\omega_x)_x \in 
\bigoplus_{x \in P[\frac{1}{p}]} W_m\Lambda_{(R[P],  P)/(R,  {*}), x}^{\bullet}$ 
by $h \circ {\rm sum} = {\rm sum} \circ (\oplus f_m)$ is equal to $0$. 
Also, by Lemma \ref{thm2} and Proposition \ref{thm20}, the morphism 
${\rm sum}$ on the right is injective. 
Thus we see that $(f_m(\omega_x))_x = (0)_x$ and so 
$\omega_x = 0$ for all $x \in P[\frac{1}{p}]$. 
\end{proof}

%%%%%%%%%%%%%%%%%%%%  Comparison isom in absolute case %%%%%%%%%%%%%
\subsection{Comparison isomorphism (I)}

Let $R$ be a commutative ring in which $p$ is nilpotent and let 
let $P$ be an fs monoid such that $P^{\rm gp}$ is torsion free. 
(Then $(R[P],P)/(R,*)$ is log smooth.) 

Let $((A_m, P), \phi_m, \delta_m)_m$ be the log Frobenius lift 
of $(R[P],P)$ over $(R,*)$ in 
Example \ref{exam} with $Q = *$ 
(so $A_m = W_m(R) \otimes_{\mathbb{Z}} \mathbb{Z}[P] = W_m(R)[P]$), and let 
$(Z_m, \Phi_m, \Delta_m)_m$ be
the corresponding log Frobenius lift (as log schemes) of 
$Z := \mbox{Spec}(R[P], P)$ over 
$\mbox{Spec}(R, *)$. 
Then, by the construction in Section 1.3 using $(Z_m, \Phi_m, \Delta_m)_m$, 
we obtain the comparison morphism 
\begin{align*}
 \Lambda_{Z_{m}/{{\rm Spec}}(W_{m}(R),*)}^{\bullet} \to W_{m} \Lambda_{Z/({\rm Spec}(R),*)}^{\bullet} % \tag{3.3.1}
\end{align*}
of complexes of quasi-coherent sheaves, hence the morphism 
\begin{equation}\label{eq:eqcp} 
c_P: \Lambda_{(W_m(R)[P], P)/(W_{m}(R),*)}^{\bullet} \to W_{m} \Lambda_{(R[P],P)/(R, *)}^{\bullet}
 % \tag{3.3.1}
\end{equation}
of complex of $W_m(R)$-modules. 

\begin{theorem}\label{thm36}
The map $c_P$ is a quasi-isomorphism.  
\end{theorem}

\begin{proof}

By Proposition \ref{thm22}, we have the decomposition 
$W_m\Lambda_{(R[P],  P)/(R,  {*})}^{k} = \bigoplus_{x \in P[\frac{1}{p}]} W_m\Lambda_{(R[P],  P)/(R,  {*}), x}^{k}$. 
We first prove that the map $c_P$ gives an isomorphism onto 
the subcomplex $\bigoplus_{x \in P} W_m\Lambda_{(R[P],  P)/(R,  {*}), x}^{\bullet}$. 
Fix an isomorphism $P^{\rm gp} \cong \mathbb{Z}^r$. Then, since 
\begin{align*}
& c_P (\xi T^x \bigwedge_I d \log T_i ) = \xi X^x \bigwedge_I d \log X_i = b(\xi,x,I), \\ 
& c_P (d(\xi T^x) \bigwedge_I d \log T_i ) = d(\xi X^x) \bigwedge_I d \log X_i = b(\xi,x,I \cup \{0\}) 
\end{align*}
for $x \in P, I \subset [1,r]$ and $\xi \in W_m(R)$, 
$c_P$ gives a surjection onto $\bigoplus_{x \in P} W_m\Lambda_{(R[P],  P)/(R,  {*}), x}^{\bullet}$. 
Also, we have the decomposition 
\[\Lambda_{(W_{m}(R)[P], P)/(W_{m}(R),  {*})}^{\bullet} \cong \bigoplus_{x \in P, \atop I \subset [1, r]} 
W_m(R) T^x \bigwedge_I d\log T_i \] 
%Since $c_P\left(T^x \bigwedge_I d \log T_i \right) = X^x \bigwedge_I d \log X_i $, 
%$c_P$ gives a surjection onto $\bigoplus_{x \in P} W_m\Lambda_{(R[P],  P)/(R,  {*}), x}^{\bullet}$. 
%Also, 
and the following commutative diagram, where $i$ is the canonical inclusion and 
$W_m(R) \mathcal{A}_x$ is as in the proof of Lemma \ref{thm33}: 
\[ \xymatrix{
\bigoplus_{x \in P, \atop I \subset [1, r]} W_m(R) T^x \bigwedge d\log T_i \ar[r]^-{i} \ar[d]_{c_P} & \bigoplus_{x \in P^{\rm gp}, \atop I \subset [1, r]} W_m(R) T^x \bigwedge d\log T_i 
= \bigoplus_{x \in P^{\rm gp}} 
W_m(R) {\mathcal{A}}_x, \ar[d]^{c_{P^{\rm gp}}} 
\\
\bigoplus_{x \in P[\frac{1}{p}]} W_m\Lambda_{(R[P],  P)/(R,  {*}), x}^{\bullet} \ar[r]^-{\oplus f_m} & \bigoplus_{x \in P^{\rm gp}[\frac{1}{p}]} W_m\Lambda_{(R[P^{\rm gp}], P^{\rm gp})/(R,  {*}), x}^{\bullet}.
}
\]
% By Proposition \ref{thm31}, $\oplus f_m$ is an injection. Also, 
By the proof of Lemma 
\ref{thm33}, any element in $W_m(R) \mathcal{A}_x$
% (with the notation in the proof of Lemma \ref{thm33}) 
is written uniquely as a $W_m(R)$-linear combination of the $p$-basic elements $\bar{e}(k,  \mathcal{P})$ with 
$k^{+} = x$. For such a $p$-basic element $\bar{e}(k, \mathcal{P})$ and $\xi \in W_m(R)$, we have 
the equality 
$c_{P^{\rm gp}} \left( \xi \bar{e}(k, \mathcal{P}) \right) = \hat{e}(\xi, k, \mathcal{P})$. 
Therefore, by Proposition \ref{thm20} and Proposition \ref{thm2}, $c_{P^{\rm gp}}$ induces the isomorphism 
$W_m(R) \mathcal{A}_x \xrightarrow{\cong} W_m\Lambda_{(R[P^{\rm gp}], P^{\rm gp})/(R,  {*}), x}^{\bullet}$. 
By taking direct sums with respect to $x \in P^{\rm gp}$, we see that 
the map $c_{P^{\rm gp}}$ is an injection. 
Therefore, $c_P$ is also an injection and so $c_P$ gives an isomorphism onto 
$\bigoplus_{x \in P} W_m\Lambda_{(R[P],  P)/(R,  {*}), x}^{\bullet}$. 

So it suffices to prove that the complex 
$\bigoplus_{x \in P[\frac{1}{p}] \setminus P} W_m\Lambda_{(R[P],  P)/(R,  {*}), x}^{\bullet}$ is acyclic. 
To do so, we prove that, for each $x \in P[\frac{1}{p}] \setminus P$, the complex 
$W_m\Lambda_{(R[P],  P)/(R,  {*}), x}^{\bullet}$ is acyclic. 

Using the isomorphism $W_m\Lambda_{(R[P],  P)/(R,  {*}), x}^{\bullet} \cong W_m\Lambda_{(R[P^{\rm gp}],  P^{\rm gp})/(R,  {*}), x}^{\bullet}$ of Proposition \ref{thm31},  
%(where we denoted the image of $x$ by the map $P[\frac{1}{p}] \to P^{\rm gp}[\frac{1}{p}] \cong \mathbb{Z}[\frac{1}{p}]^r$
%by $(x_i)_i$), 
we are reduced to proving the acyclicity of $W_m\Lambda_{(R[P^{\rm gp}],  P)/(R,  {*}), x}^{\bullet}$. 
Since we have $u_{P^{\rm gp}}(x) = u_P(x) \geq 1$ by Lemma \ref{thm3} 
and the assumption $x \in P[\frac{1}{p}] \setminus P$, 
the required acyclicity follows from Lemma \ref{thm4}. 
\end{proof}

Now we prove our main comparison isomorphism when the log structure of the base scheme is trivial. 
(Although the theorem is a particular case of our main theorem (Theorem \ref{thm:main2}), we include the proof here 
because the statement of the theorem is simpler and thus it would be helpful to the reader.) 
Let $f: (X,{\mathcal M}) \to (Y, *)$ be a log smooth morphism of fs log schemes on which $p$ is nilpotent, 
where $*$ denotes the trivial log structure. 

\begin{theorem}\label{thm10}
In the situation above, the comparison morphism 
\[\mathbb{R}u_{m*}\mathcal{O}_m \to W_{m} \Lambda_{(X, \mathcal{M})/(Y,  {*})}^{\bullet}\]
we constructed in Section 1.3 is a quasi-isomorphism. 
\end{theorem}

\begin{proof} (cf. \cite[Thm 3.5]{LZ}, \cite[Thm 7.2]{Ma}) 
The claim is etale local on $X$ and $Y$. Thus we may assume that $(Y,*) = {\rm Spec}\,(R,*)$. 
Let $x$ be a geometric point of $X$ and put $P = {\mathcal M}_x/{\mathcal O}_{X,  x}^*$. 
Then, since $P^* = \{0\}$, $P$ is an fs monoid such that $P^{\rm gp}$ is torsion free. 
Applying \cite[Lem.~4.1.1]{KF} to the morphism $f$, we see that there 
exist an etale neighborhood $U$ of $x$ in $X$
and a chart $P_U \to {\mathcal M}|_U$ of $(U,{\mathcal M}|_U)$ such that 
the induced morphism $(U,{\mathcal M}|_U) \to \mbox{Spec}(R[P],  P)$ is strict smooth. 
Then, after shrinking $U$ suitably, we can take a factorization 
$(U,{\mathcal M}|_U)  \to \mbox{Spec}(R[P \oplus {\mathbb{N}}^a], P \oplus {\mathbb{N}}^a) \to \mbox{Spec}(R[P],  P)$
such that the first morphism is strict etale. 
Thus we may assume that $(X, \mathcal{M}) = (R[P \oplus {\mathbb{N}}^a], P \oplus {\mathbb{N}}^a)$. 
Because $P$ is an fs monoid with $P^{\rm gp}$ torsion free, so is the monoid 
$P \oplus {\mathbb{N}}^a$. Hence the comparison morphism in this case is the map 
$c_{P \oplus {\mathbb{N}}^a}$ in \eqref{eq:eqcp}  (with $P$ replaced by 
$P \oplus {\mathbb{N}}^a$), and it is a quasi-isomorphism by 
Theorem \ref{thm36}. So we are done. 
\end{proof}

%%%%%%%%%%%%%%%%%%%%%%%%%%%%%%%%%%%%%%%%%%%%%%%%%%%%%%%%%%%%%%%%

\section{Relative log de Rham-Witt complex (II)}

In the first two subsections of this section, 
we study the relative log de Rham-Witt complex 
$W_m\Lambda^{\bullet}_{(R[P],P)/(R[Q],Q)}$ for a ${\mathbb Z}_{(p)}$-algebra $R$ 
and an injective $p$-saturated homomorphism of fs monoids $Q \to P$ such that 
$Q^{\rm gp}, P^{\rm gp}, P^{\rm gp}/Q^{\rm gp}$ are torsion free. 
The main result is the existence of natural decomposition 
\begin{equation}\label{eq:decompII}
W_m\Lambda^{\bullet}_{(R[P],P)/(R[Q],Q)} 
= \bigoplus_{x \in P[\frac{1}{p}]} W_m\Lambda^{\bullet}_{(R[P],P)/(R[Q],Q), x}
\end{equation}
indexed by $P[\frac{1}{p}]$, which is a relative version of the decomposition 
\eqref{eq:decomp}. The proof is done by combining the decomposition 
\eqref{eq:decomp} in the absolute case and the exact sequence 
\[W_m\Lambda_{(R[Q],  Q)/(R, {*})}^{1} \otimes_{W_m(R[Q])} W_m\Lambda_{(R[P], P)/(R, {*})}^{\bullet -1} \to W_m\Lambda_{(R[P],  P)/(R, {*})}^{\bullet} \to W_m\Lambda_{(R[P],  P)/(R[Q], Q)}^{\bullet} \to 0 \]
proven in \cite[Prop.~3.11]{Ma}. 

In the last subsection, using the decomposition \eqref{eq:decompII}, we prove that 
the comparison morphism 
\[\mathbb{R}u_{m*}\mathcal{O}_m \to W_{m} \Lambda_{(X, \mathcal{M})/(Y, {\mathcal N})}^{\bullet}\]
constructed in Section 1.3 is a quasi-isomorphism when $(X, {\mathcal M}) \to (Y, {\mathcal N})$ 
is a log smooth saturated morphism of fs log schemes on which $p$ is 
nilpotent such that $(Y,{\mathcal N})$ is etale locally log smooth over a scheme with trivial log structure. 

\subsection{Preliminaries on monoids (II)}

%\subsection{$p$-quasi saturated morphism}

In this subsection, we give preliminaries on $p$-quasi-saturared morphisms 
of monoids which we use later. 

\begin{definition}
Let $p$ be a prime number. A monoid $P$ is called 
$p$-saturated if $P$ is integral and satisfies the following condition: 
Any $x  \in P^{\rm gp} $ with $ px \in P $ belongs to $P$. 
\end{definition}

\begin{definition}
Let $P, Q$ be integral monoids. 
\begin{enumerate}[(1)]
\item  \cite[Prop.~4.1]{Kato} A homomorphism $f:Q \to P $ is integral
if the following condition is satisfied: For any elements 
$a_1,  a_2 \in Q, b_1,  b_2 \in P$ satisfying $ f(a_1)+b_1 = f(a_2)+b_2$, 
there exist $a_3,  a_4 \in Q, b \in P$ with the conditions $ b_1 = f(a_3)+b, b_2=f(a_4)+b, a_1+a_3=a_2+a_4$. 

\item  \cite[Def.~4.6]{Kato} A homomorphism $f:Q \to P$ is exact
if the following diagram is Cartesian: 
\[ \xymatrix{
Q \ar[r]^{f} \ar[d] &P \ar[d] \\
Q^{\rm gp} \ar[r]^{f^{\rm gp}} & P^{\rm gp}. 
}
\]  

\item \cite[Def.~I.3.5]{Tsuji} A morphism $f:Q \to P$ is $p$-quasi-saturated 
if the homomorphism $h: P \oplus_{Q, p} Q \to P$ defined by the following diagram 
 is exact, where $P \oplus_{Q,p} Q$ denotes the pushout of the diagram 
$P \xleftarrow{f} Q \xrightarrow{p\times} Q$ in the category of integral monoids: 
\[ \xymatrix{
Q \ar[r]^{p \times} \ar[d]^{f} & Q \ar[d] \ar[rd]^{f} & \\
P \ar[r]^{} \ar@/_10pt/[rr]_{p\times} & P \oplus_{Q, p} Q \ar[r]^(.6){h} & P. 
  }
\] 
\item \cite[Def.~I.3.12]{Tsuji} A homomorphism $f:Q \to P$ is $p$-saturated
if it is $p$-quasi-saturated and integral. 
\end{enumerate}
\end{definition}

The condition of $p$-quasi-saturatedness is described as follows, where 
$ P \oplus_{Q, p} Q =: M$: 
\medskip 

\noindent 
 $(\diamond)$ \quad If an element $(r_1-r_2, q_1-q_2) \in M^{\rm gp} $ \, ($r_1, r_2 \in P, q_1,  q_2 \in Q$) satisfies $p(r_1-r_2)+q_1-q_2 \in P$, $(r_1-r_2, q_1-q_2) \in M $, namely, 
there exists an element $x \in Q^{\rm gp}$ with $r_1 - r_2 + x \in P $ and $q_1 -q_2 -px \in Q$. 

\begin{lemma}\label{thm38}
Let $Q \to P$ be a $p$-quasi-saturated morphism of integral monoids and assume that $P$ is $p$-saturated. Then we have the following: 
\medskip 

\noindent  
$(\star)$ \quad If $x$ is an element in $P$ with $x \in p \left( P^{\rm gp}/Q^{\rm gp} \right)$ when it is regarded as an element in $P^{\rm gp}/Q^{\rm gp}$, 
there exist elements $r \in P ,  q \in Q$ such that $x = pr + q$ in $P$. 
\end{lemma}

\begin{proof}
Let $P' := \mbox{Im}(P \to P^{\rm gp}/Q^{\rm gp})$. 
(Then ${P'}^{\rm gp} = P^{\rm gp}/Q^{\rm gp}$.) 
We first prove that $P'$ is $p$-saturated. 
If $y$ is an element in $P^{\rm gp} / Q^{\rm gp} $ with $py \in P'$, 
a representative $z \in P^{\rm gp}$ of $y$ satisfies the equality 
$pz = \zeta +q_1 -q_2$ for some $\zeta \in P, q_1, q_2 \in Q$. 
Hence $p(z+q_2) = \zeta +q_1 +(p-1)q_2$ and the right hand side of it belongs to $P$. 
Since $P$ is $p$-saturated, $z+q_2 \in P$. So $y \in P'$ and thus $P'$ is 
$p$-saturated, as required. 

%The claim in the previous paragraph implies that, 
If $x$ is an element in $P$ with $x \in p \left( P^{\rm gp}/Q^{\rm gp} \right)$ when regarded as an element in $P^{\rm gp}/Q^{\rm gp}$, there exists some 
$y \in P^{\rm gp}/Q^{\rm gp}$ with $x =py$. Since $py \in P'$, 
the claim in the previous paragraph implies that $y \in P'$. Thus there exist elements 
 $r_2 \in P ,  q_1, q_2 \in Q$ with $x + q_1 = pr_2 +q_2$. 
Since $q_2-q_1 +pr_2 =x \in P$, the property $(\diamond)$ implies 
the existence of elements $x_1 , x_2 \in Q$ with 
\[
\left\{ \begin{array}{l}
 r_2 +x_1 -x_2  \in P, \\
   q_2 - q_1  +p(x_2 - x_1) \in Q. 
 \end{array}  \right.
\] Thus we have the equality $x = p(r_2 +x_1 -x_2) +(q_2 -q_1 + p(x_2 -x_1))$, 
which is a required expression. 
\end{proof}

\begin{lemma}
Let $Q \to P$ be as in Lemma \ref{thm38} and assume moreover that 
$P^{\rm gp}/Q^{\rm gp}$ does not have $p$-torsion. Then we have 
the following: 
\medskip 

\noindent 
$(\star \star)$ \, For any $n \in \mathbb{N}$, if $x$ is an element in $P$ with 
$x \in p^n \left(P^{\rm gp}/Q^{\rm gp} \right)$ when it is regarded as an element in 
$x \in P^{\rm gp}/Q^{\rm gp}$, there exist elements 
$r \in P ,  q \in Q$ such that $x = p^n r + q$. 
\end{lemma}

\begin{proof}
We prove the lemma by induction on $n$. When $n=1$, the claim is nothing but 
the property $(\star)$. We assume that $(\star \star)$ holds for $n = k$ and 
consider the case $n = k+1$. Let $x$ be an element in $P$
with $x \in p^{k+1} \left( P^{\rm gp}/Q^{\rm gp} \right)$ when regarded as an element in 
$x \in P^{\rm gp}/Q^{\rm gp}$. Then, since $x \in p^{k} \left( P^{\rm gp}/Q^{\rm gp} \right)$, 
there exist $r \in P ,  q \in Q$ with $x = p^k r + q$ by induction hypothesis. 
If we consider this equality in $P^{\rm gp}/Q^{\rm gp}$, we conclude that 
$r \in p \left(P^{\rm gp}/Q^{\rm gp}\right)$ because $P^{\rm gp}/Q^{\rm gp}$ is 
$p$-torsion free. So, by $(\star)$, there exist $r' \in P, q' \in Q$ with 
$r = pr' + q'$. Then $x = p^{k+1}r' + (p^kq' + q)$, which is a required expression. 
%Then we obtain the required claim by 
%Lemma \ref{thm38}. 
\end{proof}

Finally we recall the notion of ($p$-)saturated morphism of fine log schemes. 

\begin{definition}
\cite[Def.~II.2.10]{Tsuji} 
A morphism of fine log schemes 
$f:(X, \mathcal{M} )\to (Y , \mathcal{N})$ is called $p$-saturated if, for any 
$x \in X,  y=f(x) \in Y$, the homomorphism
$(\mathcal{N}/\mathcal{O}_Y^{*})_{\bar{y}} \to (\mathcal{M}/\mathcal{O}_X^{*})_{\bar{x}} $
induced by $f$ is $p$-saturated. $f$ is called saturated if it is $p$-saturated for any prime 
number $p$. 
\end{definition}

The following proposition immediately follows from \cite[Thm.~II.3.1]{Tsuji}. 

\begin{prop}\label{prop:sat}
Any $p$-saturated morphism of fs log schemes 
is saturated. 
\end{prop}

\subsection{Decomposition of relative log de Rham-Witt complex (II)}

In this subsection, let $R$ be a ${\mathbb Z}_{(p)}$-algebra  
and let $Q \to P$ 
be an injective $p$-quasi-saturated homomorphism of fs monoids such that 
$Q^{\rm gp}, P^{\rm gp}, P^{\rm gp}/Q^{\rm gp}$ are torsion free.
% and that 
%$(R[P],  P)/(R[Q],  \allowbreak Q)$ is log smooth. 
(Under this assumption, $Q^{\rm gp} \to P^{\rm gp}$ is also injective and 
$(R[P],  P)/(R[Q],  \allowbreak Q)$, $(R[Q],  Q)/(R,  {*})$ are log smooth.) 
In this subsection, we introduce the notion of $x$-part 
$W_m \Lambda_{(R[P], P) /(R[Q], Q),x}^{\bullet}$
of the relative log de Rham-Witt complex $W_m \Lambda_{(R[P], P) /(R[Q], Q)}^{\bullet}$ 
for $x \in P[\frac{1}{p}]$ and prove the 
direct sum decomposition 
\begin{equation*}\label{eq:decompdecomp2}
W_m\Lambda^{\bullet}_{(R[P],P)/(R,*)} 
= \bigoplus_{x \in P[\frac{1}{p}]} W_m\Lambda^{\bullet}_{(R[P],P)/(R,*), x}. 
\end{equation*}

First we recall the following proposition due to Matsuue: 

\begin{prop}\cite[Prop.~3.11]{Ma}\label{thm40}
For morphisms 
$X \to Y \to S$ of fine log schemes, we have the following exact sequence: 
 \[W_m \Lambda_{Y/S}^{1} \otimes_{W_m(\mathcal{O}_Y)} W_m\Lambda_{X/S}^{\bullet -1} \to W_m \Lambda_{X/S}^{\bullet} \to W_m\Lambda_{X/Y}^{\bullet} \to 0. \]
\end{prop}

We define the notion of $x$-part of 
$W\Lambda_{(R[P], P)/(R[Q], Q)}^{n}$, $W_m \Lambda_{(R[P], P)/(R[Q], Q)}^{n}$ 
($x \in P[\frac{1}{p}]$) in the following way: 

\begin{definition}
For $x \in P[\frac{1}{p}]$, let $W\Lambda_{(R[P],  P)/(R[Q], Q), x}^{n}$ 
be the $W(R)$-submodule of $W\Lambda_{(R[P],  P)/(R[Q], Q)}^{n}$
consisting of elements $\omega$ which has the following expression: 
\[ \omega =  \sum_{ I \subset [1,  r], |I|=n} {}^{V^{u(x)}}({\eta_{I}}X^{p^{u(x)}x}) \bigwedge_{I} d\log X_i + \sum_{ J \subset [1,  r], |J|=n-1} d^{V^{u(x)}}({\eta_{J}}X^{p^{u(x)}x}) \bigwedge_{J} d\log X_i. \]
We call $W\Lambda_{(R[P],  P)/(R[Q], Q), x}^{n}$ the $x$-part of 
$W\Lambda_{(R[P], P)/(R[Q], Q)}^{n}$. Also, we put   
\[ W_m\Lambda_{(R[P],  P)/(R[Q],  Q), x}^{n} := \mbox{Im}(W\Lambda_{(R[P],  P)/(R[Q],  Q), x}^{n} \to W_m\Lambda_{(R[P],  P)/(R[Q], Q)}^{n}) \]
and call it the $x$-part of 
$W_m\Lambda_{(R[P], P)/(R[Q], Q)}^{n}$. 
\end{definition}

By the argument in Lemma \ref{thm382}, we see that 
the derivation $d$ respects the $x$-parts. 

\begin{lemma}\label{thm26}
The exact sequence 
\begin{align} 
W_m\Lambda_{(R[Q],  Q)/(R, {*})}^{1} \otimes_{W_m(R[Q])} W_m\Lambda_{(R[P], P)/(R, {*})}^{\bullet -1} 
& \xrightarrow{\alpha}
W_m\Lambda_{(R[P],  P)/(R, {*})}^{\bullet} \nonumber \\
& \xrightarrow{\beta} W_m\Lambda_{(R[P],  P)/(R[Q], Q)}^{\bullet} \to 0 \label{eq:exseq}
\end{align}
of Proposition \ref{thm40} induces for each 
$x \in P[\frac{1}{p}]$ the exact sequence 
\begin{align*}
(W_m\Lambda_{(R[Q],  Q)/(R, {*})}^{1} \otimes_{W_m(R[Q])} W_m\Lambda_{(R[P], P)/(R, *)}^{\bullet -1})_{x} 
& \xrightarrow{\alpha_x}
W_m \Lambda_{(R[P],  P)/(R, {*}), x}^{\bullet}  \\
& \xrightarrow{\beta_x} W_m\Lambda_{(R[P],  P)/(R[Q], Q), x}^{\bullet} \to 0.  
 \end{align*}
Here, $(W_m\Lambda_{(R[Q],  Q)/(R, {*})}^{1} \otimes_{W_m(R[Q])} W_m\Lambda_{(R[P], P)/(R[Q], Q)}^{\bullet -1})_{x} $ denotes the $W_m(R)$-submodule generated by 
the tensors of elements in $W_m\Lambda_{(R[Q],  Q)/(R, {*}), x_q}^{1} \, (x_q \in Q[\frac{1}{p}])$ 
and elements in $W_m\Lambda_{(R[P], P)/(R[Q], Q), x_p}^{\bullet -1} \, (x_p \in P[\frac{1}{p}])$ 
such that $x_q + x_p = x$ in $P^{\rm gp}[\frac{1}{p}]$. 
\end{lemma}

\begin{proof}
By the definition of $x$-parts, it is easy to see that the maps $\alpha_x, \beta_x \, (x \in P[\frac{1}{p}])$ 
are well-defined as restriction of the maps $\alpha, \beta$, respectively.  
The surjectivity of the map $\beta_x$ 
%% $ W_m \Lambda_{(R[P],  P)/(R, {*}), x}^{\bullet} \to W_m\Lambda_{(R[P],  P)/(R[Q], Q), x}^{\bullet}$
follows immediately from that of the map $\beta$ and the definition of $x$-parts. 
By the exactness of \eqref{eq:exseq}, ${\rm Ker}\,\beta_x$ is the set of elements in 
$W_m \Lambda_{(R[P],  P)/(R, {*}),  x}^{\bullet}$ which belongs to 
the image of $\alpha$. Take any $\omega \in {\rm Ker}\,\beta_x$. Then, since 
%the image of the map 
%\[ W_m\Lambda_{(R[Q],  Q)/(R, {*})}^{1} \otimes_{W_m(R[Q])} W_m\Lambda_{(R[P], P)/(R[Q], Q)}^{\bullet -1} \to W_m \Lambda_{(R[P],  %P)/(R, {*})}^{\bullet}. \]
% 
\begin{align*}
& W_m\Lambda_{(R[Q],  Q)/(R, {*})}^{1} \otimes_{W_m(R[Q])} W_m\Lambda_{(R[P], P)/(R[Q], Q)}^{\bullet -1} \\ & 
\hspace{3cm} = \sum_{x' \in P[\frac{1}{p}]} (W_m\Lambda_{(R[Q],  Q)/(R, {*})}^{1} \otimes_{W_m(R[Q])} W_m\Lambda_{(R[P], P)/(R[Q], Q)}^{\bullet -1})_{x'}, 
\end{align*}
there exist $\omega'_{x'} \in (W_m\Lambda_{(R[Q],  Q)/(R, {*})}^{1} \otimes_{W_m(R[Q])} W_m\Lambda_{(R[P], P)/(R[Q], Q)}^{\bullet -1})_{x'} \, (x' \in P[\frac{1}{p}])$ with $\omega = \alpha(\sum_{x'} \omega'_{x'}) = \sum_{x'} \alpha_{x'} (\omega'_{x'})$. 
Then, since we have the decomposition 
\[ W_m \Lambda_{(R[P],  P)/(R, {*})}^{\bullet} = 
\bigoplus_{x' \in P[\frac{1}{p}]} W_m \Lambda_{(R[P],  P)/(R, {*}),  x'}^{\bullet}, \] 
we conclude that $\alpha_{x'} (\omega'_{x'}) = 0$ for $x' \not= x$ and so  
$\omega = \alpha_x(\omega'_x)$, namely, ${\rm Ker}\,\beta_x \subset {\rm Im}\,\alpha_x$. 
The inclusion in the other direction follows from the exactness of \eqref{eq:exseq}. 
So we are done. 
\end{proof}

The main result in this subsection is the following: 

\begin{prop}\label{thm8}
Let $R$ and $Q \to P$ be as in the beginning of this subsection. 
% , and fix the isomorphism $P^{\rm gp} \cong {\mathbb Z}^r$. 
\begin{enumerate}[(1)] 
\item 
%Let $x \in P[\frac{1}{p}]$ and we denote the image of $x$ in 
%$P^{\rm gp}[\frac{1}{p}] \cong {\mathbb Z}[\frac{1}{p}]^r$ by $(x_i)_i$. Then 
For $x \in P[\frac{1}{p}]$, 
we have the isomorphisms 
\[ f'_m: W_m\Lambda_{(R[P],  P)/(R[Q], Q), x}^{\bullet} \to W_m\Lambda_{(R[P^{\rm gp}], P^{\rm gp})/(R[Q], Q), x}^{\bullet}, \]
\[g'_m :  W_m\Lambda_{(R[P^{\rm gp}], P^{\rm gp})/(R[Q], Q), x }^{\bullet} \to W_m\Lambda_{(R[P],  P)/(R[Q], Q), x}^{\bullet} \]
each of which is the inverse of the other. 

\item  Any element in 
$W\Lambda_{(R[P],  P)/(R[Q], Q)}^{k}$ is written uniquely as a convergent sum of elements in 
$W\Lambda_{(R[P],  P)/(R[Q], Q),  x}^{k} \allowbreak \, (x \in P[\frac{1}{p}])$. Also, we have the canonical direct sum 
decomposition 
\[
W_m\Lambda_{(R[P],  P)/(R[Q], Q)}^{k} = \bigoplus_{x \in P[\frac{1}{p}]} W_m\Lambda_{(R[P],  P)/(R[Q], Q), x}^{k}.
\]
\end{enumerate}
\end{prop}

\begin{proof}
(1) \, The morphism $f'_m$ is defined as the map between $x$-parts of the 
map 
\[ 
W_m\Lambda_{(R[P],  P)/(R[Q], Q)}^{\bullet} \to W_m\Lambda_{(R[P^{\rm gp}], P^{\rm gp})/(R[Q], Q)}^{\bullet}
\] 
induced by $P \to P^{\rm gp}$ by  
the universality of log $F$-$V$-procomplexes. By definition, the map $f_m$ is compatible with $d$. 

Next we define the morphism $g'_m$ from the composite 
\begin{equation}\label{eq:composite}
W_m\Lambda_{(R[P^{\rm gp}], P^{\rm gp})/(R, {*}), x}^{\bullet} \xrightarrow{g_m} 
W_m \Lambda_{(R[P],  P)/(R, {*}), x}^{\bullet} \xrightarrow{\beta_x} W_m \Lambda_{(R[P],  P)/(R[Q], Q), x}^{\bullet}, 
\end{equation}
where $g_m$ is the isomorphism of Proposition \ref{thm7} and $\beta_x$ is 
the map in Proposition \ref{thm26}. By the exact sequence 
\begin{eqnarray*}
(W\Lambda_{(R[Q],  Q)/(R, {*})}^{1} \otimes_{W(R[Q])} W\Lambda_{(R[P^{\rm gp}], P^{\rm gp})/(R, {*})}^{\bullet -1})_{x}  & 
\xrightarrow{\alpha_x^{\rm gp}} &  W \Lambda_{(R[P^{\rm gp}],  P^{\rm gp})/(R, {*}), x}^{\bullet} \\
 & \xrightarrow{\beta_x^{\rm gp}} &  W\Lambda_{(R[P^{\rm gp}],  P^{\rm gp})/(R[Q], Q), x}^{\bullet} \to 0 
\end{eqnarray*}
of Proposition \ref{thm26} with $P$ replaced by $P^{\rm gp}$, it suffices to see that the composite \eqref{eq:composite} is zero on 
$$ \alpha_x^{\rm gp}((W_m\Lambda_{(R[Q],  Q)/(R, {*})}^{1} \otimes_{W_m(R[Q])} W_m\Lambda_{(R[P^{\rm gp}], P^{\rm gp})/(R, {*})}^{\bullet -1})_{x})$$ to define the map $g'_m$. 

We note that 
$(W_m\Lambda_{(R[Q],  Q)/(R, {*})}^{1} \otimes_{W_m(R[Q])} W_m\Lambda_{(R[P^{\rm gp}], P^{\rm gp})/(R, {*})}^{\bullet -1})_{x}$
is generated by the elements of the form 
$\omega_1 \otimes \omega_2$ such that 
%, where 
%$n \geq 0 ,  q_1 \in Q, q_2 \in Q^{\rm gp}, \eta \in W_{m-n} (R)$ and 
$\omega_1$ is either 
$${}^{V^{n}}(\eta X^{q_{1}}) d\log X^{q_{2}} \quad \text{or} \quad d^{V^{n}}(\eta X^{q_{1}}) $$
%$l \geq 0 ,  p_1 \in P^{\rm gp}, \zeta \in W_{m-l} (R)$ 
and $\omega_2$ is either 
$${}^{V^{l}}(\zeta X^{p_{1}}) \bigwedge_{I} d\log X_i, \quad \text{or} \quad 
d^{V^{l}}(\zeta X^{p_{1}}) \bigwedge_{J} d \log X_i, $$
where $$n \geq 0 , \, q_1 \in Q, \, q_2 \in Q^{\rm gp}, \, \eta \in W_{m-n} (R) \quad \text{and} \quad 
l \geq 0, \, p_1 \in P^{\rm gp}, \, \zeta \in W_{m-l} (R)$$ with the condition 
$\frac{q_1}{p^n} + \frac{p_1}{p^l} = x$ in $P^{\rm gp}[\frac{1}{p}]$. 
So, to define the map $g'_m$, it suffices to check that any element 
of the form $\omega_1 \otimes \omega_2$ as above is sent to $0$ by 
the composite of the map $\alpha_x^{\rm gp}$ and and the map \eqref{eq:composite}. 
\medskip 

\noindent
{\bf Case 1.} The case 
$ \omega_1 = {}^{V^{n}}(\eta X^{q_{1}}) d\log X^{q_{2}}$, $ \omega_2 =  {}^{V^{l}}(\zeta X^{p_{1}}) \bigwedge_{I} d\log X_i$. 
\begin{eqnarray*}
& & {}^{V^{n}}(\eta X^{q_{1}}) d\log X^{q_{2}} \otimes {}^{V^{l}}(\zeta X^{p_{1}}) \bigwedge_{I} d\log X_i \\ 
& \overset{\alpha_x^{\rm gp}}{\mapsto} & {}^{V^{n}}(\eta X^{q_{1}}) d\log X^{q_{2}}{}^{V^{l}}(\zeta X^{p_{1}}) \bigwedge_{I} d\log X_i \\
& = & {}^{V^{n}}(\eta X^{q_{1}}) {}^{V^{l}}(\zeta X^{p_{1}})d\log X^{q_{2}}\bigwedge_{I} d\log X_i \\
& = & (\mbox{a degree $0$ element of weight}~ x)d\log X^{q_{2}}\bigwedge_{I} d\log X_i \\
% & \mapsto & (\mbox{a degree $0$ element of weight}~ x)d\log X^{q_{2}}\bigwedge_{I} d\log X_i \\
& \mapsto & (\mbox{a degree $0$ element of weight}~ x) \, 0 \cdot \bigwedge_{I} d\log X_i \\
& = & 0. 
\end{eqnarray*}

\noindent
{\bf Case 2.} The case 
$ \omega_1 = {}^{V^{n}}(\eta X^{q_{1}}) d\log X^{q_{2}}$, $ \omega_2 = d{}^{V^{l}}(\zeta X^{p_{1}}) \bigwedge_{I} d\log X_i$. 
\begin{eqnarray*}
{}^{V^{n}}(\eta X^{q_{1}}) d\log X^{q_{2}} \otimes d{}^{V^{l}}(\zeta X^{p_{1}}) \bigwedge_{I} d\log X_i  
& \overset{\alpha_x^{\rm gp}}{\mapsto} & {}^{V^{n}}(\eta X^{q_{1}}) d\log X^{q_{2}}d{}^{V^{l}}(\zeta X^{p_{1}}) \bigwedge_{I} d\log X_i \\
& = & -{}^{V^{n}}(\eta X^{q_{1}}) d{}^{V^{l}}(\zeta X^{p_{1}})d\log X^{q_{2}}\bigwedge_{I} d\log X_i. \quad (\sharp)
\end{eqnarray*}
We divide into two cases. 
\smallskip 

\noindent
{\bf Case 2-1.} The case $n \geq l$. 

By the calculation similar to {\bf Case 1} in the proof of Proposition \ref{thm16}, 
\begin{eqnarray*}
(\sharp)& = & (\mbox{a degree $0$ element of weight}~x)d\log X^{p_{1}}d\log X^{q_{2}}\bigwedge_{I} d\log X_i \\
% & \mapsto & (\mbox{a degree $0$ weight of weight}~ x)d\log X^{p_{1}}d\log X^{q_{2}}\bigwedge_{I} d\log X_i \\
& \mapsto & (\mbox{a degree $0$ element of weight}~ x)d\log X^{q_{} }\cdot 0 \cdot \bigwedge_{I} d\log X_i \\
& = & 0. 
\end{eqnarray*}

\noindent
{\bf Case 2-2.} The case $n \leq l$. 

By the calculation similar to {\bf Case 2} in the proof of Proposition \ref{thm16}, 
{\allowdisplaybreaks
\begin{eqnarray*}
(\sharp)& = & d(\mbox{a degree $0$ element of weight}~ x)d\log X^{q_{2}}\bigwedge_{I} d\log X_i \\
& \quad & -(\mbox{a degree $0$ element of weight}~x)d\log X^{p_{1}}d\log X^{q_{2}}\bigwedge_{I} d\log X_i \\
& \mapsto & d(\mbox{a degree $0$ element of weight}~ x) \cdot 0 \cdot \bigwedge_{I} d\log X_i \\
& \quad &  -(\mbox{a degree $0$ element of weight}~ x)d\log X^{p_{1}} \cdot 0 \cdot \bigwedge_{I} d\log X_i \\
& = & 0. 
\end{eqnarray*}}
\medskip 

\noindent
{\bf Case 3.} The case $ \omega_1 = d^{V^{n}}(\eta X^{q_{1}})$, $ \omega_2 = {}^{V^{l}}(\zeta X^{p_{1}}) \bigwedge_{I} d\log X_i$. 
\begin{eqnarray*}
d^{V^{n}}(\eta X^{q_{1}})\otimes {}^{V^{l}}(\zeta X^{p_{1}}) \bigwedge_{I} d\log X_i & 
\overset{\alpha_x^{\rm gp}}{\mapsto} & d{}^{V^{n}}(\eta X^{q_{1}}){}^{V^{l}}(\zeta X^{p_{1}}) \bigwedge_{I} d\log X_i \\
& = & {}^{V^{l}}(\zeta X^{p_{1}}) d{}^{V^{n}}(\eta X^{q_{1}}) \bigwedge_{I} d\log X_i.  \quad (\flat)
\end{eqnarray*}
We divide into two cases. 
\smallskip 

\noindent
{\bf Case 3-1.} The case $l \geq n$. 

By the calculation similar to {\bf Case 1} in the proof of Proposition \ref{thm16}, 
\begin{eqnarray*}
(\flat) & = & (\mbox{a degree $0$ element of weight}~x)d\log X^{q_{1}} \bigwedge_{I} d\log X_i \\
% & \mapsto & (\mbox{a degree $0$ element of weight}~ x)d\log X^{q_{1}}\bigwedge_{I} d\log X_i \\
& \mapsto & (\mbox{a degree $0$ element of weight}~ x) \cdot 0 \cdot \bigwedge_{I} d\log X_i \\
& = & 0. 
\end{eqnarray*}

\noindent
{\bf Case 3-2.} The case $l \leq n$. 

We use the assumption $q_1 + p_1 p^{n-l} = p^n x$. 
Since $P^{\rm gp}/Q^{\rm gp}$ is torsion free and $Q \hookrightarrow P$ is quasi-$p$-saturated, 
there exist $p_3 \in P , q_3 \in Q$ with $q_3 + p_3 p^{n-l} = p^n x$ by $(\star \star)$. 
Also, since $P^{\rm gp}$ is also torsion free, $d \log X^{p_1} = d \log X^{p_3} \in W_m\Lambda_{(R[P], P)/(R[Q], Q)}^{1}$. 
Then, by the calculation similar to {\bf Case 2} in the proof of Proposition \ref{thm16}, 
{\allowdisplaybreaks
\begin{eqnarray*}
(\flat) & = & d {}^{V^{n}}(\eta {}^{F^{n-l}}\zeta p^{l} X^{p^n x })\bigwedge_{I} d\log X_i -
{}^{V^{n}}(\eta {}^{F^{n-l}}\zeta X^{p^n x}) d\log X^{p_{1}} \bigwedge_{I} d\log X_i \\
%& \mapsto &  d {}^{V^{n}}(\eta {}^{F^{n-l}}\zeta p^{l} X^{p^n x})\bigwedge_{I} d\log X_i -
%{}^{V^{n}}(\eta {}^{F^{n-l}}\zeta X^{p^n x}) d\log X^{p_{1}} \bigwedge_{I} d\log X_i \\
& \mapsto &  d {}^{V^{n}}(\eta {}^{F^{n-l}}\zeta p^{l} X^{p^n x})\bigwedge_{I} d\log X_i -
{}^{V^{n}}(\eta {}^{F^{n-l}}\zeta X^{p^n x}) d\log X^{p_{1}} \bigwedge_{I} d\log X_i \\
& = &  d {}^{V^{n}}(\eta {}^{F^{n-l}}\zeta p^{l} X^{p^n x})\bigwedge_{I} d\log X_i -
{}^{V^{n}}(\eta {}^{F^{n-l}}\zeta X^{p^n x}) d\log X^{p_{3}} \bigwedge_{I} d\log X_i \\
& = &  d {}^{V^{n}}(\eta {}^{F^{n-l}}\zeta X^{q_3 + p_3 p^{n-l}}) \bigwedge_{I} d\log X_i -
{}^{V^{n}}(\eta {}^{F^{n-l}}\zeta X^{p^n x}) d\log X^{p_{3}} \bigwedge_{I} d\log X_i \\
%& = &  {}^{V^{n}}(\eta {}^{F^{n-l}}\zeta X^{q_3}) d{}^{V^{l}}( X^{p_3})\bigwedge_{I} d\log X_i -
%{}^{V^{n}}(\eta {}^{F^{n-l}}\zeta X^{p^n x}) d\log X^{p_{3}} \bigwedge_{I} d\log X_i \\
%& = &  {}^{V^{n}}(\eta {}^{F^{n-l}}\zeta X^{p^n x}) d\log X^{p_{3}} \bigwedge_{I} d\log X_i -
%{}^{V^{n}}(\eta {}^{F^{n-l}}\zeta X^{p^n x}) d\log X^{p_{3}} \bigwedge_{I} d\log X_i \\
& = & {}^{V^{l}}(\zeta X^{p_{3}}) d{}^{V^{n}}(\eta X^{q_{3}}) \bigwedge_{I} d\log X_i \\
& = & {}^{V^{l}}(\zeta X^{p_{3}}) \cdot 0 \cdot \bigwedge_{I} d\log X_i \\
& = & 0. 
\end{eqnarray*}}% 

\noindent
{\bf Case 4.} The case 
$ \omega_1 =  d^{V^{n}}(\eta X^{q_{1}})$, $ \omega_2 =  d^{V^{m}}(\zeta X^{p_{1}}) \bigwedge_{I} d\log X_i$. 
\begin{eqnarray*}
d^{V^{n}}(\eta X^{q_{1}})\otimes d^{V^{l}}(\zeta X^{p_{1}}) \bigwedge_{I} d\log X_i 
& \overset{\alpha_x^{\rm gp}}{\mapsto} & d{}^{V^{n}}(\eta X^{q_{1}})d^{V^{l}}(\zeta X^{p_{1}}) \bigwedge_{I} d\log X_i. \quad (\natural)
\end{eqnarray*}
We divide into two cases. 
\smallskip 

\noindent
{\bf Case 4-1.} The case $l \geq n$. 

By the calculation similar to {\bf Case 1} in the proof of Proposition \ref{thm16}, 
{\allowdisplaybreaks{
\begin{eqnarray*}
(\natural)& = & d(\mbox{a degree $0$ element of weight}~x)d\log X^{q_{1}} \bigwedge_{I} d\log X_i \\
% & \mapsto & d(\mbox{a degree $0$ element of weight}~ x)d\log X^{q_{1}}\bigwedge_{I} d\log X_i \\
& \mapsto & d(\mbox{a degree $0$ element of weight}~ x) \cdot 0 \cdot \bigwedge_{I} d\log X_i \\
& = & 0. 
\end{eqnarray*}}}%

\noindent
{\bf Case 4-2.} The case $l \leq n$. 

We use the assumption $q_1 + p_1 p^{n-l} = p^n x$. 
Since $P^{\rm gp}/Q^{\rm gp}$ is torsion free and $Q \hookrightarrow P$ is quasi-$p$-saturated, 
there exist $p_3 \in P , q_3 \in Q$ with $q_3 + p_3 p^{n-l} = p^n x$ by $(\star \star)$. 
Also, since $P^{\rm gp}$ is also torsion free, $d \log X^{p_1} = d \log X^{p_3} \in W_m\Lambda_{(R[P], P)/(R[Q], Q)}^{1}$. 
Then, by the calculation similar to {\bf Case 2} in the proof of Proposition \ref{thm16}, 
{\allowdisplaybreaks
\begin{eqnarray*}
(\natural) & = & d{}^{V^{n}}(\eta {}^{F^{n-l}}\zeta p^{l} X^{p^n x}) d\log X^{p_{1}}\bigwedge_{I} d\log X_i \\
% & \mapsto & d{}^{V^{n}}(\eta {}^{F^{n-l}}\zeta p^{l} X^{p^n x}) d \log X^{p_{1}}\bigwedge_{I} d\log X_i \\
& \mapsto & d{}^{V^{n}}(\eta {}^{F^{n-l}}\zeta p^{l} X^{p^n x}) d \log X^{p_{1}}\bigwedge_{I} d\log X_i \\
& = & d{}^{V^{n}}(\eta {}^{F^{n-l}}\eta p^{l} X^{p^n x}) \log X^{p_{3}}\bigwedge_{I} d\log X_i \\
& = & d{}^{V^{n}}(\eta {}^{F^{n-l}}\zeta p^{l} X^{q_3 + p^{n-l} p_3}) d \log X^{p_{3}}\bigwedge_{I} d\log X_i \\
& = & d{}^{V^{n}}(\eta X^{q_{3}})d^{V^{l}}(\zeta X^{p_{3}}) \bigwedge_{I} d\log X_i \\
%& = & d\left( {}^{V^{n}}(\eta {}^{F^{n-l}}\zeta X^{q_3}){}^{V^{l}}(X^{p_3}) \right) d \log X^{p_{3}}\bigwedge_{I} d\log X_i \\
%& = & {}^{V^{n}}(\eta {}^{F^{n-l}}\zeta X^{q_3}) d^{V^{l}}(X^{p_3}) d \log X^{p_{3}}\bigwedge_{I} d\log X_i \\
& = & 0 \cdot d^{V^{l}}(\zeta X^{p_{3}}) \bigwedge_{I} d\log X_i \\
& = & 0. 
\end{eqnarray*}
}
By the above calculation, we have checked that the morphism $g'_m$ is well-defined. 
Note that the isomorphism $f_m$ in Proposition \ref{thm7} and the morphism $f'_m$ fit into 
the commutative diagram with exact horizontal lines 
\[ \xymatrix{
W_m \Lambda_{(R[P],  P)/(R, {*}), x}^{\bullet} \ar[d]_{f_m} \ar[r]_{\beta_x} & W_m \Lambda_{(R[P],  P)/(R[Q], Q), x}^{\bullet} \ar[d]_{f'_m} \ar[r] & 0\\
W_m \Lambda_{(R[P^{\rm gp}], P^{\rm gp})/(R, {*}), x}^{\bullet}  \ar[r]_{\beta_x^{\rm gp}} & W_m \Lambda_{(R[P^{\rm gp}], P^{\rm gp})/(R[Q], Q), x}^{\bullet} \ar[r]  & 0, } 
\]
and that there exists a similar commutative diagram with two horizontal lines exchanged and 
morphisms $f_m, f'_m$ replaced by morphisms $g_m, g'_m$. 
%the morphism $f'_m$ is surjective. By the same reason, $g'_m$ is also surjective. 
Then, we can check the equalities 
$f'_m \circ g'_m = {\rm id}$,  $g'_m \circ f'_m = {\rm id}$ from 
the equalities $f_m \circ g_m = {\rm id}$,  $g_m \circ f_m = {\rm id}$, the above commutative 
diagram and the similar diagram with $g_m$, $g'_m$. So we are done. 

(2) \, 
We see easily that $\{W_m\Lambda_{(R[P^{\rm gp}],  P^{\rm gp})/(R[Q], Q)}^{\bullet}\}_{m \in {\mathbb N}}$ 
has a structure of log $F$-$V$-procomplex over an $(R[Q^{\rm gp}], Q^{\rm gp})$-algebra 
$(R[P^{\rm gp}],  P^{\rm gp})$, and by the universality, we have the canonical isomorphism of 
log $F$-$V$-procomplexes 
%By Proposition-Definition \ref{thm15}, 
%we have the isomorphism 
\[ \{W_m\Lambda_{(R[P^{\rm gp}],  P^{\rm gp})/(R[Q], Q)}^{\bullet}\}_{m \in {\mathbb N}} \cong 
\{W_m\Lambda_{(R[P^{\rm gp}],  P^{\rm gp})/(R[Q^{\rm gp}], Q^{\rm gp})}^{\bullet}\}_{m \in {\mathbb N}}. \]
Hence we have the isomorphism 
\[ W_m\Lambda_{(R[P^{\rm gp}],  P^{\rm gp})/(R[Q], Q), x}^{\bullet} \cong 
W_m\Lambda_{(R[P^{\rm gp}],  P^{\rm gp})/(R[Q^{\rm gp}], Q^{\rm gp}),x}^{\bullet}. \]
Also, since the map $Q^{\rm gp} \to P^{\rm gp}$ is injective and 
$P^{\rm gp}, Q^{\rm gp}, P^{\rm gp}/Q^{\rm gp}$ are torsion free, there exists an isomorphism 
\begin{equation}\label{eq:pq-decomp}
P^{\rm gp} \cong Q^{\rm gp} \oplus (P^{\rm gp}/Q^{\rm gp}),% \cong {\mathbb{Z}}^s \oplus {\mathbb{Z}}^{r-s}. 
\end{equation}
which induces an isomorphism 
\begin{align*}
W_m \Lambda_{(R[P^{\rm gp}],  P^{\rm gp})/(R[Q^{\rm gp}], Q^{\rm gp})}^{\bullet} 
& \xrightarrow{\cong} W_m\Lambda_{(R[Q^{\rm gp}][P^{\rm gp}/Q^{\rm gp}],  
Q^{\rm gp} \oplus (P^{\rm gp}/Q^{\rm gp}))/(R[Q^{\rm gp}], Q^{\rm gp})}^{\bullet} \\ 
& \xleftarrow{\cong} 
W_m\Lambda_{(R[Q^{\rm gp}][P^{\rm gp}/Q^{\rm gp}],  P^{\rm gp}/Q^{\rm gp})/(R[Q^{\rm gp}], {*})}^{\bullet}. 
\end{align*}
For $x \in P^{\rm gp}[\frac{1}{p}]$, denote the image of $x$ by the isomorphism 
\[P^{\rm gp}[\tfrac{1}{p}] \cong Q^{\rm gp}[\tfrac{1}{p}] \oplus (P^{\rm gp}/Q^{\rm gp})[\tfrac{1}{p}] \] 
% \cong {\mathbb{Z}}[\tfrac{1}{p}]^s \oplus {\mathbb{Z}}[\tfrac{1}{p}]^{r-s}\] 
induced by \eqref{eq:pq-decomp} by $x'+ x''$. 
% = (x_i)_{i=1}^{s} + (x_i)_{i=s+1}^r$. 
Also, we put 
{\allowdisplaybreaks{
\begin{multline*}
W_m\Lambda_{(R[Q^{\rm gp}][P^{\rm gp}/Q^{\rm gp}],  P^{\rm gp}/Q^{\rm gp})/(R[Q^{\rm gp}], {*}), x}^{\bullet} \\
{} := \mbox{Im} \left( W_m \Lambda_{(R[P^{\rm gp}],  P^{\rm gp})/(R[Q^{\rm gp}], Q^{\rm gp}),  x}^{\bullet} \to W_m\Lambda_{(R[Q^{\rm gp}][P^{\rm gp}/Q^{\rm gp}],  P^{\rm gp}/Q^{\rm gp})/(R[Q^{\rm gp}], {*})}^{\bullet} \right). 
\end{multline*}}}%
Beware that the right hand side is not the $x$-part of 
$W_m\Lambda_{(R[Q^{\rm gp}][P^{\rm gp}/Q^{\rm gp}],  P^{\rm gp}/Q^{\rm gp})/(R[Q^{\rm gp}], {*})}^{\bullet}$ but 
a subset of the $x''$-part of it. By definition, we have an isomorphism 
\[ f''_m: W_m \Lambda_{(R[P^{\rm gp}],  P^{\rm gp})/(R[Q^{\rm gp}], Q^{\rm gp}),  x}^{\bullet} \overset{\cong}{\to} 
W_m\Lambda_{(R[Q^{\rm gp}][P^{\rm gp}/Q^{\rm gp}],  P^{\rm gp}/Q^{\rm gp})/(R[Q^{\rm gp}], {*}), x}^{\bullet}.  
\]

An element in 
$W_m\Lambda_{(R[Q^{\rm gp}][P^{\rm gp}/Q^{\rm gp}],  P^{\rm gp}/Q^{\rm gp})/(R[Q^{\rm gp}], {*}), x}^{\bullet}$ is written as a sum of the images of 
$P^{\rm gp}$-basic Witt differentials in 
$W_m \Lambda_{(R[P^{\rm gp}],  P^{\rm gp})/(R[Q^{\rm gp}], Q^{\rm gp}),  x}^{\bullet}$ by 
$f'_m$. As for the $P^{\rm gp}$-basic Witt differentials in 
$W_m \Lambda_{(R[P^{\rm gp}],  P^{\rm gp})/(R[Q^{\rm gp}], Q^{\rm gp}),  x}^{\bullet}$, 
we have the following equalities (note that $u(x) = \min \{u(x'), u(x'') \}$): 
\begin{align*}
& {}^{V^{u(x)}}(\eta X^{p^{u(x)}x}) \bigwedge_I d\log X_i = {}^{V^{u(x)}}(\eta X^{p^{u(x)}x'} X^{p^{u(x)}x''}) \bigwedge_I d\log X_i \\ 
= \, & 
\left\{ \begin{array}{ll}
{}^{V^{u(x'')}}(\eta X^{p^{u(x)}x'} X^{p^{u(x'')}x''}) \bigwedge_I d\log X_i & (u(x') \leq u(x'') ), \\
{}^{V^{u(x'')}}({}^{V^{u(x')-u(x'')}}(\eta X^{p^{u(x')}x'}) X^{p^{u(x'')}x''}) \bigwedge_I d\log X_i & (u(x') \geq u(x'')), 
\end{array} \right. 
\end{align*}
\begin{align*}
& d{}^{V^{u(x)}}(\eta X^{p^{u(x)}x}) \bigwedge_I d\log X_i = d{}^{V^{u(x)}}(\eta X^{p^{u(x)}x'} X^{p^{u(x)}x''}) \bigwedge_I d\log X_i \\ 
= \, & 
\left\{ \begin{array}{ll}
d{}^{V^{u(x'')}}(\eta X^{p^{u(x)}x'} X^{p^{u(x'')}x''}) \bigwedge_I d\log X_i & (u(x') \leq u(x'') ), \\
d{}^{V^{u(x'')}}({}^{V^{u(x')-u(x'')}}(\eta X^{p^{u(x')}x'}) X^{p^{u(x'')}x''}) \bigwedge_I d\log X_i & (u(x') \geq u(x'')). 
\end{array} \right. 
\end{align*}
Thus we see that 
an element in 
$W_m\Lambda_{(R[Q^{\rm gp}][P^{\rm gp}/Q^{\rm gp}],  P^{\rm gp}/Q^{\rm gp})/(R[Q^{\rm gp}], {*}), x}^{\bullet}$ is written as a sum of the elements of the form 
\begin{equation}\label{eq:pq1}
{}^{V^{u(x'')}}(\eta' X^{p^{u(x'')}x''}) \bigwedge_I d\log X_i
\end{equation} 
or the form 
\begin{equation}\label{eq:pq2}
d{}^{V^{u(x'')}}(\eta' X^{p^{u(x'')}x''}) \bigwedge_I d\log X_i
\end{equation} 
such that $\xi := {}^{V^{u(x'')}}\eta'$ is of the form 
\begin{equation}\label{eq:pq3}
\xi = 
{}^{V^{u(x)}}(\eta X^{p^{u(x)}x'}) = 
{}^{V^{u(x')}}({}^{V^{u(x)-u(x')}}\eta X^{p^{u(x')}x'})
\in W_m(R[Q^{\rm gp}]).
\end{equation}
The elements \eqref{eq:pq1}, \eqref{eq:pq2} are 
$P^{\rm gp}/Q^{\rm gp}$-basic Witt differentials in 
$W_m\Lambda_{(R[Q^{\rm gp}][P^{\rm gp}/Q^{\rm gp}],  P^{\rm gp}/Q^{\rm gp})/(R[Q^{\rm gp}], {*})}^{\bullet}$ with weight $x''$ such that $\xi$ is of the form \eqref{eq:pq3}. 
Then, by claim 1 in the proof of Proposition \ref{thm2}, 
we see that any element in 
$W_m\Lambda_{(R[Q^{\rm gp}][P^{\rm gp}/Q^{\rm gp}],  P^{\rm gp}/Q^{\rm gp})/(R[Q^{\rm gp}], {*}), x}^{\bullet}$ is written as a sum of 
basic Witt differentials $\hat{e}(\xi , k , \mathcal{P})$ with $k^+ = x''$ and 
$\xi$ as in the form \eqref{eq:pq3}. 
If we fix $x''$ and make $x'$ vary, the weight of $\xi$ in \eqref{eq:pq3} 
as $Q^{\rm gp}$-basic Witt differential of degree $0$ varies. So we conclude that 
the $x''$-part of $W_m\Lambda_{(R[Q^{\rm gp}][P^{\rm gp}/Q^{\rm gp}],  P^{\rm gp}/Q^{\rm gp})/(R[Q^{\rm gp}], {*})}^{\bullet}$ is written as the direct sum 
\begin{align} \label{3.2.4.2}
 \bigoplus_{y = y'+y'' \in P^{\rm gp}[\frac{1}{p}], \atop y''=x''}
W_m\Lambda_{(R[Q^{\rm gp}][P^{\rm gp}/Q^{\rm gp}],  P^{\rm gp}/Q^{\rm gp})/(R[Q^{\rm gp}], {*}), y}^{\bullet}. % \tag{3.2.4.2}
\end{align}
Then, making $x''$ vary, we obtain the direct sum decomposition  
\[ W_m\Lambda_{(R[Q^{\rm gp}][P^{\rm gp}/Q^{\rm gp}],  P^{\rm gp}/Q^{\rm gp})/(R[Q^{\rm gp}], {*})}^{\bullet} = 
\bigoplus_{x \in P^{\rm gp}[\frac{1}{p}]}
W_m\Lambda_{(R[Q^{\rm gp}][P^{\rm gp}/Q^{\rm gp}],  P^{\rm gp}/Q^{\rm gp})/(R[Q^{\rm gp}], {*}), x}^{\bullet}.\] 
For $x \in P[\frac{1}{p}]$, we put 
\[ \tilde{f}_m := 
f''_m \circ f'_m: 
W_m \Lambda_{(R[P],  P)/(R[Q], Q), x}^{\bullet} \overset{\cong}{\to} W_m\Lambda_{(R[Q^{\rm gp}][P^{\rm gp}/Q^{\rm gp}],  P^{\rm gp}/Q^{\rm gp})/(R[Q^{\rm gp}], {*}), x}^{\bullet}. \]

Now we give the proof of the assertion (2), imitating 
the last step of the proof of Proposition \ref{thm22}. 
We only have to prove the uniqueness of the expression and so 
it suffices to prove that, if we have an equality
\[ 0 = \sum_{x \in P[\frac{1}{p}]} \omega_x \] 
 in $W_m\Lambda_{(R[P],  P)/(R[Q],  Q)}^{k}$ 
where the right hand side is a finite sum 
with $\omega_x \in W_m\Lambda_{(R[P],  P)/(R[Q], Q), x}^{k}$, 
all $\omega_x$'s are equal to $0$. 

Consider the following commutative diagram
\[ \xymatrix{
\bigoplus_{x \in P[\frac{1}{p}]} W_m\Lambda_{(R[P],  P)/(R[Q], Q), x}^{\bullet} \ar[r]^-{\oplus \tilde{f}_m}_-{\cong} \ar[d]^{{\rm sum}} & \bigoplus_{x \in P[\frac{1}{p}]} 
W_m \Lambda_{(R[Q^{\rm gp}][P^{\rm gp}/Q^{\rm gp}], 
P^{\rm gp}/Q^{\rm gp})/(R[Q^{\rm gp},  {*}), x}^{\bullet} \ar[d]^{{\rm sum}} \\
  W_m\Lambda_{(R[P],  P)/(R[Q], Q)}^{\bullet}  \ar[r]^-j & 
W_m\Lambda_{(R[Q^{\rm gp}][P^{\rm gp}/Q^{\rm gp}], P^{\rm gp}/Q^{\rm gp})/(R[Q^{\rm gp}],  {*})}^{\bullet}, 
  }
\]
where $\oplus \tilde{f}_m$ is the direct sum of the isomorphisms $\tilde{f}_m$, 
$j$ is the composite 
\begin{multline*}
W_m\Lambda_{(R[P],  P)/(R[Q],  Q)}^{\bullet}  \to W_m\Lambda_{(R[P^{\rm gp}], P^{\rm gp})/(R[Q],  Q)}^{\bullet} \\
{} \xrightarrow{\cong} W_m\Lambda_{(R[P^{\rm gp}], P^{\rm gp})/(R[Q^{\rm gp}], Q^{\rm gp})}^{\bullet} \xrightarrow{\cong}  W_m\Lambda_{(R[Q^{\rm gp}][P^{\rm gp}/Q^{\rm gp}], P^{\rm gp}/Q^{\rm gp})/(R[Q^{\rm gp}],  {*})}^{\bullet} \end{multline*}
and ${\rm sum}$ are the morphisms of taking sum.  
By assumption, the image of the element $$(\omega_x)_x \in 
\bigoplus_{x \in P[\frac{1}{p}]} W_m\Lambda_{(R[P],  P)/(R[Q], Q), x}^{\bullet}$$ 
by $j \circ {\rm sum} = {\rm sum} \circ (\oplus \tilde{f}_m)$ is equal to $0$. 
Also, as we saw above, the morphism 
${\rm sum}$ on the right is injective. 
Thus we see that $(\tilde{f}_m(\omega_x))_x = (0)_x$ and so 
$\omega_x = 0$ for all $x \in P[\frac{1}{p}]$. 
So we are done. 
\end{proof}

\subsection{Comparison isomorphism (II)}

%In this subsection, we finish the proof of our first main comparison 
%isomorphism between the complex computing the relative log crystalline 
%cohomology and the relative log de Rham-Witt complex. 
%
%First we prove the comparison isomorphism 
%of relative de Rham complex and relative log de Rham-Witt complex 
%for certain pre-log rings. 
Let $R$ be a commutative ring in which $p$ is nilpotent and let 
$Q \to P$ be an injective $p$-saturated homomorphism of fs monoids 
such that $Q^{\rm gp}, P^{\rm gp}, P^{\rm gp}/Q^{\rm gp}$ are torsion free. 
% and that $(R[P],  P)/(R[Q],  Q)$ is log smooth. 
(Under this assumption, 
$Q^{\rm gp} \to P^{\rm gp}$ is also injective and $(R[P],  P)/(R[Q],  Q)$, $(R[Q],  Q)/(R,  {*})$ are log smooth.) 

Let $((A_m, P), \phi_m, \delta_m)_m$ be the log Frobenius lift 
of $(R[P],P)$ over $(R[Q],Q)$ in 
Example \ref{exam} (so $A_m = W_m(R[Q]) \otimes_{\mathbb{Z}[Q]} \mathbb{Z}[P]$), and let 
$(Z_m, \Phi_m, \Delta_m)_m$ be
the corresponding log Frobenius lift (as log schemes) of 
$Z := \mbox{Spec}(R[P], P)$ over 
$\mbox{Spec}(R[Q], Q)$. 
Then, by the construction in Section 1.3 using $(Z_m, \Phi_m, \Delta_m)_m$, 
we obtain the comparison morphism 
\begin{align*}
 \Lambda_{Z_{m}/{{\rm Spec}}(W_{m}(R[Q],  Q))}^{\bullet} \to W_{m} \Lambda_{Z/{{\rm Spec}}(R[Q], Q)}^{\bullet} % \tag{3.3.1}
\end{align*}
of complexes of quasi-coherent sheaves, hence the morphism 
\begin{align} \label{3.3.1}
c_{P,Q}:  \Lambda_{(A_m, P)/W_{m}(R[Q],  Q)}^{\bullet} \to W_{m} \Lambda_{(R[P],P)/(R[Q], Q)}^{\bullet}
 % \tag{3.3.1}
\end{align}
of complex of $W_m(R[Q])$-modules. 

\begin{theorem}\label{qis1}
In the above situation, the morphism 
$c_{P,Q}$ is a quasi-isomorphism. 
\end{theorem}

To prove Theorem \ref{qis1}, first we prove the following lemma. 

\begin{lemma}\label{thm41}
Let $R$ and $Q \to P$ be as above. Then the canonical map 
\[ \Lambda_{(W_m(R[Q]) \otimes_{\mathbb{Z}[Q]} \mathbb{Z}[P], P)/(R[Q], Q)}^{\bullet} \to \Lambda_{(W_m(R[Q]) \otimes_{\mathbb{Z}[Q]} \mathbb{Z}[P^{\rm gp}], P^{\rm gp})/W_m(R[Q], Q)}^{\bullet} \]
is injective. 
\end{lemma}

\begin{proof} 
We put $A_m := W_m(R[Q]) \otimes_{\mathbb{Z}[Q]} \mathbb{Z}[P]$ as above and 
put $A'_m := W_m(R[Q]) \otimes_{\mathbb{Z}[Q]}  \mathbb{Z}[P^{\rm gp}]$. 
Since 
\begin{align*}
\Lambda_{(W_m(R[Q]) \otimes_{\mathbb{Z}[Q]} \mathbb{Z}[P], P)/W_m(R[Q], Q)}^{\bullet} & \cong \left( W_m(R[Q]) \otimes_{\mathbb{Z}[Q]} \mathbb{Z}[P] \right) \otimes_{\mathbb{Z}} \bigwedge^{\bullet}(P^{\rm gp}/Q^{\rm gp}), \\ 
\Lambda_{(W_m(R[Q]) \otimes_{\mathbb{Z}[Q]} \mathbb{Z}[P^{\rm gp}], P^{\rm gp})/W_m(R[Q], Q)}^{\bullet} & 
\cong \left( W_m(R[Q]) \otimes_{\mathbb{Z}[Q]} \mathbb{Z}[P^{\rm gp}] \right) \otimes_{\mathbb{Z}} \bigwedge^{\bullet}(P^{\rm gp}/Q^{\rm gp}), 
\end{align*}
it suffices to prove that the canonical map $A_m \to A'_m$ is injective. 

By Proposition \ref{thm11}, any element in 
$A_m$ is written as a sum of the elements of the form 
${}^{V^{u(q)}} \left( \eta X^{p^{u(q)}q} \right) \otimes T^r$ with $q \in Q[\frac{1}{p}], r \in P, \eta \in W_{m-u(q)}(R)$. 
Consider the following claim: 
\bigskip 

\noindent
{\bf claim.} Let $q, q' \in Q[\frac{1}{p}], r,  r' \in P$ with $q +r = q' +r' \in P[\frac{1}{p}]$. 
Then we have the equality 
${}^{V^{u(q)}} \left( \eta X^{p^{u(q)}q} \right) \otimes T^r = {}^{V^{u(q')}} \left( \eta X^{p^{u(q')}q'} \right) \otimes T^{r'}$
in $A_m$. 
\bigskip 

Suppose that the claim is true and assume we are given an element  
\[ \omega = \sum_{i=1}^n {}^{V^{u(q_i)}} \left( \eta_i X^{p^{u(q_i)}q_i} \right) \otimes T^{r_i} \] 
in $A_m$ which is sent to $0 \in A'_m$. We may assume by the claim that 
the values $p_i := q_i + r_i \, (1 \leq i \leq n)$ are all different. If we define 
$\delta'_m : A'_m = W_m(R[Q]) \otimes_{\mathbb{Z}[Q]} \mathbb{Z}[P^{\rm gp}] \to W_m(R[P^{\rm gp}])$
is the extension of the canonical morphism 
$W_m(R[Q]) \to W_m(R[P^{\rm gp}])$ which sends $1 \otimes T^x (x \in P^{\rm gp})$ to $X^x$, 
we have the equality 
\[ 0 = \delta'_m(\omega) =  \sum_{i=1}^m 
\delta'_m \left( {}^{V^{u(q_i)}} \left( \eta_i X^{p^{u(q_i)}q_i} \right) \otimes T^{r_i} \right)  = 
\sum_{i=1}^n {}^{V^{u(q_i)}} \left( \eta_i X^{p^{u(q_i)}p_i} \right). \]
Since the cokernel of $Q^{\rm gp} \to P^{\rm gp}$ is $p$-torsion free, we have 
$u_P(p_i) = u_P(q_i) = u_Q(q_i)$. Hence the expression on the right hand side of the above 
equality is unique by Lemma \ref{thm11}. 
Thus we have $\eta_i = 0$ for all $i$ and so $\omega=0$, namely, the map $A_m \to A'_m$ is injective 
as required. Therefore, it suffices to prove the above claim. 
 
Let 
$q, q' \in Q[\frac{1}{p}], r,  r' \in P$ with $q +r = q' +r' \in P[\frac{1}{p}]$. 
As we saw above, 
$u(q) := u_Q(q) = u_P(q) = u_P(q+r)$ and so $u(q) = u(q')$. Thus 
$ p^{u(q)}q + p^{u(q)}r = p^{u(q)}q'  + p^{u(q)} r' \in P$. 
Since $Q \to P$ is an integral morphism, there exist elements $a, a' \in Q, b \in P$ 
satisfying the following: 
\begin{eqnarray*}
p^{u(q)}r & = & a + b, \\
p^{u(q)}r' & = & a' + b, \\
p^{u(q)}q + a & = & p^{u(q)}q' + a'. 
\end{eqnarray*}
If we regard $b \in P$ as an element in $P^{\rm gp}/Q^{\rm gp}$, it belongs to 
$p^{u(q)}\left(P^{\rm gp}/Q^{\rm gp} \right)$. Hence, by $(\star \star)$, 
there exist elements $\tilde{r} \in P , \tilde{q} \in Q$ such that $b = p^{u(q)}\tilde{r} + \tilde{q}$. 
This implies the equality $p^{u(q)}r = a  + \tilde{q} + p^{u(q)}\tilde{r}$ and so 
$a  + \tilde{q} \in Q \cap p^{u(q)}P^{\rm gp}$. Since the cokernel of the map $Q^{\rm gp} \to P^{\rm gp}$ is $p$-torsion free, 
$Q \cap p^{u(q)}P^{\rm gp} = Q \cap Q^{\rm gp} \cap p^{u(q)}P^{\rm gp} = Q \cap p^{u(q)}Q^{\rm gp}$. 
Moreover, since $Q$ is saturated, $Q \cap p^{u(q)}Q^{\rm gp} = p^{u(q)}Q$. 
Thus there exists an element $c$ in $Q$ such that 
$a  + \tilde{q} = p^{u(q)}c$. Since $P^{\rm gp}$ is $p$-torsion free, this implies the equality $r = c + \tilde{r}$. 
By the same argument, there exists an element $c'$ in $Q$ such that 
$a' + \tilde{q} = p^{u(q)}c'$. Since $P^{\rm gp}$ is $p$-torsion free, we obtain the equality $r' = c' + \tilde{r}$. Hence 
\begin{align*}
{}^{V^{u(q)}} \left( \eta X^{p^{u(q)}q} \right) \otimes T^r  & =  {}^{V^{u(q)}} \left( \eta X^{p^{u(q)}q} \right) X^c \otimes T^{\tilde{r}} 
 =   {}^{V^{u(q)}} \left( \eta X^{p^{u(q)}q+p^{u(q)}c} \right) \otimes T^{\tilde{r}} \\ 
&  =   {}^{V^{u(q)}} \left( \eta X^{p^{u(q)}q+a+\tilde{q}} \right) \otimes T^{\tilde{r}} 
 =   {}^{V^{u(q)}} \left( \eta X^{p^{u(q)}q'+a'+\tilde{q}} \right) \otimes T^{\tilde{r}} \\ 
&  =   {}^{V^{u(q)}} \left( \eta X^{p^{u(q)}q'} \right) \otimes T^{r'} 
 =   {}^{V^{u(q')}} \left( \eta X^{p^{u(q')}q'} \right) \otimes T^{r'},  
\end{align*}
and so the claim is proved. So we are done. 
\end{proof}

Now we give a proof of Theorem \ref{qis1}, using Lemma \ref{thm41}. 

\begin{proof}[Proof of Theorem \ref{qis1}]
We use the notation in the proof of Lemma \ref{thm41}. Then we have the 
following commuative diagram
\[ \xymatrix{
\Lambda_{(A_m, P)/W_m(R[Q], Q)}^{\bullet} \ar[r]^{c_{P,Q}} \ar[d] & W_m\Lambda_{(R[P], P )/(R[Q],  Q)}^{\bullet} \ar[d] \\
\Lambda_{(A'_m, P^{\rm gp})/W_m(R[Q], Q)}^{\bullet} \ar[r]^{c_{P^{\rm gp},Q}} \ar[d]_{\cong} & W_m\Lambda_{(R[P^{\rm gp}], P^{\rm gp})/(R[Q],  Q)}^{\bullet} \ar[d]_{\cong} \\
\Lambda_{(W_m(R[Q^{\rm gp}]) \otimes_{W_m(R)[Q^{\rm gp}]} W_m(R)[P^{\rm gp}], P^{\rm gp} )/W_m(R[Q^{\rm gp}], Q^{\rm gp})}^{\bullet} \ar[r]^-{c_{P^{\rm gp},Q^{\rm gp}}}  \ar[d]_{\cong} & W_m\Lambda_{(R[P^{\rm gp}], P^{\rm gp} )/(R[Q^{\rm gp}],  Q^{\rm gp})}^{\bullet} \ar[d]_{\cong} \\
\Lambda_{(W_m(R[Q^{\rm gp}])[P^{\rm gp}/Q^{\rm gp}], P^{\rm gp}/Q^{\rm gp})/W_m(R[Q^{\rm gp}], *)}^{\bullet} \ar[r]^{c_{P^{\rm gp}/Q^{\rm gp}}} & W_m\Lambda_{(R[Q^{\rm gp}][P^{\rm gp}/Q^{\rm gp}], P^{\rm gp}/Q^{\rm gp})/(R[Q^{\rm gp}], *)}^{\bullet}, 
}\]
where the horizontal arrows are comparison morphisms and the 
vertical arrows are defined by the universality. The arrows with the symbol $\cong$ 
are isomorphisms. Also, 
By Lemma \ref{thm41}, the top left vertical arrow 
$\Lambda_{(A_m, P)/W_m(R[Q], Q)}^{\bullet} \to \Lambda_{(A'_m, P^{\rm gp})/W_m(R[Q], Q)}^{\bullet}$ is injective. 
Also, the map $c_{P^{\rm gp}/Q^{\rm gp}}$ is 
injective by the proof of Theorem \ref{thm36}. 
Thus we see that the map $c_{P,Q}$ is injective. 

Fix an isomorphism 
$P^{\rm gp}/Q^{\rm gp} \cong \mathbb{Z}^{s}$ and 
we put $P' := \{ q+ r \in P[\frac{1}{p}] ; q \in Q[\frac{1}{p}], r \in P \}$. 
For an element $x = q + r$ in $P'$ and for $I \subset [1,s]$, 
the map $c_{P,Q}$ sends elements in $\Lambda_{(A_m, P)/W_m(R[Q], Q)}^{\bullet}$ as
follows: 
 \begin{align}
{}^{V^{u(q)}}(\eta X^{p^{u(q)}q}) \otimes T^{r} \bigwedge_{i \in I} d\log T_i & \,\,\mapsto\,\,  
{}^{V^{u(x)}}(\eta X^{p^{u(x)}x})\bigwedge_{i \in I} d\log X_i,  \label{eq:cpq} \\
{}^{V^{u(q)}}(\eta X^{p^{u(q)}q}) \otimes dT^{r} \bigwedge_{i \in I} d\log T_i & \,\,\mapsto\,\, d{}^{V^{u(x)}}(\eta X^{p^{u(x)}x})\bigwedge_{i \in I} d\log X_i.  \nonumber 
 \end{align}
Since $\Lambda_{(A_m, P)/W_m(R[Q], Q)}^{\bullet}$ is generated by elements 
on the left hand sides (with various $x$ and $I$) and 
$W_m\Lambda_{(R[P], P )/(R[Q],  Q),  x}^{\bullet}$ is generated by 
elements on the right hand side (with various $I$), we see that the injective map 
$c_{P,Q}$ induces the isomorphism 
\[\Lambda_{(A_m, P)/W_m(R[Q], Q)}^{\bullet} \xrightarrow{\cong} \bigoplus_{x \in P'} W_m\Lambda_{(R[P], P )/(R[Q],  Q), x}^{\bullet}. \]
So, to prove the theorem, it suffices to see the acyclicity of the 
complex 
\[\bigoplus_{x \in P[\frac{1}{p}] \setminus P'} W_m\Lambda_{(R[P], P )/(R[Q],  Q),  x }^{\bullet}.\]
Noting the isomorphism 
\begin{eqnarray*}
W_m\Lambda_{(R[P],  P)/(R[Q], Q), x}^{\bullet}  & \cong & W_m\Lambda_{(R[P^{\rm gp}], P^{\rm gp})/(R[Q], Q), x}^{\bullet}  \\
& \cong & W_m \Lambda_{(R[P^{\rm gp}], P^{\rm gp})/(R[Q^{\rm gp}], Q^{\rm gp}), x }^{\bullet}  \\
& \cong & W_m \Lambda_{(R[Q^{\rm gp}][P^{\rm gp}/Q^{\rm gp}], P^{\rm gp}/Q^{\rm gp})/(R[Q^{\rm gp}],  {*}), x }^{\bullet}  
\end{eqnarray*} 
constructed in the proof of Proposition \ref{thm8}, we see that it suffices to prove 
the acyclicity of the complex 
$W_m\Lambda_{(R[Q^{\rm gp}][P^{\rm gp}/Q^{\rm gp}], P^{\rm gp}/Q^{\rm gp})/(R[Q^{\rm gp}], {*}), x }^{\bullet}$ 
for any $x \in P[\frac{1}{p}] \setminus P'$. 

Fix $x \in P[\frac{1}{p}] \setminus P'$. If $u_{P^{\rm gp}/Q^{\rm gp}}(x) = 0$, we can write 
$x$ as $x = x_1 +x_2$ for some $x_1 \in Q^{\rm gp}[\frac{1}{p}], x_2 \in P^{\rm gp}$. 
Then we have $p^{u(x_1)}x = p^{u(x_1)}x_1 + p^{u(x_1)}x_2$. Since $p^{u(x_1)}x \in p^{u(x_1)} \left(P^{\rm gp}/Q^{\rm gp} \right)$, there exist elements $x'_1 \in Q, x'_2 \in P$ 
such that $p^{u(x_1)}x = x'_1 + p^{u(x'_1)}x'_2$ by $(\star \star)$. Then, since 
$P^{\rm gp}$ is $p$-torsion free, $x = \frac{x'_1}{p^{u(x_1)}} + x'_2 \in P'$, which is 
a contradiction. So we conclude that $u_{P^{\rm gp}/Q^{\rm gp}}(x) \geq 1$. 

Now we write the image of $x$ by the composite 
$ P[\frac{1}{p}] \subset  P^{\rm gp}[\frac{1}{p}] \cong 
Q^{\rm gp}[\frac{1}{p}] \oplus (P^{\rm gp}/Q^{\rm gp})[\frac{1}{p}]$ by 
$y_1 + y_2$. Then, by \eqref{3.2.4.2}, 
the complex $W_m\Lambda_{(R[Q^{\rm gp}][P^{\rm gp}/Q^{\rm gp}], P^{\rm gp}/Q^{\rm gp})/(R[Q^{\rm gp}], {*}), x }^{\bullet}$ is a direct summand of 
the $y_2$-part of the complex $W_m\Lambda_{(R[Q^{\rm gp}][P^{\rm gp}/Q^{\rm gp}], P^{\rm gp}/Q^{\rm gp})/(R[Q^{\rm gp}], {*})}^{\bullet}$. 
Since $u_{P^{\rm gp}/Q^{\rm gp}}(y_2) \geq 1$ 
(namely, $y_2 \in  (P^{\rm gp}/Q^{\rm gp})[\frac{1}{p}] \setminus  (P^{\rm gp}/Q^{\rm gp})$) by the argument in the previous paragraph, we see by Lemma \ref{thm4}
that the $y_2$-part of the complex $W_m\Lambda_{(R[Q^{\rm gp}][P^{\rm gp}/Q^{\rm gp}], P^{\rm gp}/Q^{\rm gp})/(R[Q^{\rm gp}], {*})}^{\bullet}$ is acyclic. 
Hence the complex 
\[ 
W_m\Lambda_{(R[Q^{\rm gp}][P^{\rm gp}/Q^{\rm gp}], P^{\rm gp}/Q^{\rm gp})/(R[Q^{\rm gp}], {*}), x }^{\bullet}\] is acyclic, as required. So the proof of the theorem is finished. 
\end{proof}

\begin{rem}
We remark here that Theorem \ref{qis1} does not hold if the monoid homomorphism 
$Q \to P$ does not have the property $(\star \star)$. 
Assume that $R$ is of characteristic $p>0$ and let $P, Q$ be  
\[ P := \{ (x,  y) \in\mathbb{N}^2 ; y \leq p^{2k} x , x\leq p^{2k}y \}, % \,\, (k \gg 1), 
\qquad Q = \{ (x,  y) \in \mathbb{N}^2 ; x=y \} \]
with $Q \to P$ the natural inclusion. We denote the image of an element 
$(x,  y) \in P$ by the map $P \to R[P]$ by $X^xY^y$. 
Then, when $k > m \geq 3$, 
${}^{V}(pXY^{p^{k}+1}) = {}^{V^{2}}(X^p Y^{p^{k+1}+p}) \neq 0$ in 
$W_m \Lambda_{(R[P],  P)/(R[Q], Q)}^{0} = W_m(R[P])$ and it is not in the image of 
$W_m(R[Q]) \otimes_{{\mathbb{Z}}[Q]} {\mathbb{Z}}[P] \to W_m[P]$ 
(the degree $0$ part of the map $c_{P,Q}$) by the computation in \eqref{eq:cpq}. 
On the other hand, we have the equality 
\begin{eqnarray*}
 d^{V}(pXY^{p^{k}+1}) & = & ^{V}(d XY^{p^{k}+1}) \\
& = &  ^{V}( XY^{p^{k}+1} (d\log X+ (p^{k}+1)d\log Y) \\
& = &  ^{V}( XY^{p^{k}+1}p^{k}d\log Y) \\
& = &  0
\end{eqnarray*}
in $W_m\Lambda^1_{(R[P],P)/(R[Q],Q)}$. 
Thus the element ${}^{V}(pXY^{p^{k}+1}) = {}^{V^{2}}(X^p Y^{p^{k+1}+p})$ is a 
nontrivial $0$-cocycle in $W_m \Lambda_{(R[P],  P)/(R[Q], Q)}^{\bullet}$
which is not in the image of the comparison morphism $c_{P,Q}$. 
\end{rem}

Now we prove our main comparison isomorphism 
when the base log scheme is etale locally log smooth over a scheme with trivial log structure. 
(Although the theorem is also a particular case of our main theorem (Theorem \ref{thm:main2}), 
we include the proof here because the statement of the theorem is simpler and thus 
it would be helpful to the reader.) Let 
$f: (X, {\mathcal{M}}) \to (Y, {\mathcal{N}})$ be a log smooth saturated 
morphism of fs log schemes on which $p$ is nilpotent and suppose that, 
etale locally, $(Y, {\mathcal{N}})$ is log smooth over a log scheme of the form 
$({\rm Spec}\,R, {*})$. 

\begin{theorem}\label{thm:main1}
In the above situation, the comparison morphism 
\[ \mathbb{R}u_{m*} \mathcal{O}_m \to W_{m} \Lambda_{(X, \mathcal{M})/(Y,  \mathcal{N})}^{\bullet} \]
we constructed in Section 1.3 is a quasi-isomorphism. 
\end{theorem}

\begin{proof}
We may assume that $(Y, {\mathcal{N}})$ is log smooth over $({\rm Spec}\,R, {*})$. 
Also, as we see later in the proof, we may assume that $p$ is nilpotent in $R$. 
If we reduce the theorem to the case where the morphism 
$f: (X, {\mathcal M}) \to (Y, {\mathcal N})$ is the 
morphism ${\rm Spec}(R[P],  P) \to {\rm Spec}(R[Q],  Q)$ 
associated to an injective $p$-saturated homomorphism   
$Q \to P$ of fs monoids such that $Q^{\rm gp}, P^{\rm gp}, 
P^{\rm gp}/Q^{\rm gp}$ are torsion free, the theorem follows from 
Theorem \ref{qis1}. In the following, we prove this reduction. 

Let $x$ be a geometric point of $X$ and let $y = f \circ x$
be the geometric point of $Y$ induced by $x$. We put 
$Q = {\mathcal N}_y/{\mathcal O}_{Y,  y}^*$, 
$P = {\mathcal M}_x/{\mathcal O}_{X,  x}^*$. Then the induced homomorphism 
$Q \to P$ is a  $p$-saturated morphism of fs monoids and it is injective by   
\cite[Prop.~4.1(2)]{Kato}. Also, since 
$P^* = \{0\}, Q^* = \{0\}$, $P^{\rm gp}, Q^{\rm gp}$ are 
torsion free. 

We prove that 
$P^{\rm gp}/Q^{\rm gp}$ is also torsion free. 
By \cite[Thm.~II.3.1]{Tsuji}, $f$ is $\ell$-saturated
for any prime number $\ell$. Thus $Q \to P$ is $\ell$-saturated. 
Therefore, the map 
$$ h: P \oplus_{Q,  \ell} Q \to P; 
\quad (\alpha,  \beta) \mapsto \ell\alpha+\beta $$
is exact. 

If $x$ is an element in $P^{\rm gp}$ with $\ell x =: y \in Q^{\rm gp}$, 
the element 
$(x,  -y) \in P^{\rm gp} \oplus_{Q^{\rm gp},  \ell} Q^{\rm gp}$ is sent by 
$h^{\rm gp}$ to 
$0 \in P \subseteq P^{\rm gp}$. Hence, by the exaceness of $h$, 
there exist $a \in P, b \in Q$ such that 
$(x,  -y) = (a,  b) \in P^{\rm gp} \oplus_{Q^{\rm gp},  \ell} Q^{\rm gp}$. 
Then we have $\ell a + b = 0 \in P \subseteq P^{\rm gp}$. Since   
$P^*$ is trivial, we conclude that $a = b = 0$. Hence 
$(x,  -y) = (0,  0) \in P^{\rm gp} \oplus_{Q^{\rm gp},  \ell} Q^{\rm gp}$
and so there exists an element $z \in Q^{\rm gp}$ such that 
$(x,  -y) = (z, -\ell z)$, namely, $x \in Q^{\rm gp}$. 
This argument shows that $P^{\rm gp}/Q^{\rm gp}$ does not have 
$\ell$-torsion. Since this is true for any prime number 
$\ell$, we conclude that $P^{\rm gp}/Q^{\rm gp}$ is torsion free, as required. 

In the following, we will take a nice chart of $f$ by the following diagram: 
\[ \xymatrix{
(U, {\mathcal{M}}|_{U}) \ar[rr] \ar[rd] \ar[rddd]_{f'} \ar[rrrd] & & \mbox{Spec}(R[P \oplus {\mathbb{N}}^{a+b}], P \oplus {\mathbb{N}}^{a+b}) \ar[d] & \\
& \text{fiber product} \ar[r] \ar[dd] & \mbox{Spec}(R[P \oplus {\mathbb{N}}^{a}], P \oplus {\mathbb{N}}^{a}) \ar[r] \ar[d] & \mbox{Spec}(R[P], P) \ar[dd] \\
& & \mbox{Spec}(R[Q \oplus {\mathbb{N}}^{a}], Q \oplus {\mathbb{N}}^{a}) \ar[rd] & \\
& (V, {\mathcal{N}}|_{V}) \ar[rr] \ar[ru] & & \mbox{Spec}(R[Q],  Q).  
}
\]

First we apply \cite[Lem.~4.1.1]{KF} to $(Y,  {\mathcal N})/{\rm Spec}(R,  {*})$
and to $(X,  {\mathcal M})/(Y,  {\mathcal N})$. Then we obtain a 
an etale neighborhood $U$ of $x$ in $X$,  an etale neighborhood $V$ of $y$ in $Y$, 
a morphism $f': (U, {\mathcal{M}}|_{U}) \to (V, {\mathcal{N}}|_{V})$ over $f$ and 
a chart $Q \to P$ of it such that the induced morphisms 
\begin{equation}\label{eq:vu}
(V,{\mathcal{N}}|_{V}) \to \mbox{Spec}(R[Q],  Q), \qquad 
(U,{\mathcal{M}}|_{U}) \to (V,{\mathcal{N}}|_{V}) \times_{\mbox{\scriptsize{Spec}}(R[Q],Q)} \mbox{Spec}(R[P],P)
\end{equation} are 
strict smooth. Then, after shrinking $V$ suitably, we can take a decomposition 
\[ (V,{\mathcal{N}}|_{V}) \to \mbox{Spec}(R[Q \oplus {\mathbb{N}}^a], Q \oplus {\mathbb{N}}^a) \to \mbox{Spec}(R[Q],  Q) \] 
of the first map in \eqref{eq:vu} in which the first map is strict etale. 
Since $p$ is nilpotent on $V$, the first map in the above diagram factors through 
$\mbox{Spec}((R/p^nR)[Q \oplus {\mathbb{N}}^a], Q \oplus {\mathbb{N}}^a)$ for some $n$ and the map 
\[ (V,{\mathcal{N}}|_{V}) \to \mbox{Spec}((R/p^nR)[Q \oplus {\mathbb{N}}^a], Q \oplus {\mathbb{N}}^a)\] 
is strict etale. Thus, we may assume that $R/p^nR = R$, namely, $p$ is nilpotent in $R$. 

Let 
$\mbox{Spec}(R[P \oplus {\mathbb{N}}^{a}], P \oplus {\mathbb{N}}^{a}) \to \mbox{Spec}(R[Q \oplus {\mathbb{N}}^{a}], Q \oplus {\mathbb{N}}^{a})$ be the morphism we obtain 
by pulling back the map 
${\rm Spec}(R[P],  P) \to {\rm Spec}(R[Q],  Q)$ to 
$\mbox{Spec}(R[Q \oplus {\mathbb{N}}^a], Q \oplus {\mathbb{N}}^a)$. 
Then the morphism $(U,{\mathcal{M}}|_{U}) \to \mbox{Spec}(R[P], P)$ induce a morphism 
\begin{equation}\label{eq:u}
(U,{\mathcal{M}}|_{U}) \to \mbox{Spec}(R[P \oplus {\mathbb{N}}^{a}], P \oplus {\mathbb{N}}^{a}), 
\end{equation}
which factors through the fiber product of $(V,{\mathcal N}|_V)$ and 
$\mbox{Spec}(R[P \oplus {\mathbb{N}}^{a}], P \oplus {\mathbb{N}}^{a})$ over 
$\mbox{Spec}(R[Q \oplus {\mathbb{N}}^{a}], Q \oplus {\mathbb{N}}^{a})$. 
Since the morphism from this fiber product to 
$\mbox{Spec}(R[P \oplus \mathbb{N}^{a}], P \oplus \mathbb{N}^{a})$ is strict etale 
and the morphism from $(U, {\mathcal M}|_U)$ to this fiber product is strict smooth, 
we conclude that the map \eqref{eq:u} is 
strict smooth. Thus, after shrinking $U$ suitably, we have a decomposition 
\[ (U,{\mathcal M}|_U) \to \mbox{Spec}(R[P \oplus {\mathbb{N}}^{a+b}], P \oplus {\mathbb{N}}^{a+b}) \to \mbox{Spec}(R[P \oplus {\mathbb{N}}^a], P \oplus {\mathbb{N}}^a)\] 
of the morphism \eqref{eq:u} in which the first map is strict etale.

By the argument above, we obtain the following commutative diagram in which 
the horizontal arrows are strict etale: %and the vertical arrows are log smooth: 
\[ \xymatrix{
(U, {\mathcal{M}}|_{U}) \ar[r] \ar[d] & \mbox{Spec}(R[P \oplus {\mathbb{N}}^{a+b}], P \oplus {\mathbb{N}}^{a+b}) \ar[d] \\
(V, {\mathcal{N}}|_{V}) \ar[r] & \mbox{Spec}(R[Q \oplus {\mathbb{N}}^{a}], Q \oplus {\mathbb{N}}^{a}). 
}
\]
Noting that the statement of the theorem is etale local on $X$ and $Y$, we can 
reduce the proof of the theorem to the case where $f$ is the morphism 
$\mbox{Spec}(R[P \oplus {\mathbb{N}}^{a+b}], P \oplus {\mathbb{N}}^{a+b})
\to \mbox{Spec}(R[Q \oplus {\mathbb{N}}^{a}], Q \oplus {\mathbb{N}}^{a}).$
(See also the proof of Theorem \ref{thm10}.) 
Since $Q \to P$ is an injective $p$-saturated morphism of fs monoids 
such that $Q^{\rm gp}, P^{\rm gp}, P^{\rm gp}/Q^{\rm gp}$ are torsion free, 
the same properties are satisfied for the homomorphism 
$Q \oplus {\mathbb{N}}^a \to P \oplus {\mathbb{N}}^{a+b}$. 
Hence we obtained the required reduction and so we are done. 
\end{proof}

\section{Relative log de Rham-Witt complex (III)}

In this section, first 
we study the relative log de Rham-Witt complex 
$W_m\Lambda^{\bullet}_{(R[P]/JR[P] ,P)/(R[Q]/JR[Q],Q)}$ for a ${\mathbb Z}_{(p)}$-algebra $R$, 
an injective $p$-saturated homomorphism of fs monoids $Q \to P$ such that 
$Q^{\rm gp}, P^{\rm gp}, P^{\rm gp}/Q^{\rm gp}$ are torsion free and 
a radical ideal $J \subset Q$. 
%and that 
%$(R[P], P)/(R[Q],Q)$ is log smooth. 
(For precise description of our setting, see Section 4.2.) 
The main result is the existence of natural decomposition 
\begin{equation}\label{eq:decompIII}
W_m\Lambda^{\bullet}_{(R[P]/JR[P] ,P)/(R[Q]/JR[Q],Q)} 
= \bigoplus_{x \in P[\frac{1}{p}] \atop \text{$J$-minimal}} W_m\Lambda^{\bullet}_{(R[P],P)/(R[Q],Q), x}
\end{equation}
indexed by $J$-minimal elements in $P[\frac{1}{p}]$. (For definition of $J$-minimality, see 
Section 4.1.) 
Then, using this decomposition, 
we prove that 
the comparison morphism 
\[\mathbb{R}u_{m*} \mathcal{O}_m \to W_m \Lambda_{(X, \mathcal{M})/(Y, \mathcal{N})}^{\bullet} \]
constructed in Section 1.3 is a quasi-isomorphism when 
%when $(Y, {\mathcal N})$ is a hollow fs log scheme 
%on which $p$ is nilpotent and 
$(X, \mathcal{M}) \to (Y, \mathcal{N})$ is a log smooth saturated morphism of fs log schemes 
on which $p$ is nilpotent such that $(Y,{\mathcal N})$ satisfies the condition 
$(\spadesuit)$ given in the introduction. 

\subsection{Preliminaries on monoids (III)}

First we recall the notion of ideals of monoids. 

\begin{definition} 
\begin{enumerate}[(1)] 
\item 
\cite[Def.~I.1.1]{Tsuji}
A subset $J$ of a monoid $Q$ is called an ideal if, for any $x \in J$ and $q \in Q$, 
$x + q \in J$. 
\item 
An ideal $J$ of a monoid $Q$ is radical if any $x \in Q$ with 
$nx \in J$ for some $n \in {\mathbb N}_{> 0}$ belongs to $J$. 
\end{enumerate}
\end{definition}

In the following in this subsection, let 
$Q \to P$ be an injective $p$-saturated homomorphism of fs monoids 
such that $Q^{\rm gp}, P^{\rm gp}, P^{\rm gp}/Q^{\rm gp}$ are torsion free, 
and let $J \subset Q$ be a radical ideal. 

\begin{definition}
Let $Q \to P, J$ be as above. An element $x$ in $P$ is called $J$-minimal 
if we cannot write $x$ as $x = r + q$ with $r \in P, q \in J$. 
\end{definition}

\begin{lemma}\label{lem:min}
An element $x$ in $P$ is $J$-minimal if and only if so is $px$. 
\end{lemma}

\begin{proof}
If $x$ is not $J$-minimal and if we write $x = r + q$ with $r \in P, q \in J$, 
$px = pr + pq$ with $pr \in P, pq \in J$ and so $px$ is not $J$-minimal. 

Assume that $px$ is not $J$-minimal and write $px = r + q$ with 
$r \in P, q \in J$. Then $r$ belongs to $p(P^{\rm gp}/Q^{\rm gp})$ 
when we regard it as an element of $P^{\rm gp}/Q^{\rm gp}$. 
Then, by $(\star)$, there exists elements $q' \in Q, r' \in P$ with 
$r = pr' + q'$. Thus $px = pr' + (q + q')$. So 
$q + q' \in Q \cap pP^{\rm gp} = Q \cap pQ^{\rm gp} = pQ$. 
(The first equality follows from the torsion freeness of 
$P^{\rm gp}/Q^{\rm gp}$ and the second equality follows from the 
saturatedness of $Q$.) Thus there exists an element 
$q'' \in Q$ with $q + q' = pq''$ and so 
$px = pr' + pq''$. Since $P^{\rm gp}$ is torsion free, this implies the 
equality 
$x = r' + q''$. Also, since $pq'' = q + q' \in J$ and $J$ is radical, 
$q'' \in J$. Hence $x$ is not $J$-minimal. 
\end{proof}

We can extend the definition of $J$-minimality to elements in $P[\frac{1}{p}]$: 

\begin{definition}
An element $x$ in $P[\frac{1}{p}]$ is called $J$-minimal 
if $p^nx$ is a $J$-minimal element in $P$ for some $n \in {\mathbb Z}$.  
\end{definition}

By Lemma \ref{lem:min}, $x \in P$ is $J$-minimal if and only if 
it is $J$-minimal as an element in $P[\frac{1}{p}]$. 
By definition, $x \in P[\frac{1}{p}]$ is $J$-minimal if and only if 
$px$ is $J$-minimal.

\subsection{Decomposition of relative log de Rham-Witt complex (III)}

In this subsection, let $R$ be a ${\mathbb Z}_{(p)}$-algebra and let  
$Q \to P$ be an injective $p$-saturated homomorphism of fs monoids 
such that $Q^{\rm gp}, P^{\rm gp}, P^{\rm gp}/Q^{\rm gp}$ are torsion free. 
Also, let $J \subset Q$ be a radical ideal. We denote the ideal of 
$R[Q]$ (resp. $R[P]$) generated by the image of the canonical map 
$J \to Q \to R[Q]$ (resp. $J \to Q \to P \to R[P]$) by 
$JR[Q]$ (resp. $JR[P]$). 
Then we have naturally a pre-log ring $(R[P]/JR[P], P)$ over $(R[Q]/JR[Q],Q)$. 
We study the relative log de Rham-Witt complex 
$W_m\Lambda^{\bullet}_{(R[P]/JR[P] ,P)/(R[Q]/JR[Q],Q)}$.

%Let \[ I := {\rm Ker}(R[P] = R[P] \otimes_{R[Q]} R[Q] \xrightarrow{{\rm id} \otimes \varphi} R[P] \otimes_{R[Q]} R) = 
%\bigoplus_{0 \not= q \in Q \atop r \in P}R T^{r+q}. \]  

By \cite[Prop.~3.11(2)]{Ma}, we have the canonical isomorphism 
\[ W_m\Lambda^{\bullet}_{(R[P] \otimes_{R[Q]} R ,P)/(R,Q)}
\cong W_m\Lambda^{\bullet}_{(R[P], P)/(R[Q],Q)}/{\mathbb I}_m^{\bullet}, \] 
where ${\mathbb I}_m^{\bullet}$ is the differential graded ideal 
\[ 
W_m(JR[P])W_m\Lambda^{\bullet}_{(R[P], P)/(R[Q],Q)} + dW_m(JR[P])W_m\Lambda^{\bullet-1}_{(R[P], P)/(R[Q],Q)}\] 
of $W_m\Lambda^{\bullet}_{(R[P], P)/(R[Q],Q)}$. 
%\[ 
%W_m\Lambda^{\bullet}_{(R[P] \otimes_{R[Q]} R ,P)/(R,Q)}
%\cong 
%\frac{W_m\Lambda^{\bullet}_{(R[P], P)/(R[Q],Q)}}{W_m(I)W_m\Lambda^{\bullet}_{(R[P], P)/(R[Q],Q)} + %dW_m(I)W_m\Lambda^{\bullet-1}_{(R[P], P)/(R[Q],Q)}}. \] 
To study ${\mathbb I}_m^{\bullet}$, we study the differential graded ideal 
\[ {\mathbb I}^{\bullet} := 
W(JR[P])W\Lambda^{\bullet}_{(R[P], P)/(R[Q],Q)} + dW(JR[P])W\Lambda^{\bullet-1}_{(R[P], P)/(R[Q],Q)} \] 
of $W\Lambda^{\bullet}_{(R[P], P)/(R[Q],Q)}$. 

First we study the degree $0$ part ${\mathbb{I}}^0$. 
In the following, we call the image of $P$-basic Witt differentials 
$b(\xi,x), b(\xi,x,I) \in W\Lambda^{\bullet}_{(R[P],P)/(R,*)}$ 
by the morphism 
$W\Lambda^{\bullet}_{(R[P],P)/(R,*)} \to W\Lambda^{\bullet}_{(R[P],P)/(R[Q],Q)}$ 
also $P$-basic Witt differentials and denote them by the same symbols. 

\begin{lemma}\label{lem:I0} 
Any element in ${\mathbb{I}}^0$ is written as a convergent sum of 
$P$-basic Witt differentials $b(\xi,x)$ of degree $0$ with $x \in P[\frac{1}{p}]$ 
not $J$-minimal. 
\end{lemma}

\begin{proof}
Noting that any element in $JR[P]$ is an $R$-linear combination of 
the elements $T^x$ with $r \in P$ not $J$-minimal, we see that 
any element in $W(JR[P])$ is written as a convergent sum of the elements 
${}^{V^m}(\eta X^{x})$ with $\eta \in W(R), x \in P$ such that  
$x$ is not $J$-minimal. 
If we denote by $n$ the maximal element in ${\mathbb N}$ with 
$p^{-n}x \in P$ and $n \leq m$, 
\[ {}^{V^m}(\eta X^{x}) = {}^{V^{m-n}}({}^{V^n}\!\eta X^{p^{-n}x}) = b({}^{V^m}\!\eta, p^{-m}x) \] 
and $p^{-m}x$ is a non-$J$-minimal element in $P[\frac{1}{p}]$. Since 
we have ${\mathbb{I}}^0 = W(JR[P])W(R[P]) = W(JR[P])$, we obtain the assertion. 
\end{proof}

\begin{comment}
Next we study the degree $1$ part ${\mathbb{I}}^1$. 

\begin{lemma}\label{lem:I1} 
Any element in ${\mathbb{I}}^1$ is written as a covergent sum of 
$P$-basic Witt differentials $b(\xi,x,n)$ of degree $1$ with $x \in P[\frac{1}{p}]$ 
non minimal. 
\end{lemma}

\begin{proof}
Since 
${\mathbb I}^{1} := 
W(I)W\Lambda^{1}_{(R[P], P)/(R[Q],Q)} + W(R[P]) dW(I)$, 
any element in ${\mathbb I}^{1}$ is written as a convergent sum of 
elements of the following three types: 
\begin{enumerate}
\item $b(\xi,x) d\log X^y$ with $x$ non minimal. 
\item $b(\xi_1,x_1)db(\xi_2,x_2)$ with $x_1$ non minimal. 
\item $b(\xi_1,x_1)db(\xi_2,x_2)$ with $x_2$ non minimal. 
\end{enumerate}
The calculation in the proof of Proposition \ref{thm16} shows that 
the above elements can be written as a sum of 
$P$-basic Witt differentials of weight $x$ in type 1 and weight 
$x_1 + x_2$ in type 2. Since $x_1 + x_2$ is non minimal if so is either 
$x_1$ or $x_2$, we obtain the assertion. 
\end{proof}
\end{comment}

Then we obtain the following description of the elements in ${\mathbb I}^n$ for general $n$: 

\begin{prop}\label{prop:In} 
Any element in ${\mathbb{I}}^n$ is written as a covergent sum of 
$P$-basic Witt differentials $b(\xi,x,I)$ of degree $n$ with $x \in P[\frac{1}{p}]$ 
not $J$-minimal. 
\end{prop}

\begin{proof}
%We prove the claim by induction on $n$. So we assume that the claim holds for 
%$n = k$ and prove the claim in the case $n = k+1$. 
We may assume that $n \geq 1$. Since 
\[ {\mathbb I}^{n} := 
W(JR[P])W\Lambda^{n}_{(R[P], P)/(R[Q],Q)} + dW(JR[P])W\Lambda^{n-1}_{(R[P], P)/(R[Q],Q)}, \]  
any element in ${\mathbb I}^{n}$ is written as a convergent sum of 
elements of the following types by Lemma \ref{lem:I0} and Corollary \ref{thm23}: 
\begin{enumerate}
\item $b(\xi_1,x_1) b(\xi,x,I)$ with $|I| = n$ and $x_1$ not $J$-minimal. 
\item $d b(\xi_1,x_1) \cdot b(\xi,x,I)$ with $|I| = n-1$ and $x_1$ not $J$-minimal.  
\end{enumerate}
The calculation in Section 2.2 shows that 
the above elements can be written as a sum of 
$P$-basic Witt differentials of weight $x_1 + x$. 
Since $x_1 + x$ is not $J$-minimal if  
$x_1$ is not $J$-minimal, we obtain the assertion. 
\end{proof}

Using the above proposition, we prove the following proposition, 
which is the main result in this subsection. 

\begin{prop}\label{prop:decompIII}
We have the canonical isomorphism 
\begin{equation}\label{eq:decompmar}
W_m\Lambda^{\bullet}_{(R[P]/JR[P],P)/(R[Q]/JR[Q],Q)} 
\cong \bigoplus_{x \in P[\frac{1}{p}] \atop \text{$J$-minimal}} W_m\Lambda^{\bullet}_{(R[P],P)/(R[Q],Q), x}. 
\end{equation}
\end{prop}

\begin{proof}
Since the canonical surjection 
$W\Lambda^{\bullet}_{(R[P], P)/(R[Q],Q)} \to W_m\Lambda^{\bullet}_{(R[P], P)/(R[Q],Q)}$ 
sends ${\mathbb I}^{\bullet}$ surjectively to 
${\mathbb I}_m^{\bullet}$, we see that ${\mathbb I}_m^{\bullet}$ is equal to the set of elements 
which are written as a sum of $P$-basic Witt differentials with non-$J$-minimal weights by Proposition \ref{prop:In}. 
Thus the decomposition in Proposition \ref{thm8}(2) induces the isomorphism \eqref{eq:decompmar}. 
%\begin{equation*}%\label{eq:decompIII}
%W_m\Lambda^{\bullet}_{(R[P] \otimes_{R[Q]} R ,P)/(R,Q)} 
%\cong \bigoplus_{x \in P[\frac{1}{p}] \atop \text{minimal}} W_m\Lambda^{\bullet}_{(R[P],P)/(R[Q],Q), x},  
%\end{equation*}
%as required. 
\end{proof}

\subsection{Comparison isomorphism (III)}

In this subsection, we prove our main comparison isomorphism 
between the complex computing the relative log crystalline cohomology and 
the relative log de Rham-Witt complex. 

To prove this, first we prove the comparison isomorphism of 
relative de Rham complex and relative log de Rham-Witt complex in certain situation. 
Let $R$, $Q \to P$, $J$ and $(R[P]/JR[P],P)/(R[Q]/JR[Q],Q)$ 
be as in the previous subsection. 
%, and let $W_m(\varphi): W_m(R[Q]) \to W_m(R)$ 
%be the morphism of Witt rings induced by $\varphi$. 

Let $((A_m,P), \phi_m, \delta_m)_m$ be the 
log Frobenius lift of $(R[P],P)$ over $(R[Q],Q)$ in 
Example \ref{exam} (so $A_m = W_m(R[Q]) \otimes_{\mathbb{Z}[Q]} \mathbb{Z}[P]$). 
Then, if we put $\overline{A}_m := W_m(R[Q]/JR[Q]) \otimes_{W_m(R[Q])} A_m$, 
the morphisms $\phi_m, \delta_m$ naturally induce the morphisms 
$\overline{\phi}_m: \overline{A}_m \to \overline{A}_{m-1}$, 
$\overline{\delta}_m: \overline{A}_m \to W_m(R[P]/JR[P])$ such that 
the triple $((\overline{A}_m,P), \overline{\phi}_m, \overline{\delta}_m)_m$
forms a log Frobenius lift of $(R[P]/JR[P],P)$ over $(R[Q]/JR[Q],Q)$. 

Let 
$(\overline{Z}_m, \overline{\Phi}_m, \overline{\Delta}_m)_m$ be
the corresponding log Frobenius lift (as log schemes) of 
$\overline{Z} := \mbox{Spec}(R[P]/JR[P],P)$ over 
$\mbox{Spec}(R[Q]/JR[Q], Q)$. 
Then, by the construction in Section 1.3 using $(\overline{Z}_m, \overline{\Phi}_m, \overline{\Delta}_m)_m$, 
we obtain the comparison morphism 
\begin{align*}
 \Lambda_{\overline{Z}_{m}/{{\rm Spec}}(W_{m}(R[Q]/JR[Q],  Q))}^{\bullet} \to 
W_{m} \Lambda_{\overline{Z}/{{\rm Spec}}(R[Q]/JR[Q], Q)}^{\bullet} % \tag{3.3.1}
\end{align*}
of complexes of quasi-coherent sheaves, hence the morphism 
\begin{align} \label{444}
c_{P,Q,J}:  \Lambda_{(\overline{A}_m, P)/W_{m}(R[Q]/JR[Q], Q)}^{\bullet} \to 
W_{m} \Lambda_{(R[P]/JR[P],P)/(R[Q]/JR[Q], Q)}^{\bullet}
 % \tag{3.3.1}
\end{align}
of complex of $W_m(R[Q]/JR[Q])$-modules. 

\begin{theorem}\label{thm:qisqis1}
In the above situation, the morphism $c_{P,Q,J}$ is a quasi-isomorphism. 
\end{theorem}

\begin{proof}
By the base change property of relative log de Rham complex, we have the isomorphism 
\begin{align*}
\Lambda_{(\overline{A}_m, P)/W_{m}(R[Q]/JR[Q], Q)}^{\bullet} & \cong 
W_m(R[Q]/JR[Q]) \otimes_{W_m(R[Q])} \Lambda_{(A_m, P)/W_m(R[Q], Q)}^{\bullet} \\  & 
= \Lambda_{(A_m, P)/W_m(R[Q], Q)}^{\bullet} / W_m(JR[Q]) \Lambda_{(A_m, P)/W_m(R[Q], Q)}^{\bullet}. 
\end{align*}
% where $J = {\rm Ker}(\varphi: R[Q] \to R) = \bigoplus_{0 \not= q \in Q} RT^q$. 
An element in $W_m(JR[Q]) \Lambda_{(A_m, P)/W_m(R[Q], Q)}^{\bullet}$ is 
written as a sum of elements of the forms 
\[ 
{}^{V^{u(q)}}(\eta X^{p^{u(q)}q})T^{r} \bigwedge_{i \in I} d\log T_i, \qquad 
{}^{V^{u(q)}}(\eta X^{p^{u(q)}q}) \otimes dT^{r} \bigwedge_{i \in I} d\log T_i 
\] 
with $q \in J[\frac{1}{p}] := \{x \in Q[\frac{1}{p}]; \exists n \in {\mathbb N}_{>0}, p^nx \in J\}, r \in P$. 

Now recall that the morphism 
\[ c_{P,Q}:  \Lambda_{(A_m, P)/(W_{m}(R[Q],  Q))}^{\bullet} \to W_{m} \Lambda_{(R[P],P)/(R[Q], Q)}^{\bullet}\] 
in Section 3.3 sends elements in $\Lambda_{(A_m, P)/W_m(R[Q], Q)}^{\bullet}$ as 
 \begin{align*}
{}^{V^{u(q)}}(\eta X^{p^{u(q)}q}) \otimes T^{r} \bigwedge_{i \in I} d\log T_i & \,\,\mapsto\,\,  
{}^{V^{u(x)}}(\eta X^{p^{u(x)}x})\bigwedge_{i \in I} d\log X_i,  \\
{}^{V^{u(q)}}(\eta X^{p^{u(q)}q}) \otimes dT^{r} \bigwedge_{i \in I} d\log T_i & \,\,\mapsto\,\, d{}^{V^{u(x)}}(\eta X^{p^{u(x)}x})\bigwedge_{i \in I} d\log X_i, 
 \end{align*}
with $x = q + r \in Q[\frac{1}{p}] + P =: P'$. Also, the map $c_{P,Q}$ induces the isomorphism 
\begin{equation}\label{eq:qpp}
\Lambda_{(A_m, P)/W_m(R[Q], Q)}^{\bullet} \xrightarrow{\cong} \bigoplus_{x \in P'} W_m\Lambda_{(R[P], P )/(R[Q],  Q), x}^{\bullet}. 
\end{equation}
Thus the image of $W_m(JR[Q]) \Lambda_{(A_m, P)/W_m(R[Q], Q)}^{\bullet}$ by the isomorphism \eqref{eq:qpp} 
is the set of sums of $P$-basic elements whose weight $x \in P'$ satisfies the following condition: 
\medskip

\noindent  
(*) \, $x$ can be written as a sum 
$x = q + r$ with $q \in J[\frac{1}{p}], r \in P$. 
\medskip 

It is easy to see that the condition (*) implies that $x$ is not $J$-minimal as an element in $P[\frac{1}{p}]$. 
On the other hand, if $x = q + r \in P' \, (q \in Q[\frac{1}{p}], r \in P)$ is not 
$J$-minimal as an element in $P[\frac{1}{p}]$, we have the equality 
$p^nx = q' + r'$ for some $n \geq 0, q' \in J, r' \in P$. Then, by $(\star \star)$, 
there exists some $r'' \in P, q'' \in Q$ with $r' = p^nr'' + q''$. Hence  
$p^nx = (q' + q'') + p^nr''$ and so $x = p^{-n}(q'+q'') + r''$ with 
$p^{-n}(q'+q'') \in J[\frac{1}{p}]$, namely, $x$ satisfies the condition (*). 
%Also, if $x \in P'$ does not satisfy the condition (*), $x$ should be in $P$ and should be minimal 
%as an element in $P$. Thus it is minimal as an element in $P[\frac{1}{p}]$. 
Therefore, the image of $W_m(JR[Q]) \Lambda_{(A_m, P)/W_m(R[Q], Q)}^{\bullet}$ by the isomorphism \eqref{eq:qpp} 
is equal to $\bigoplus_{x \in P' \atop \text{not $J$-minimal}} W_m\Lambda_{(R[P], P )/(R[Q],  Q), x}^{\bullet}$. 
Since $c_{P,Q,J}$ is compatible with $c_{P,Q}$, the decomposition in Theorem \ref{prop:decompIII}
and the isomorphism \eqref{eq:qpp} induce the isomorphism 
\begin{equation}\label{eq:qpp2}
\Lambda_{(\overline{A}_m, P)/W_m(R[Q]/JR[Q], Q)}^{\bullet} \xrightarrow{\cong} 
\bigoplus_{x \in P' \atop \text{$J$-minimal}} W_m\Lambda_{(R[P]/JR[P], P )/(R[Q]/JR[Q],  Q), x}^{\bullet}. 
\end{equation}
By the decomposition in Theorem \ref{prop:decompIII} and \eqref{eq:qpp2}, 
it suffices to see the acyclicity of the complex 
\[\bigoplus_{x \in P[\frac{1}{p}] \setminus P' \atop \text{$J$-minimal}} W_m\Lambda_{(R[P], P )/(R[Q],  Q),  x }^{\bullet}\]
to prove the theorem, and it is already proven in the proof of Theorem \ref{qis1}. So we are done. 
\end{proof}

Finally, we prove our main comparison isomorphism in this article. 
%Recall that a log scheme $(Z, {\mathcal{L}})$ is hollow if 
%any non-invertible local section of ${\mathcal{L}}$ is sent to $0$ by the structure 
%morphism ${\mathcal{L}} \to {\mathcal O}_Z$. 
Let 
$f: (X, {\mathcal{M}}) \to (Y, {\mathcal{N}})$ be a log smooth saturated 
morphism of fs log schemes on which $p$ is nilpotent and assume that 
$(Y, {\mathcal{N}})$ satisfies the following condition (which we already wrote in the introduction): 
\bigskip 

\noindent 
$(\spadesuit)$ \, Etale locally around any geometric point $y$ of $Y$, 
there exist a chart $\varphi: Q_Y \to {\mathcal N}$ inducing the isomorphism 
$Q \xrightarrow{\cong} {\mathcal N}_y/{\mathcal O}_{Y,y}^*$, 
a radical ideal $J$ of $Q$ such that the composite 
$J_Y \subset Q_Y \xrightarrow{\varphi} {\mathcal N} \to {\mathcal O}_Y$ is zero  
and a ring homomorphism 
$\psi: R \to \Gamma(Y,{\mathcal O}_Y)$ such that 
the morphism $(Y, \mathcal{N}) \to {\rm Spec}(R[Q]/JR[Q],Q)$ induced by 
$\varphi, \psi$ is strict smooth. 
\bigskip 

\begin{theorem}\label{thm:main2}
In the above situation, the comparison morphism 
\[ \mathbb{R}u_{m*} \mathcal{O}_m \to W_{m} \Lambda_{(X, \mathcal{M})/(Y,  \mathcal{N})}^{\bullet} \]
we constucted in Section 1.3 is a quasi-isomorphism. 
\end{theorem}

\begin{rem}\label{rem:remrem}
Let $(Y, \mathcal{N})$ be an fs log scheme. Then 
etale locally around any geometric point $y$ of $Y$, 
there always exists a chart 
$\varphi: Q_Y \to {\mathcal N}$ inducing the isomorphism 
$Q \xrightarrow{\cong} {\mathcal N}_y/{\mathcal O}_{Y,y}^*$, by 
the argument in \cite[p.186]{KF}. 
\begin{enumerate}[(1)]
\item 
When the log structure ${\mathcal N}$ is trivial, 
the condition $(\spadesuit)$ is always satisfied because 
$Q$ above is trivial and 
$Y$ is etale locally isomorphic to an affine scheme. 
Thus Theorem \ref{thm10} is included in Theorem \ref{thm:main2} and so 
the case (1) of Matsuue's result explained in the introduction 
(\cite[Thm.~7.2]{Ma}) is included in Theorem \ref{thm:main2}. 
\item 
When $(Y,{\mathcal N})$ is etale locally log smooth over a scheme with trivial log structure, 
we see by the argument in the proof of Theorem \ref{thm:main1}
(which uses \cite[Lem.~4.1.1]{KF}) that the condition 
$(\spadesuit)$ is always satisfied with $J = \emptyset$. Thus Theorem 
\ref{thm:main1} is also included in Theorem \ref{thm:main2}. 
\item 
Recall that $(Y, {\mathcal{N}})$ is called hollow if 
any non-invertible local section of ${\mathcal{N}}$ is sent to $0$ by the structure 
morphism ${\mathcal{N}} \to {\mathcal O}_Y$. If 
$(Y, {\mathcal{N}})$ is hollow, 
the condition $(\spadesuit)$ is always satisfied because, when $Y = {\rm Spec}\,R$, 
the chart $Q_Y \to {\mathcal N}$ above gives the isomorphism 
$(Y,{\mathcal N}) \cong {\rm Spec}(R[Q]/(Q\setminus \{0\})R[Q],Q)$. 
So the case (2) of Matsuue's result explained in the introduction 
(\cite[Thm.~7.9]{Ma}) is included in Theorem \ref{thm:main2}. 
\end{enumerate}
\end{rem}

\begin{proof}
The proof is almost the same as that of Theorem \ref{thm:main1}. 
We may assume that $(Y, {\mathcal{N}})$ is strict smooth over 
${\rm Spec}(R[Q]/JR[Q],Q)$ as in the condition $(\spadesuit)$. 
Also, as we see later in the proof, we may assume that $p$ is nilpotent in $R$. 
If we reduce the theorem to the case where the morphism 
$f: (X, {\mathcal M}) \to (Y, {\mathcal N})$ is the 
morphism ${\rm Spec}(R[P]/JR[P],  P) \to {\rm Spec}(R[Q]/JR[Q],  Q)$ 
associated to an injective $p$-saturated homomorphism   
$Q \to P$ of fs monoids such that $Q^{\rm gp}, P^{\rm gp}, 
P^{\rm gp}/Q^{\rm gp}$ are torsion free, the theorem follows from 
Theorem \ref{thm:qisqis1}. In the following, we prove this reduction. 

Let $x$ be a geometric point of $X$ over $y$ and put 
$P = {\mathcal M}_x/{\mathcal O}_{X,  x}^*$. Then, as in the proof of 
Theorem \ref{thm:main1}, we see that the homomorphism 
$Q \to P$ is an injective $p$-saturated morphism of fs monoids 
with $P^{\rm gp}, Q^{\rm gp}, P^{\rm gp}/Q^{\rm gp}$ 
torsion free. 

In the following, we put $\overline{R[Q]} := R[Q]/JR[Q]$, 
 $\overline{R[P]} := R[P]/JR[P]$. 
We will take a nice chart of $f$ by the following diagram: 
\[ \xymatrix{
(U, {\mathcal{M}}|_{U}) \ar[rr] \ar[rd] \ar[rddd]_{f'} \ar[rrrd] & & \mbox{Spec}(\overline{R[P \oplus {\mathbb{N}}^{a+b}]}, P \oplus {\mathbb{N}}^{a+b}) \ar[d] & \\
& \text{fiber product} \ar[r] \ar[dd] & \mbox{Spec}(\overline{R[P \oplus {\mathbb{N}}^{a}]}, P \oplus {\mathbb{N}}^{a}) \ar[r] \ar[d] & \mbox{Spec}(\overline{R[P]}, P) \ar[dd] \\
& & \mbox{Spec}(\overline{R[Q \oplus {\mathbb{N}}^{a}]}, Q \oplus {\mathbb{N}}^{a}) \ar[rd] & \\
& (Y, {\mathcal{N}}) \ar[rr] \ar[ru] & & \mbox{Spec}(\overline{R[Q]},  Q).  
}
\]

First, applying \cite[Lem.~4.1.1]{KF} to $(X,  {\mathcal M})/(Y,  {\mathcal N})$, we obtain a 
an etale neighborhood $U$ of $x$ in $X$, 
a morphism $f': (U, {\mathcal{M}}|_{U}) \to (Y, {\mathcal{N}})$ over $f$ and 
a chart $Q \to P$ of it such that the induced morphism
\begin{equation}\label{eq:vu2}
% (V,{\mathcal{N}}|_{V}) \to \mbox{Spec}(R[Q],  Q), \qquad 
(U,{\mathcal{M}}|_{U}) \to (Y,{\mathcal{N}}) \times_{\mbox{\scriptsize{Spec}}(\overline{R[Q]},Q)} 
\mbox{Spec}(\overline{R[P]},P)
\end{equation} is 
strict smooth. After shrinking $Y$ suitably, we can take a decomposition 
\[ (Y,{\mathcal{N}}) \to \mbox{Spec}(\overline{R[Q \oplus {\mathbb{N}}^a]}, Q \oplus {\mathbb{N}}^a) \to 
\mbox{Spec}(\overline{R[Q]},  Q) \] 
of the strict smooth map $(Y,{\mathcal N}) \to \mbox{Spec}(\overline{R[Q]},  Q)$ in $(\spadesuit)$ 
in which the first map is strict etale, where 
$\overline{R[Q \oplus {\mathbb{N}}^a]} := 
R[Q \oplus {\mathbb{N}}^a]/(J \oplus {\mathbb N}^a)R[Q \oplus {\mathbb{N}}^a]$. 
Since $p$ is nilpotent on $Y$, the first map in the above diagram factors through 
$\mbox{Spec}(\overline{(R/p^nR)[Q \oplus {\mathbb{N}}^a]}, Q \oplus {\mathbb{N}}^a)$ for some $n$ 
(where $\overline{(R/p^nR)[Q \oplus {\mathbb{N}}^a]} := 
(R/p^nR)[Q \oplus {\mathbb{N}}^a]/(J \oplus {\mathbb N}^a)(R/p^nR)[Q \oplus {\mathbb{N}}^a]$) 
and the map 
\[ (Y,{\mathcal{N}}) \to \mbox{Spec}(\overline{(R/p^nR)[Q \oplus {\mathbb{N}}^a]}, Q \oplus {\mathbb{N}}^a)\] 
is strict etale. Thus, we may assume that $R/p^nR = R$, namely, $p$ is nilpotent in $R$. 

Let 
$\mbox{Spec}(\overline{R[P \oplus {\mathbb{N}}^{a}]}, P \oplus {\mathbb{N}}^{a}) \to \mbox{Spec}(
\overline{R[Q \oplus {\mathbb{N}}^{a}]}, Q \oplus {\mathbb{N}}^{a})$ be the morphism we obtain 
by pulling back the map 
${\rm Spec}(\overline{R[P]},  P) \to {\rm Spec}(\overline{R[Q]},  Q)$ to 
$\mbox{Spec}(\overline{R[Q \oplus {\mathbb{N}}^a]}, Q \oplus {\mathbb{N}}^a)$. 
(Hence $\overline{R[P \oplus {\mathbb{N}}^a]} := 
R[P \oplus {\mathbb{N}}^a]/(J \oplus {\mathbb N}^a)R[P \oplus {\mathbb{N}}^a]$.) 
Then the morphism $(U,{\mathcal{M}}|_{U}) \to \mbox{Spec}(\overline{R[P]}, P)$ induce a morphism 
\begin{equation}\label{eq:u2}
(U,{\mathcal{M}}|_{U}) \to \mbox{Spec}(\overline{R[P \oplus {\mathbb{N}}^{a}]}, P \oplus {\mathbb{N}}^{a}), 
\end{equation}
which factors through the fiber product of $(Y,{\mathcal N})$ and 
$\mbox{Spec}(\overline{R[P \oplus {\mathbb{N}}^{a}]}, P \oplus {\mathbb{N}}^{a})$ over 
$\mbox{Spec}(\overline{R[Q \oplus {\mathbb{N}}^{a}]}, Q \oplus {\mathbb{N}}^{a})$. 
Since the morphism from this fiber product to 
$\mbox{Spec}(\overline{R[P \oplus \mathbb{N}^{a}]}, P \oplus \mathbb{N}^{a})$ is strict etale 
and the morphism from $(U,{\mathcal M}|_U)$ to this fiber product is strict smooth, 
we conclude that the map \eqref{eq:u2} is 
strict smooth. Thus, after shrinking $U$ suitably, we have a decomposition 
\[ (U,{\mathcal M}|_U) \to \mbox{Spec}(\overline{R[P \oplus {\mathbb{N}}^{a+b}]}, P \oplus {\mathbb{N}}^{a+b}) \to 
\mbox{Spec}(\overline{R[P \oplus {\mathbb{N}}^a]}, P \oplus {\mathbb{N}}^a)\] 
of the morphism \eqref{eq:u2} in which the first map is strict etale, 
where $\overline{R[P \oplus {\mathbb{N}}^{a+b}]} := 
R[P \oplus {\mathbb{N}}^{a+b}]/(J \oplus {\mathbb N}^{a})R[P \oplus {\mathbb{N}}^{a+b}]$.

By the argument above, we obtain the following commutative diagram in which 
the horizontal arrows are strict etale: %and the vertical arrows are log smooth: 
\[ \xymatrix{
(U, {\mathcal{M}}|_{U}) \ar[r] \ar[d] & \mbox{Spec}(\overline{R[P \oplus {\mathbb{N}}^{a+b}]}, P \oplus {\mathbb{N}}^{a+b}) \ar[d] \\
(Y, {\mathcal{N}}) \ar[r] & \mbox{Spec}(\overline{R[Q \oplus {\mathbb{N}}^{a}]}, Q \oplus {\mathbb{N}}^{a}). 
}
\]
Noting that the statement of the theorem is etale local on $X$ and $Y$, we can 
reduce the proof of the theorem to the case where $f$ is the morphism 
$\mbox{Spec}(\overline{R[P \oplus {\mathbb{N}}^{a+b}]}, P \oplus {\mathbb{N}}^{a+b})
\to \mbox{Spec}(R\overline{[Q \oplus {\mathbb{N}}^{a}]}, Q \oplus {\mathbb{N}}^{a}).$
Since $Q \to P$ is an injective $p$-saturated morphism of fs monoids 
such that $Q^{\rm gp}, P^{\rm gp}, P^{\rm gp}/Q^{\rm gp}$ are torsion free, 
the same properties are satisfied for the homomorphism 
$Q \oplus {\mathbb{N}}^a \to P \oplus {\mathbb{N}}^{a+b}$. 
Also, since $J$ is a radical ideal of $Q$, $J \oplus {\mathbb N}^a$ is also a radical ideal of 
$Q \oplus {\mathbb N}^a$. 
Hence we obtained the required reduction and so we are done. 
\end{proof}

\appendix 

\setcounter{theorem}{0}
\renewcommand{\thetheorem}{\Alph{section}.\arabic{theorem}}

\section{Compatibility with the construction of Hyodo-Kato}

In this appendix, let $Y = {\rm Spec}\,k$ with $k$ a perfect field of characteristic $p>0$ 
and let $f: (X,{\mathcal M}) \to (Y,{\mathcal N})$ be a log smooth $p$-saturated 
morphism between fine log schemes. 
(By \cite[Prop. II 2.14]{Tsuji}, it is equvalent to say that $f$ is a 
a log smooth morphism of Cartier type \cite[2.12]{HK} between fine log schemes.) 
Let 
$$ u_m: ((X, {\mathcal M})/W_m(Y,{\mathcal N}))^{\rm log}_{\rm crys} \to X_{\rm et} $$
be the canonical morphism of topoi introduced in Section 1.3 and 
let ${\mathcal O}_m$ be the structure crystal 
in $((X, {\mathcal M})/W_m(Y,{\mathcal N}))^{\rm log}_{\rm crys}$. 
Hyodo-Kato \cite[\S 4]{HK} defined the sheaf $W_m\omega^{q}_{(X,{\mathcal M})/(Y,{\mathcal N})}$ 
(denoted simply as $W_m\omega^{q}_{X}$ there) by 
$$ W_m\omega^{q}_{(X,{\mathcal M})/(Y,{\mathcal N})} := {\mathbb R}^qu_{m *}{\mathcal O}_m, $$
and defined the maps 
\begin{align*}
& d: W_m\omega^{q}_{(X,{\mathcal M})/(Y,{\mathcal N})} \to W_m\omega^{q+1}_{(X,{\mathcal M})/(Y,{\mathcal N})}, 
\quad 
F: W_{m+1}\omega^{q}_{(X,{\mathcal M})/(Y,{\mathcal N})} \to W_m\omega^{q}_{(X,{\mathcal M})/(Y,{\mathcal N})}, \\ 
& V: W_m\omega^{q}_{(X,{\mathcal M})/(Y,{\mathcal N})} \to W_{m+1}\omega^{q}_{(X,{\mathcal M})/(Y,{\mathcal N})}, 
\quad \pi: W_{m+1}\omega^{q}_{(X,{\mathcal M})/(Y,{\mathcal N})} \to W_m\omega^{q}_{(X,{\mathcal M})/(Y,{\mathcal N})}
\end{align*}
such that 
\begin{align}
& d^2 = 0, \quad FV = VF = p, \quad dF = pFd, \quad Vd = pdV, \quad FdV = d, \label{eqa:dd} \\ 
& d \pi = \pi d, \quad F \pi = \pi F, \quad V \pi = \pi V. \nonumber  
\end{align}
(See \cite[Cor.~6.28]{N} for a detailed proof of the last three equalities.)
%Also, $\pi$ is a map of graded algebras 
%$W_{m+1}\omega^{\bullet}_{(X,{\mathcal M})/(Y,{\mathcal N})} \to 
%W_m\omega^{\bullet}_{(X,{\mathcal M})/(Y,{\mathcal N})}$, although it is not written explicitly 
%in \cite{HK} or in \cite{N}: This follows from the fact that each arrow in the local expression of the map $\pi$ 
%given in \cite[(6.4.5)]{N} is compatible with the product structure. 
%\cite[p.247]{HK}. 
Hence we see that 
$$ \{ W_m\omega^{\bullet}_{(X,{\mathcal M})/(Y,{\mathcal N})} := 
(W_m\omega^{\bullet}_{(X,{\mathcal M})/(Y,{\mathcal N})}, d), \pi: W_{m+1}\omega^{\bullet}_{(X,{\mathcal M})/(Y,{\mathcal N})}
\to W_m\omega^{\bullet}_{(X,{\mathcal M})/(Y,{\mathcal N})}\}_m $$ 
forms a projective system of complexes. This is the definition of 
the log de Rham-Witt complex in \cite{HK}. 

Then Hyodo-Kato \cite[Thm.~4.19]{HK} defined the comparison morphism 
\begin{equation}\label{eqa:comp}
{\mathbb R}u_{m *}{\mathcal O}_m \to  
 W_m\omega^{\bullet}_{(X,{\mathcal M})/(Y,{\mathcal N})} 
\end{equation}
and proved that it is a quasi-isomorphism. 

In this appendix, we prove the compatibility of the log de Rham-Witt complex and the 
comparison morphism in \cite{HK} with those in our article. The main result is as follows: 

\begin{theorem}\label{thma:main}
Let the notations be as above. 
\begin{enumerate}[(1)] 
\item There exists a canonical isomorphism 
$$ \{W_m\Lambda^{\bullet}_{(X,{\mathcal M})/(Y,{\mathcal N})}\}_m \xrightarrow{\cong}
\{W_m\omega^{\bullet}_{(X,{\mathcal M})/(Y,{\mathcal N})}\}_m$$ 
of projective systems of differential graded algebras which is compatible with the maps $F, V$. 
\item Via the isomorphism in $(1)$, the comparison morphism \eqref{eqa:comp} is the same as
the comparson morphism 
\begin{equation}\label{eqa:comp0}
 {\mathbb R}u_{m *}{\mathcal O}_m \to  
 W_m\Lambda^{\bullet}_{(X,{\mathcal M})/(Y,{\mathcal N})}  
\end{equation}
defined in Section 1.3. 
\end{enumerate}
\end{theorem}

\begin{cor}
In the above situation, the comparison morphism \eqref{eqa:comp0}
% \[ \mathbb{R}u_{m*} \mathcal{O}_m \to W_{m} \Lambda_{(X, \mathcal{M})/(Y,  \mathcal{N})}^{\bullet} \]
% we constucted in Section 1.3 
is a quasi-isomorphism. 
\end{cor}

\begin{rem}
The above corollary is not contained in Theorem \ref{thm:main2} because 
the log structures ${\mathcal M}, {\mathcal N}$ here are not necessarily fs. 
On the other hand, when the log structures ${\mathcal M}, {\mathcal N}$ are fs, 
the above corollary is a special case of Theorem \ref{thm:main2} by 
Proposition \ref{prop:sat}, Remark \ref{rem:remrem}(3) and the fact that 
the log structure ${\mathcal N}$ on ${\rm Spec}\,k$ is necessarily hollow. 
Note that, when we are in the situation of this appendix with ${\mathcal M}, {\mathcal N}$ fs, 
our proof of comparison isomorphism given in Theorem \ref{thm:main2} is 
different from the proof by Hyodo-Kato.
\end{rem}

We give some preliminaries on the complex 
$W_m\Lambda^{\bullet}_{(X,{\mathcal M})/(Y,{\mathcal N})}$ which we need in the proof of 
Theorem \ref{thma:main}. 
First we explain the definition of $d, F, V$ in the case where there exists a log Frobenius lift 
$((Z_m, {\mathcal L}_m), \Phi_m, \Delta_m)_m$ of the morphism $f: (X,{\mathcal M}) \to (Y,{\mathcal N})$. 
(This assumption holds etale locally on $X$.) In this case, we have 
the quasi-isomorphism 
$$ {\mathbb R}u_{m *}{\mathcal O}_m \xrightarrow{\cong} 
{\mathcal O}_{{\bar Z}_m} \otimes_{{\mathcal O}_{Z_m}} 
\Lambda^{\bullet}_{(Z_m,{\mathcal L}_m)/W_m(Y, {\mathcal N})}
\xleftarrow{=} \breve{\Lambda}^{\bullet}_{({\bar Z}_m,{\bar {\mathcal L}}_m)/W_m(Y, {\mathcal N})} =: C_m^{\bullet} 
$$ 
(where $({\bar Z}_m,{\bar {\mathcal L}}_m)$ is the log pd-envelope of 
the closed immersion $(X,{\mathcal M}) \to (Z_m,{\mathcal L}_m)$) and so 
$$ W_m\omega^{q}_{(X,{\mathcal M})/(Y,{\mathcal N})} = H^q(C_m^{\bullet}). $$
Then the map $d: H^q(C_m^{\bullet}) \to H^{q+1}(C_m^{\bullet})$ 
is defined as the connecting homomorphism induced by the exact sequence 
$$ 0 \to C_m^{\bullet} \xrightarrow{p} C_{2m}^{\bullet} \xrightarrow{\rm proj.} C_m^{\bullet} \to 0, $$
the map $F: H^q(C_{m+1}^{\bullet}) \to H^q(C_m^{\bullet})$ 
is defined as the one induced by the projection $C_{m+1}^{\bullet}  \to C_m^{\bullet} $ 
and the map $V: H^q(C_m^{\bullet}) \to H^q(C_{m+1}^{\bullet})$ 
is defined as the one induced by the map 
$p:C_m^{\bullet}  \to C_{m+1}^{\bullet}$ of multiplication by $p$.

Let $\tau: W_m({\mathcal O}_X) \to H^0(C_m^{\bullet})$ be the 
ring homomorphism $(a_0, \dots, a_{m-1}) \mapsto 
\sum_{i=0}^{m-1}p^i \tilde{a}_i^{p^{m-i}}$, where 
$\tilde{a}_i$ is any lift of $a_i$ in ${\mathcal O}_{Z_m}$. 
(The well-definedness of this map is proven in \cite[p.251]{HK}.) 
Also, let 
$d\log: W_m({\mathcal M}) = {\mathcal M} \oplus {\rm Ker}(W_m({\mathcal O}_X)^* \to {\mathcal O}_X^*)
\to H^1(C_m^{\bullet})$ be the map defined by 
$b \mapsto d\log \tilde{b}$ for $b \in {\mathcal M}$ and 
$b \mapsto \tau(b)^{-1}d(\tau(b))$ for $b \in {\rm Ker}(W_m({\mathcal O}_X)^* \to {\mathcal O}_X^*)$, 
where $\tilde{b}$ is any lift of $b \in {\mathcal M}$ in ${\mathcal L}_m$. 
(The well-definedness is again proven in \cite[p.251]{HK}.) 
Then we can check the following: 

\begin{lemma}\label{lema:4}
Let the notations be as above. 
\begin{enumerate}[(1)]
\item 
Via $\tau$, 
$W_m\omega^{\bullet}_{(X,{\mathcal M})/(Y,{\mathcal N})} = H^{\bullet}(C_m^{\bullet})$ is 
a differential graded $W_m({\mathcal O}_X)/W_m({\mathcal O}_Y)$-algebra. 
\item 
Via $\tau$ and $d\log$, 
$W_m\omega^{\bullet}_{(X,{\mathcal M})/(Y,{\mathcal N})} = H^{\bullet}(C_m^{\bullet})$ is 
a log differential graded algebra on $W_m(X,{\mathcal M})/W_m(Y,{\mathcal N})$. 
\item 
The structure of log differential graded algebras on 
$W_m\omega^{\bullet}_{(X,{\mathcal M})/(Y,{\mathcal N})} = H^{\bullet}(C_m^{\bullet}) \, (m \in 
{\mathbb N})$ is compatible with respect to $m$. 
\item 
The map $\tau$ is compatible with $F,  V$. 
\item 
The maps $F,  V$ are compatible with the map
$\pi: W_{m+1}\omega^{\bullet}_{(X,{\mathcal M})/(Y,{\mathcal N})} \to 
W_{m}\omega^{\bullet}_{(X,{\mathcal M})/(Y,{\mathcal N})}.$
\item 
The map $F: W_{m+1}\omega^{\bullet}_{(X,{\mathcal M})/(Y,{\mathcal N})} \to 
W_m\omega^{\bullet}_{(X,{\mathcal M})/(Y,{\mathcal N})}$ induces a morphism 
of graded $W_{m+1}({\mathcal O}_X)$-algebras 
$W_m\omega^{\bullet}_{(X,{\mathcal M})/(Y,{\mathcal N})} \to 
W_m\omega^{\bullet}_{(X,{\mathcal M})/(Y,{\mathcal N}), [F]}$, where 
the symbol $[F]$ is as in Definition \ref{def:rellogdrw}(4)(iv) (with $S$ replaced by ${\mathcal O}_X$). 
\item 
The following equalities hold. 
\begin{eqnarray*}
     {}^{FV}\omega  & = & p\omega \quad (\omega \in W_m\omega^{\bullet}_{(X,{\mathcal M})/(Y,{\mathcal N})}),   \\
   {}^{F}d^{V}\omega & = &d \omega \quad (\omega \in W_m\omega^{\bullet}_{(X,{\mathcal M})/(Y,{\mathcal N})}), \\
  {}^{F}d [x] & = & [x^{p-1}] d [x]  \quad (x \in {\mathcal O}_X), \\
  {}^{V}(\omega {}^{F}\eta) & = & {}^{V}(\omega) \eta \quad ( \eta \in W_{m+1}\omega^{\bullet}_{(X,{\mathcal M})/(Y,{\mathcal N})}), \\  
  {}^{F}(d\log q) & = &  d\log q \quad( q \in {\mathcal M} \subseteq W_{m+1}({\mathcal M})).
\end{eqnarray*}
\end{enumerate}
\end{lemma}

\begin{rem}
The last equality in (7) does not always hold when $q \in W_{m+1}({\mathcal M})$. 
\end{rem}

\begin{proof}
(1) \, It is standard that the map $\tau$ is a homomorphism of algebras. 
So it suffices to check the three equalities in Definition \ref{def:1.1.1}(2). 
The first equality is easy the third equality is already recalled in \eqref{eqa:dd}. 
So we check the second equality. 
Assume that $\omega$ is represented by a cocycle $\alpha \in C_m^i$ and 
$\eta$ is represented by a cocycle $\beta \in C_m^j$. Then 
$d(\omega\eta)$ is represented by $\gamma \in C_m^{i+j+1}$ with 
$p^m\gamma = d(\tilde{\alpha}\tilde{\beta})$, where 
$\tilde{\alpha} \in C_{2m}^i, \tilde{\beta} \in C_{2m}^j$ are lifts of 
$\alpha, \beta$, respectively. On the other hand, if we take 
$\delta \in C_m^{i+1}, \epsilon \in C_m^{j+1}$ so that 
$p^m\delta = d(\tilde{\alpha}), p^m\epsilon = d(\tilde{\beta})$, 
$(d\omega)\eta + (-1)^i \omega d\eta$ is represented by 
$\delta\beta + (-1)^i\alpha\epsilon$. Since we have 
\[ 
p^m (\delta\beta + (-1)^i\alpha\epsilon) 
= (d \tilde{\alpha}) \tilde{\beta} + (-1)^i \tilde{\alpha} (d \tilde{\beta}) 
= d(\tilde{\alpha}\tilde{\beta}), \] 
we obtain the equality $\gamma = \delta\beta + (-1)^i\alpha\epsilon$, hence 
the required equality. 

(2) \, We denote the map ${\mathcal M} \to {\mathcal O}_X$ by $\alpha$ and the map 
$W_m({\mathcal M}) \to W_m({\mathcal O}_X)$ by $\alpha_m$. 
Also, we denote the map ${\mathcal L} \to {\mathcal O}_{Z_m}$ by $\tilde{\alpha}$.
First we check that the pair $(d \circ \tau, d\log)$ is a log derivation. 
The condition in Definition \ref{def:rellogdrw}(1)(i) is written as 
% nothing but the equality 
\begin{equation}\label{eqa:1}
d \circ \tau(\alpha_m(b)) = \tau(\alpha_m(b)) d\log (b)  
\end{equation}
in our notation here, 
and it suffices to check it for $b \in {\mathcal M}$ and for $b \in {\rm Ker}(
W_m({\mathcal O}_X)^* \to {\mathcal O}_X^*)$. 
For $b \in {\mathcal M}$, we have 
\begin{align*}
& \text{(LHS of \eqref{eqa:1})} = 
d \circ \tau([\alpha(b)]) = d(\tilde{\alpha}(\tilde{b})^{p^m})
= \tilde{\alpha}(\tilde{b})^{p^m-1}d \tilde{\alpha}(\tilde{b}), \\ 
& \text{(RHS of \eqref{eqa:1})} = 
\tilde{\alpha}(\tilde{b})^{p^m}  d\log (\tilde{b}) 
= \tilde{\alpha}(\tilde{b})^{p^m-1}d \tilde{\alpha}(\tilde{b}),
\end{align*}
as required. For $b \in {\rm Ker}(
W_m({\mathcal O}_X)^* \to {\mathcal O}_X^*)$, it is immediate from the 
definition of $d\log$. Also, the condition \ref{def:rellogdrw}(1)(ii) is immediate 
and so $(d \circ \tau, d\log)$ is a log derivation. 

So it suffices to check the equality $d \circ d\log (b) = 0$ to prove the assertion (2). 
It is immediate for $b \in {\mathcal M}$, and for $b \in 
{\rm Ker}(
W_m({\mathcal O}_X)^* \to {\mathcal O}_X^*)$ 
it follows from the calculation 
\[ d(\tau(b)^{-1}d(\tau(b))) = d(\tau(b)^{-1}) d(\tau(b)) = 
\tau(b)^{-2} d(\tau(b)) d(\tau(b)) = 0, \]
where the first and the second equalities hold by (1) and 
the third equality holds because it holds in $C_m^{\bullet}$.  
So the proof of the assertion (2) is finished. 

(3) \, First recall that $\pi$ is a map of graded algebras 
$W_{m+1}\omega^{\bullet}_{(X,{\mathcal M})/(Y,{\mathcal N})} \to 
W_m\omega^{\bullet}_{(X,{\mathcal M})/(Y,{\mathcal N})}$, although it is not remarked explicitly 
in \cite{HK} or in \cite{N}: This follows from the fact that each arrow in the local expression of the map $\pi$ 
given in \cite[(6.4.5)]{N} is compatible with the product structure. 
% \cite[p.247]{HK}. 
We already mentioned the compatibility of $\pi$ and $d$ in \eqref{eqa:dd}. 
So it suffices to check the compatibility of $\pi$ and the maps $\tau, d\log$. 
This follows from the commutativity of the diagram \cite[(7.0.6)]{N}. 

(4) \, Both $F \circ \tau$ and $\tau \circ F$ send 
$(a_0, \dots, a_m) \in W_{m+1}({\mathcal O}_X)$ to 
$\sum_{i=0}^{m-1}p^i \tilde{a}_i^{p^{m+1-i}}$. 
Also, both $V \circ \tau$ and $\tau \circ V$ send 
$(a_0, \dots, a_{m-1}) \in W_{m}({\mathcal O}_X)$ to
$\sum_{i=0}^{m-1} p^{i+1}\tilde{a}_i^{p^{m-i}}$. 

(5) \, It is already mentioned in \eqref{eqa:dd}. 

(6) \, It is immediate from the definition that $F$ is a morphism of 
graded algebras. The compatibility of $F$ with the map $F: W_{m+1}({\mathcal O}_X) 
\to W_m({\mathcal O}_X)$ follows from (4). 

(7) \, The first equality is immediate from the definition of $F$ and $V$. 
If $\omega$ is represented by a cocycle $\alpha \in C_m^{\bullet}$ and 
if we take a lift $\tilde{\alpha} \in C_{2m}^{\bullet}$ of $\alpha$, 
$d^{V}\omega$ is represented by $\beta \in C_{m+1}^{\bullet}$ with 
$p^{m+1}\beta = pd\tilde{\alpha}$ in $C_{2m+2}^{\bullet}$. 
Then the image of $\beta$ in $C_{m}^{\bullet}$ (denoted again by $\beta$), which represents 
${}^Fd^{V}\omega$, satisfies the equality 
$p^m \beta = d\tilde{\alpha}$ in $C_{2m}^{\bullet}$ and so 
we see the second equality ${}^Fd^{V}\omega = d \omega$. 
The third equality ${}^{F}d [x] = [x^{p-1}] d [x]$ follows from the 
computation that both hand sides are represented by 
$\tilde{x}^{p^{m+1}-1}d\tilde{x} \in C_m^{\bullet}$. 
The fourth equality immediately follows from the definition of $F, V$, and 
the fifth equality immediately follows from the definition of $F, d\log$. 
\end{proof}

\begin{comment}
Then we can check that, via $\tau$, 
$W_m\omega^{q}_{(X,{\mathcal M})/(Y,{\mathcal N})} = H^q(C_m^{\bullet})$ is 
a differential graded $W_m({\mathcal O}_X)/W_m({\mathcal O}_Y)$-algebra. 
Next, let 
$d\log: W_m({\mathcal M}) \to H^1(C_m^{\bullet})$ be the map  
$b \mapsto d\log \tilde{b}$, where $\tilde{b}$ is any lift of $b$ in ${\mathcal L}$. 
(The well-definedness is again proven in \cite[p.251]{HK}.) 
Noting that $d \circ \tau: W_m({\mathcal O}_X) \to H^0(C_m^{\bullet})$
is given by $(a_0, \dots, a_{m-1}) \mapsto \sum_{i=0}^{m-1}\tilde{a}_i^{p^{m-i}-1}d\tilde{a}_i$, 
we can check that the pair $(d \circ \tau, d\log)$ is a log derivation 
on $W_m(X,{\mathcal M})/W_m(Y,{\mathcal N})$. We
can check also the equality $d \circ d\log = 0$ and so 
$W_m\omega^{\bullet}_{(X,{\mathcal M})/(Y,{\mathcal N})} = H^{\bullet}(C_m^{\bullet})$ is 
a log differential graded algebra on $W_m(X,{\mathcal M})/W_m(Y,{\mathcal N})$. 
In particular, there exists a canonical morphism 
\begin{equation}\label{eqa:sur}
\Lambda^{\bullet}_{W_m(X,{\mathcal M})/W_m(Y,{\mathcal N})}
\to W_m\omega^{\bullet}_{(X,{\mathcal M})/(Y,{\mathcal N})}. 
\end{equation}
\end{comment}

Now assume further that the morphism $f: (X,{\mathcal M}) \to (Y,{\mathcal N})$ is of the form 
${\rm Spec}\,(S,Q) \to {\rm Spec}\,(k,P)$ induced by a map of pre-log rings 
$(k,P) \to (S,Q)$ (this condition holds etale locally on $X$ and on $Y$), 
and put $W_m \omega^q_{(S,Q)/(k,P)} := \Gamma(X, W_m\omega^{q}_{(X,{\mathcal M})/(Y,{\mathcal N})}) = 
\Gamma(X, H^q(C_m^{\bullet}))$. 
Then, by Lemma \ref{lema:4}, 
% what we have shown in the previous paragraph, 
$\{W_m \omega^{\bullet}_{(S,Q)/(k,P)}\}_m$ forms 
a log $F$-$V$-procomplex over the $(k,P)$-algebra $(S,Q)$. 
%a projective system of 
%log differential graded $W_m(S,Q)/W_m(k,P)$-algebras equipped with maps 
%\[ F: W_{m+1} \omega^{\bullet}_{(S,Q)/(k,P)} \to W_m \omega^{\bullet}_{(S,Q)/(k,P)}, 
%\quad V: W_m \omega^{\bullet}_{(S,Q)/(k,P)} \to W_{m+1} \omega^{\bullet}_{(S,Q)/(k,P)}. \] 
Therefore, we have the canomical morphism of 
log $F$-$V$-procomplexes 
\[ 
\{W_m \Lambda^{\bullet}_{(S,Q)/(k,P)}\}_m \to W_m \omega^{\bullet}_{(S,Q)/(k,P)}\}_m \] 
over the $(k,P)$-algebra $(S,Q)$. Sheafifying this map, we obtain the 
morphism 
\begin{equation}\label{eqa:map}
\{\Phi_m\}_m : \{W_m\Lambda^{\bullet}_{(X,{\mathcal M})/(Y,{\mathcal N})}\}_m \to 
\{W_m\omega^{\bullet}_{(X,{\mathcal M})/(Y,{\mathcal N})}\}_m
\end{equation}
of projective systems of log differential graded algebras 
on $W_m(X,{\mathcal M})/W_m(Y,{\mathcal N})$ compatible with the maps $F, V$. 
This is the morphism in the statement of Theorem \ref{thma:main}. 

\begin{proof}[Proof of Theorem \ref{thma:main}]
First we prove (1). It suffices to prove that the map $\Phi_m$ defined above 
is an isomorphism. 
Since both $W_m\Lambda^{\bullet}_{(X,{\mathcal M})/(Y,{\mathcal N})}$ and  
$W_m\omega^{\bullet}_{(X,{\mathcal M})/(Y,{\mathcal N})}$ are  
log differential graded algebras on $W_m(X,{\mathcal M})/W_m(Y,{\mathcal N})$, 
we have the canonical morphisms 
\[ 
\psi_{\Lambda}: \Lambda^{\bullet}_{W_m(X,{\mathcal M})/W_m(Y,{\mathcal N})} 
\to W_m\Lambda^{\bullet}_{(X,{\mathcal M})/(Y,{\mathcal N})}, 
\quad 
\psi_{\omega}: \Lambda^{\bullet}_{W_m(X,{\mathcal M})/W_m(Y,{\mathcal N})} 
\to W_m\omega^{\bullet}_{(X,{\mathcal M})/(Y,{\mathcal N})} 
\] 
of log differential graded algebras on $W_m(X,{\mathcal M})/W_m(Y,{\mathcal N})$
such that the diagram 
\begin{equation}\label{eqa:tri}
\xymatrix{
& \Lambda^{\bullet}_{W_m(X,{\mathcal M})/W_m(Y,{\mathcal N})} 
\ar[dl]_{\psi_{\Lambda}} \ar[dr]^{\psi_{\omega}} \\ 
 W_m\Lambda^{\bullet}_{(X,{\mathcal M})/(Y,{\mathcal N})} 
\ar[rr]^{\Phi_m} & & 
 W_m\omega^{\bullet}_{(X,{\mathcal M})/(Y,{\mathcal N})} 
}
\end{equation}
is commutative, by the universal property of log differential module 
$\Lambda^1_{W_m(X,{\mathcal M})/W_m(Y,{\mathcal N})}$.  
By the construction of $W_m\Lambda^{\bullet}_{(X,{\mathcal M})/(Y,{\mathcal N})}$ 
given in \cite[3.4]{Ma}, $\psi_{\Lambda}$ is surjective. On the other hand, 
by \cite[Prop.~4.7]{HK}, $\psi_{\omega}$ is surjective with 
${\rm Ker}\,\psi_{\omega}$ the differential graded ideal locally generated by 
elements of the form 
\begin{equation}\label{eqa:elements}
 {}^{V^i}[a] d^{V^j}[\alpha(b)] - 
{}^{V^i}[a \alpha (b)^{p^{i-j}}] d\log [b] 
\quad (0 \leq j \leq i < m, a \in {\mathcal O}_X, b \in {\mathcal M}), 
\end{equation}
where $\alpha$ denotes the structure morphism ${\mathcal M} \to {\mathcal O}_X$. 
To prove that $\Phi_m$ is an isomorphism, it suffices to check that 
${\rm Ker}\,\psi_{\omega} \subseteq {\rm Ker}\,\psi_{\Lambda}$, namely, 
the elements \eqref{eqa:elements} are zero in 
$W_m\Lambda^{\bullet}_{(X,{\mathcal M})/(Y,{\mathcal N})}$. 
This is true because the equalities 
\begin{align*}
{}^{V^i}[a \alpha (b)^{p^{i-j}}] d\log [b] 
& = 
{}^{V^i}([a \alpha (b)^{p^{i-j}}] d\log [b])  
= 
{}^{V^i}([a] {}^{F^{i-j}}d[\alpha(b)]) 
= 
{}^{V^i}([a] {}^{F^{i}}\!d^{V^j}[\alpha(b)])
= 
{}^{V^i} [a] d^{V^j}[\alpha(b)]
\end{align*}
hold in $W_m\Lambda^{\bullet}_{(X,{\mathcal M})/(Y,{\mathcal N})}$. 
So the proof of (1) is finished. 

Next we prove (2). In \cite[pp.~260--261]{HK}, the comparison morphism 
\eqref{eqa:comp} is defined as the composite of the morphism 
\[ {\mathbb R}u_{m *}{\mathcal O}_m \to 
\breve{\Lambda}^{\bullet}_{(X,{\mathcal M})/(Y,{\mathcal N})} \] 
defined in \eqref{eq:comp} and the map 
\[ \breve{\Lambda}^{\bullet}_{(X,{\mathcal M})/(Y,{\mathcal N})} \to 
W_m\omega^{\bullet}_{(X,{\mathcal M})/(Y,{\mathcal N})} \] 
induced by $\psi_{\omega}$. From this definition and the construction 
of the comparison morphism 
\[ \breve{\Lambda}^{\bullet}_{(X,{\mathcal M})/(Y,{\mathcal N})} \to 
W_m\Lambda^{\bullet}_{(X,{\mathcal M})/(Y,{\mathcal N})} \] 
in \eqref{eq:comp}, 
the required compatibility follows from the commutativity of 
the diagram \eqref{eqa:tri}. So the proof of (2) is also finished. 
\end{proof}

\section*{Acknowledgements}
This article is a revised and slightly improved version of the master thesis of the first author under the 
supervision of the second author, which was 
originally written in Japanese. The authors would like to thank Professor Takeshi Tsuji for 
kindly sending us the preprint \cite{Tsuji}. The second author is partly supported by JSPS 
KAKENHI (Grant Numbers 17K05162 and 15H02048).

\end{document}